\documentclass[11 pt]{article}
\usepackage{amssymb, amsfonts, amsthm, amsbsy, latexsym, color,bbm,url}
\usepackage{graphicx}
\usepackage{epstopdf}
\usepackage{fullpage}

\newcommand{\bitem}{\begin{itemize}}
\newcommand{\eitem}{\end{itemize}}
\newcommand{\beq}{\begin{equation}}
\newcommand{\eeq}{\end{equation}}
\newcommand{\goto}{\rightarrow}
\newcommand{\cC}{{\mathcal C}}

\newcommand{\cW}{{\mathcal W}}
\newcommand{\bR}{{\bf R}}
\newcommand{\RR}{{\mathbf R}}
\newcommand{\bZ}{{\bf Z}}
\newcommand{\bN}{{\bf N}}

\newcommand{\cT}{{\mathcal T}}

\newcommand{\cH}{{\mathcal H}}
\newcommand{\cM}{{\mathcal M}}
\newcommand{\eps}{{\epsilon}}
\newcommand{\argmin}{\mbox{argmin}}
\newcommand{\supp}{\mbox{supp}\;}
\newcommand{\sinc}{{\;\rm sinc}}

\newcommand{\cL}{{\mathcal L}}
\newcommand{\cF}{{\mathcal F}}

\newcommand{\ip}[2]{\langle#1,#2\rangle}
\newcommand{\ipa}[1]{\langle#1\rangle}
\newcommand{\ang}[1]{\langle#1\rangle}
\newcommand{\abs}[1]{\left| #1 \right|}
\newcommand{\absip}[2]{| \langle#1,#2\rangle |}
\newcommand{\norm}[1]{\|#1\|}

\theoremstyle{plain}
\newtheorem{theorem}{Theorem}[section]

\newtheorem{proposition}[theorem]{Proposition}

\theoremstyle{definition}
\newtheorem{definition}[theorem]{Definition}

\newtheorem{lemma}[theorem]{Lemma}

\begin{document}

\title{Analysis of Inpainting via Clustered Sparsity and Microlocal Analysis}

\author{Emily J. King, Gitta Kutyniok, and Xiaosheng Zhuang}

\maketitle

\begin{abstract}
Recently, compressed sensing techniques in combination with both wavelet and directional representation
systems have been very effectively applied to the problem of image inpainting. However, a mathematical
analysis of these techniques which reveals the underlying geometrical content is completely missing. In this paper,
we provide the first comprehensive analysis in the continuum domain utilizing the novel concept of
clustered sparsity, which besides leading to a\-symptotic error bounds also makes the superior behavior
of directional representation systems over wavelets precise. First, we propose an abstract model for
problems of data recovery and derive error bounds for two different recovery schemes, namely $\ell_1$
minimization and thresholding. Second, we set up a particular microlocal model for an image governed by edges
inspired by seismic data as well as a particular mask to model the missing data, namely a linear singularity masked by a horizontal strip. Applying the abstract estimate in the
case of wavelets and of shearlets we prove that -- provided the size of the missing part is asymptotically
to the size of the analyzing functions -- asymptotically precise inpainting can be obtained for this model. Finally, we
show that shearlets can fill strictly larger gaps than wavelets in this model. 
\end{abstract}

\\
{\bf Acknowledgements.} Emily J. King is supported by a fellowship for postdoctoral researchers from the Alexander
von Humboldt Foundation. Gitta Kutyniok would like to thank David Donoho for discussions on this and related topics.
She is grateful to the Department of Statistics at Stanford University and the Department
of Mathematics at Yale University for their hospitality and support during her visits.
She also acknowledges support by the Einstein Foundation Berlin, by Deutsche Forschungsgemeinschaft
(DFG) Heisenberg fellowship KU 1446/8, Grant SPP-1324 KU 1446/13 and DFG Grant KU 1446/14, and by the
DFG Research Center {\sc Matheon} ``Mathematics for key technologies'' in Berlin. Xiaosheng Zhuang
acknowledges support by DFG Grant KU 1446/14.  Finally, the authors are thankful to the anonymous referees for their comments and suggestions.
\vspace{.1in}

\section{Introduction}

A common problem in many  fields of scientific research is that of missing data.  The human visual system has an amazing ability to fill in the missing parts
of images, but automating this process is not trivial.  Also, depending on the type of data, the human senses may be
unable to fill in the gaps.  Conservators working to repair damaged paintings use the term \emph{inpainting} to describe
the process.  This word now also means digitally recovering missing data in videos and images.  The removal of overlaid
text in images, the repair of scratched photos and audio recordings, and the recovery of missing blocks in a streamed
video are all examples of inpainting. Seismic data are also commonly incomplete due to land development and bodies of
water preventing optimal sensor placement \cite{seis1,seis2}.  In seismic processing flow, data recovery plays an important role.

One very common approach to inpainting is using variational methods \cite{BBCSV01,BBS01,BSCB00,ChSh01}. However,
recently the novel meth\-odology of compressed sensing, namely exact recovery of sparse or sparsified data
from highly incomplete linear non-adaptive measurements by $\ell_1$ minimization or thresholding, has been
very effectively applied to this problem. The pioneering paper is \cite{ESQD05}, which uses curvelets as
sparsifying system for inpainting. Various intriguing successive empirical results have since then been
obtained using applied harmonic analysis in combination with convex optimization \cite{CCS10,DJLSX11,ESQD05}.  These three papers do contain theoretical analyses of the convergence of their algorithms to the minimizers of specific optimization problems but not theoretical analyses of how well those optimizers actually inpaint.
Other theoretical analysis of those types of methods (imposing sparsity with a discrete dictionary) typically use a discrete model of the original image which does not allow the geometry of the
problem to be taken into account.  However, variational methods are built on continuous methods and may be analyzed
using a continuous model, for example, \cite{CKS02}.  Also, some work has been done to compare variational approaches
with those built on $\ell_1$ minimization \cite{CDOS12,Meyer}.
Finally, in works such as \cite{seis1} and \cite{seis2}, intuition behind
 why directional representation systems such as curvelets and shearlets outperform wavelets when inpainting images strongly governed by curvilinear structures such as seismic
images is given.  So, although there are many theoretical results concerning inpainting, they mainly concern algorithmic convergence or variational methods.

The preliminary results presented in the \emph{SPIE Proceedings} paper~\cite{KKZ11b} combined with the theory in this paper provide the first comprehensive analysis
of discrete dictionaries inpainting the continuum domain utilizing the novel
concept of clustered sparsity, which besides leading to a\-symptotic error bounds also makes the superior
behavior of direction\-al representation systems over wavelets precise. Along the way, our abstract model
and analysis lay a common theoretical foundation for data recovery problems when utilizing either analysis-side
$\ell_1$ minimization or thresholding as recovery schemes (Section~\ref{sec:abs_anal}).  These theoretical results are then used as tools to analyze a specific inpainting model (Sections~\ref{sec:set_up} -- \ref{sec:neg_wave}).


\subsection{A Continuum Model}\label{sec:cont-mod}

One of the first practitioners of curvelet inpainting for applications was the seismologist Felix Herrmann, who
achieved superior recovery results for images which consisted of cur\-vilinear singularities in which vertical
strips are missing due to missing sensors. These techniques were soon also
exploited for astronomical imaging, etc., the common bracket being the governing by curvilinear singularities.
It is evident, that no {\em discrete} model can appropriately capture such a geometrical content.

Thus a continuum domain model seems appropriate. In fact, in this paper we choose a distributional model which
is a distribution $w\cL$ acting on Schwartz functions $g\in\mathcal{S}'(\bR^2)$ by
\[
\ip{w\cL}{g} = \int_{-\rho}^{\rho} w(x_1) g(x_1,0) dx_1,
\]
the weight $w$ and length $2\rho$ being specified in the main body of the paper. Essentially, the weight $w$ sets up the linear singularity that is smooth in the vertical direction, while the value of $\rho$ corresponds to the length of the singularity.  Mimicking the seismic imaging
situation, we might then choose the shape of the missing part to be
\[
\cM_{h} = \mathbbm{1}_{\{|x_1| \le h\}},
\]
i.e., a vertical strip of width $2h$. Clearly, $h$ cannot be too large relative to $\rho$ or else we are erasing too much of $w\cL$.  Further, we let $P_{\cM_{h}}$ and $P_{\RR^2 \setminus \cM_{h}}$  denote the
orthogonal projection of $L^2(\RR^2)$ onto the missing part and the known part, respectively.
One task can now be formulated mathematically precise in the following way. Given
\[
    f = P_{\RR^2 \setminus \cM_h} w\cL,
\]
recover $w\cL$.

It should be mentioned that such a microlocal viewpoint was first introduced and studied in the situation of image
separation \cite{DK10}.


\subsection{Sparsifying Systems}

It was recently made precise that the optimal sparsifying systems for such images governed by anisotropic structures
are curvelets \cite{CD04} and shearlets \cite{ShearBook,Kli10}.  Of these two systems shearlets have the advantage
that they provide a unified concept of the continuum and digital
domain, which curvelets do not achieve. However, many inpainting algorithms even still use wavelets, and one might
ask whether shearlets provably outperform wavelets. In fact, we will make the superior behavior of shearlets
within our model situation precise.

For our analysis, we will use systems of wave\-lets and shearlets which are defined below.  Both systems are smooth Parseval frames.  Parseval frames generalize orthonormal ba\-ses in a manner which will be useful in the sequel.
\begin{definition}
A collection of vectors $\Phi = \{\varphi_i\}_{i \in I}$ in a separable Hilbert space $\cH$ forms a \emph{Parseval frame} for
$\cH$ if for all $x \in \cH$,
\[
\sum_{i\in I} \absip{x}{\varphi_i}^2 = \norm{x}^2.
\]
With a slight abuse of notation, given a Parseval frame $\Phi$, we also use $\Phi$ to denote the \emph{synthesis operator}
\[
\Phi: \ell_2(I) \rightarrow \cH, \quad \Phi(\{c_i\}_{i\in I}) =\sum_{i\in I} c_i \varphi_i.
\]
With this notation, $\Phi^\ast$ is called the \emph{analysis operator}.
\end{definition}


\subsubsection{Wavelets}

Meyer wavelets are some of the earliest known examples of orthonormal wavelets; they also happen to have high regularity
\cite{Dau92,Mey85}. We modify the classic system to get a decomposition of the Fourier domain that is comparable to the shearlet system that we will use.  For the construction, let $\nu \in C^\infty(\bR)$ satisfy $\nu(\cdot) + \nu(1-\cdot) =
\mathbbm{1}_{\bR}(\cdot)$, where the
\emph{indicator function} $\mathbbm{1}_A$ is defined to take the value $1$ on $A$ and $0$ on $A^c$,  and
\[
\nu(x) = \left\{ \begin{array}{lrc}  0 & : & x \leq 0, \\ 1 & : & x \geq 1. \end{array}\right.
\]
Then the Fourier transform of the 1D Meyer generator is defined by
\[
W(\xi) = \left\{ \begin{array}{lrc}  e^{-16\pi i\xi/3} \sin\left[ \frac{\pi}{2} \nu\left( 16 \abs{\xi} - 1 \right)\right] & : & \frac{1}{16}
\leq \abs{\xi} \leq \frac{1}{8}, \\   e^{-8\pi i\xi/3} \cos\left[ \frac{\pi}{2} \nu\left( 8\abs{\xi} - 1 \right)\right] & : & \frac{1}{8}
 \leq \abs{\xi} \leq \frac{1}{4}, \\0 & : & \textrm{else}, \end{array}\right.
\]
and the Fourier transform of the scaled 1D Meyer scaling function is
\[
\hat{\phi}(\xi) = \left\{ \begin{array}{lrc}  1 & : & \abs{\xi} \leq \frac{1}{16}, \\   \cos\left[ \frac{\pi}{2} \nu\left( 16 \abs{\xi}
- 1 \right)\right] & : & \frac{1}{16} \leq \abs{\xi} \leq \frac{1}{8}, \\0 & : & \textrm{else}, \end{array}\right.
\]
where we use the following
Fourier transform definition for $f \in L^1(\bR^n)$
\[
\cF{f}:=\hat{f} =\int_{\bR^n} f(x) e^{-2 \pi i \ip{x}{\cdot}} dx,
\]
(where $\ip{\cdot}{\cdot}$ is the standard Euclidean inner product) which can be naturally extended to functions in $L^2(\bR^n)$.
The inverse Fourier transform is given by
\[
\cF^{-1}f:=\check{f}= \int_{\bR^n} f(\xi) e^{2 \pi i \ip{\cdot}{\xi}}d\xi.
\]
We will not detail the interpretation of a scaling function but refer the interested reader to \cite{Dau92,Mey85}. Then we define the $C^\infty \cap L^2(\RR^2)$-functions
$W^v$, $W^h$, and $W^d$ by
\[
W^v(\xi) = \hat{\phi}(\xi_1)W(\xi_2),\quad W^h(\xi) = W(\xi_1)\hat{\phi}(\xi_2), \quad \textrm{and } W^d(\xi) = W(\xi_1)W(\xi_2).
\]
We denote
\[
\hat\psi_{\lambda}(\xi) = 2^{-j}W^{\iota}(\xi/2^j)e^{-2\pi i \ip{k}{\xi/2^j}}, \quad \lambda=(\iota,j,k).
\]
Then the {\em Parseval Meyer wavelet system} is given by
\[
\lbrace\psi_{\lambda}:\lambda=(\iota,j,k), \iota \in\{ h,v,d\}, j\in\bZ, k\in\bZ^2\rbrace
\]

We have not yet shown that this system forms a Parseval frame.  It is known (in various forms, for example \cite{Chr03,CS93b,Dau92,Jing99,MyDiss}) that if for $\{\psi^\iota \in L^2(\bR^d)\}_\iota$
\[
\sum_\iota \sum_{k \neq 0} \sum_{j \in \bZ}\vert \hat{\psi}^\iota(2^j \xi) \hat{\psi}^\iota(2^j \xi - k)\vert = 0 \textrm{ a.e. $\xi$}
\]
and
\[
\sum_{j \in \bZ} \vert \hat{\psi^\iota}(2^j \xi) \vert^2 =1 \textrm{ a.e. $\xi$},
\]
then
\[
\{ 2^{j/2} \psi^\iota(2^j \cdot - k): j \in \bZ, k \in \bZ^d, \iota \}
\]
is a Parseval frame for $L^2(\bR^d)$. The Meyer wavelet system defined above easily satisfies this.

\begin{figure}[h]
\begin{center}
\includegraphics[width = 1.5in]{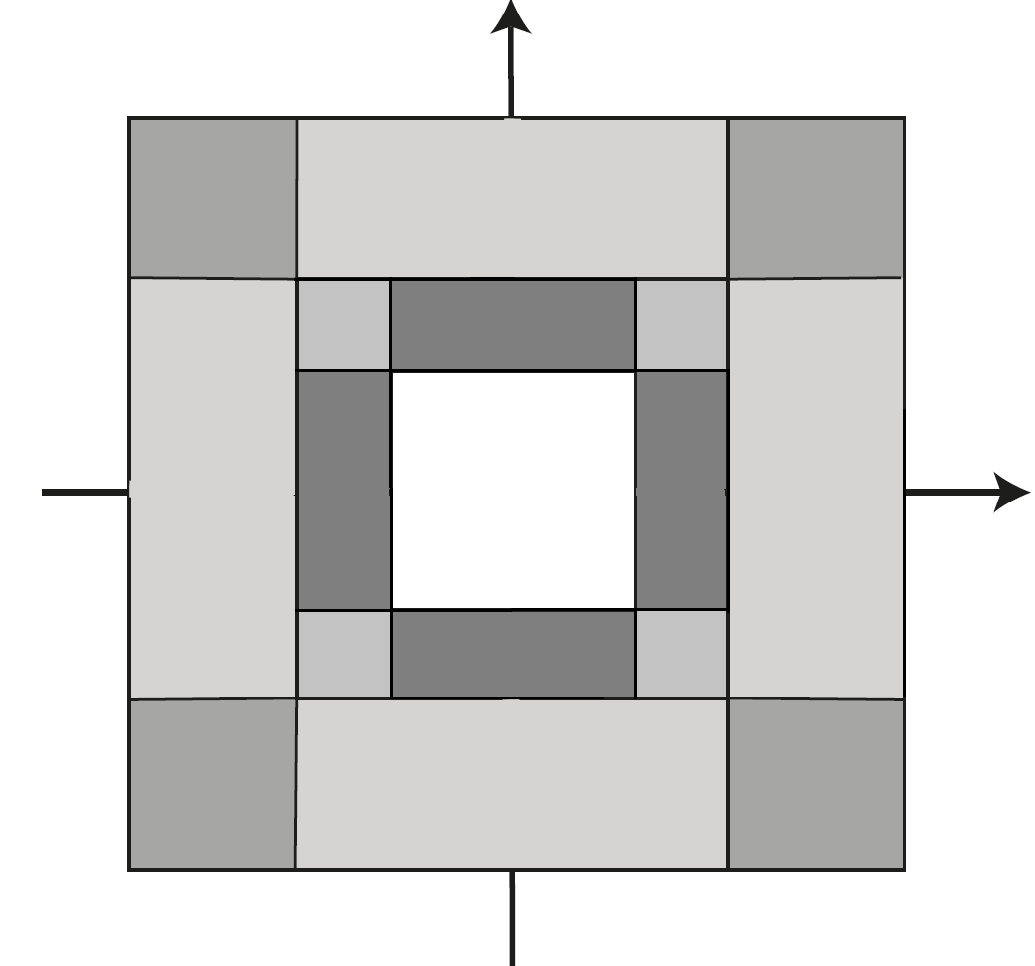}
\caption{Frequency tiling of Meyer wavelets.}
\label{fig:wave}
\end{center}
\end{figure}

\subsubsection{Shearlets}
\label{sec:shear_def}

We will use the construction of Guo and Labate of smooth Parseval frames of shearlets \cite{KanL12} which is a modification of cone-adapted shearlets (see, for example \cite{ShearBook}).  Let the parabolic scaling matrices $A_a^h$ and  $A_a^v$ and shearing matrices $S_s^h$ and $S_s^v$ be defined as
\begin{eqnarray*}
A_a^h&=&\left[\begin{array}{cc}a^2&0\\0&a\end{array}\right]
,\quad
 A_a^v=\left[\begin{array}{cc}a&0\\0&{a^2}\end{array}\right],
\\ S_s^h &=& \left[\begin{array}{cc}1&s\\0&1\end{array}\right],
\quad \textrm{and} \quad S_s^v = \left[\begin{array}{cc}1&0\\s&1\end{array}\right].
\end{eqnarray*}
We use these dilation matrices as these are used in \cite{KanL12} and given particulars of their construction, it is not straightforward to adopt their methods to the dilation matrix $\left[\begin{array}{cc}a&0\\0&\sqrt{a}\end{array}\right]$.
In addition, given the fact that the matrices defined above always have integer values when $a$ is an integer, they are reasonable from the point of view of implementation.
Let $V \in L^2(\bR)\cap C^\infty(\bR)$ satisfy $\supp V \subseteq [-1,1]$, and
\[
\sum_{k=-1}^1 |V(\xi + k)|^2 = 1, \quad \xi \in [-1,1].
\]
Further set $V^h (\xi) = V(\xi_2/\xi_1)$ and $V^v (\xi) = V(\xi_1/\xi_2)$.  For $\xi = (\xi_1,\xi_2) \in \bR^2$, define
\[
\hat{\mathbf{\phi}}(\xi) = \hat{\phi}(\xi_1)\hat{\phi}(\xi_2)
\]
and
\[
\cW(\xi) = \sqrt{|\hat{\mathbf{\phi}}(2^{-2}\xi)|^2-|\hat{\mathbf{\phi}}(\xi)|^2}.
\]
We define the following shearlet system for $L^2(\bR^2)$
\begin{equation}
\left\{ \mathbf{\phi}_k : k \in \bZ \right\}  \cup  \left\{ \sigma^\iota_{j,\ell,k,} : j \geq 0, |\ell| < 2^j, k \in \bZ^2, \iota \in \{h,v\}\right\} \cup  \left\{ \sigma_{j,\ell,k} : j \geq 0, \ell = \pm 2^j, k \in \bZ^2 \right\},\label{eqn:shearsys}
  \end{equation}
 where
 \[ \mathbf{\phi}_k = \mathbf{\phi}(\cdot-k); \]
 for $\iota \in \{h,v\}$,
 \[ \hat{\sigma}^\iota_{j,\ell,k}(\xi) = 2^{-3j/2}\cW(2^{-2j}\xi)V^\iota(\xi A_{2^{-j}}^\iota S_{-\ell}^\iota)e^{2\pi i \ip{\xi A_{2^{-j}}^\iota S_{-\ell}^\iota}{ k}}; \]
 for $j = 0$ and $\ell=\pm1$,
 \[ \hat{\sigma}_{0,\ell,k}(\xi) = \left\{ \begin{array}{lcr}   \cW(\xi)V(\frac{\xi_2}{\xi_1}-\ell)e^{2 \pi i \ip{\xi}{k}}&: &\vert \frac{\xi_2}{\xi_1}\vert \leq 1\\
\cW(\xi)V(\frac{\xi_1}{\xi_2}-\ell)e^{2 \pi i \ip{\xi}{k}}&: &\vert \frac{\xi_2}{\xi_1}\vert > 1 \end{array}\right.;\]
 and for $j \geq 1, \ell=\pm2^j$, $ \hat{\sigma}_{j,\ell,k}(\xi) =$
  \[ \left\{ \begin{array}{lcr}   2^{-\frac{3}{2}j-\frac{1}{2}}\cW(2^{-2j}\xi)V(2^j\frac{\xi_2}{\xi_1}-\ell)e^{\pi i \ip{\xi A_{2^{-j}}^h S_{-\ell}^h}{k}}&: &\vert \frac{\xi_2}{\xi_1}\vert \leq 1\\
 2^{-\frac{3}{2}j-\frac{1}{2}}\cW(2^{-2j}\xi)V(2^j\frac{\xi_1}{\xi_2}-\ell)e^{\pi i \ip{\xi A_{2^{-j}}^h S_{-\ell}^h}{k}}&: &\vert \frac{\xi_2}{\xi_1}\vert > 1 \end{array}\right..\]
 The $\sigma_{j,k\ell}$ are the ``seam" elements that piece together the $\sigma^\iota_{j,\ell,k}$ and $\mathbf{\phi}_k$.  We now have the following result from  \cite[Theorem~5p]{KanL12}.
 \begin{theorem}
 The system defined in (\ref{eqn:shearsys}) is a Parseval frame for $L^2(\bR^2)$.  Furthermore, the elements of this system are $C^\infty$ and band-limited.
 \end{theorem}
We will sometimes employ the notation
\[
\hat{\sigma}_\eta  = \hat{\sigma}^\iota_{j,\ell,k}, \quad \eta=(\iota, j,\ell,k),
\]
where $\iota \in \{h, v, \emptyset\}$, $j \in \bZ$, $k \in \bZ^2$, and $\ell \in \bZ$.

Fix a $j \geq 0$.  Then the support of each $\hat{\sigma}^\iota_{j,\ell,k}$ and $\hat{\sigma}_{j,\ell,k}$ lies in the Cartesian corona
\begin{equation}\label{eqn:corona}
C_j = [-2^{2j-1},2^{2j-1}]^2\backslash[-2^{2j-4},2^{2j-4}]^2.
\end{equation}
The position of the support inside the corona is determined by the values of $\ell$ and $\iota$, with the ``seam" elements $\hat{\sigma}_{j,\ell,k}$ having support in the corners. Thus, the shearlet system induces the frequency tiling in Figure~\ref{fig:shear}
(cf. Figure~\ref{fig:wave} for the frequency tiling of Meyer wavelets).
\begin{figure}[h]
\begin{center}
\includegraphics[width = 1.4in]{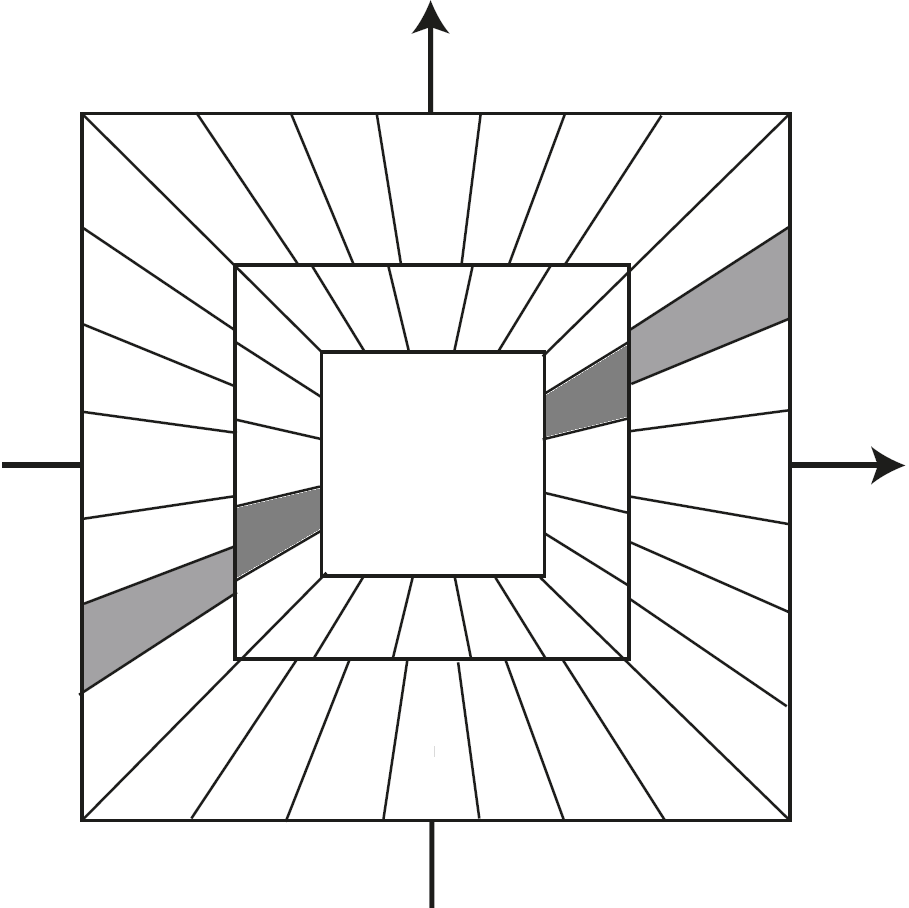}
\caption{Frequency tiling of the shearlet system.}
\label{fig:shear}
\end{center}
\end{figure}


\subsection{Recovery Algorithms}

We next decide upon a recovery strategy. Com\-pressed sensing offers a variety of such, the most common
ones being $\ell_1$ minimization and thresholding. We will also use these. However, for preparation purposes
to derive an asymptotic scale dependent analysis -- the fact that the energy of our model is arbitrary high
frequencies requires this approach --, we first perform a band-pass filtering on $w\cL$ (see Eqn.~(\ref{eqn:filter})). The band-pass filter
will be roughly speaking chosen according to the band given by the wavelets and shearlets, see Figures~\ref{fig:wave} and~\ref{fig:shear}, leading to the sequence
\[
(f_j)_j = (P_{\RR^2 \setminus \cM_{h}} w\cL_j)_j.
\]

The $\ell_1$ minimization problem we choose has the form
\begin{equation} \label{eq:Inp}
 L_j = \argmin_{L}  \| \Phi^\ast L\|_1 \mbox{ subject to } f_j = P_{\RR^2 \setminus \cM_{h}} L,
\end{equation}
where $\Phi$ is a Parseval frame.  We emphasize that this approach to inpainting minimizes the {\em analysis} coefficients and is hence related to the newly
introduced cosparsity model \cite{NDEG11,NDEG12}.  The choice will be explained further in Subsection~\ref{sec:anal}.

The thresholding strategy we choose is brutally simple. We only perform one step of hard thresholding, namely,
setting $\cT_j = \{ i : \absip{f_j}{\phi_i} \ge \beta_j\}$ for some threshold $\beta_j$, the reconstructed
image is
\begin{equation}\label{eq:Thresh}
L_j = \Phi \mathbbm{1}_{\cT_j} \Phi^\ast w\cL_j.
\end{equation}
For the asymptotic analysis, the $\beta_j$ are explicitly computed in the proofs of  Lemmas~\ref{lemm:waveletthresholdstripthrescoeff} and~\ref{lemm:curveletthresholdstripthrescoeff}.  In practice, as is usual with parameters in algorithms, one must be careful when selecting the $\beta$.

It will be surprising that the geometry of wavelets and shearlets is strong enough to achieve the same asymptotic recovery results as for
$\ell_1$ minimization for the respective systems.  However, thresholding techniques can be viewed as approximations of $\ell_1$ minimization and many parallel results have been found for $\ell_1$ minimization and thresholding. For example, $\ell_1$ minimization \cite{DK10} and thresholding \cite{DK08a} applied to the geometric separation problem both achieve asymptotic separation.  In fact, thresholding can be used to separate wavefront sets \cite{DK08a}.  Iterative thresholding algorithms have successfully approximated solutions to such diverse sparsity problems as multidimensional NMR spectroscopy \cite{Dro07} and finding row-sparse solutions to underdetermined linear systems \cite{Fou11}.


\subsection{Microlocal Analysis}
\begin{figure}[h]
\begin{center}
\begin{tabular}{cc}
\includegraphics[width = 1.8in]{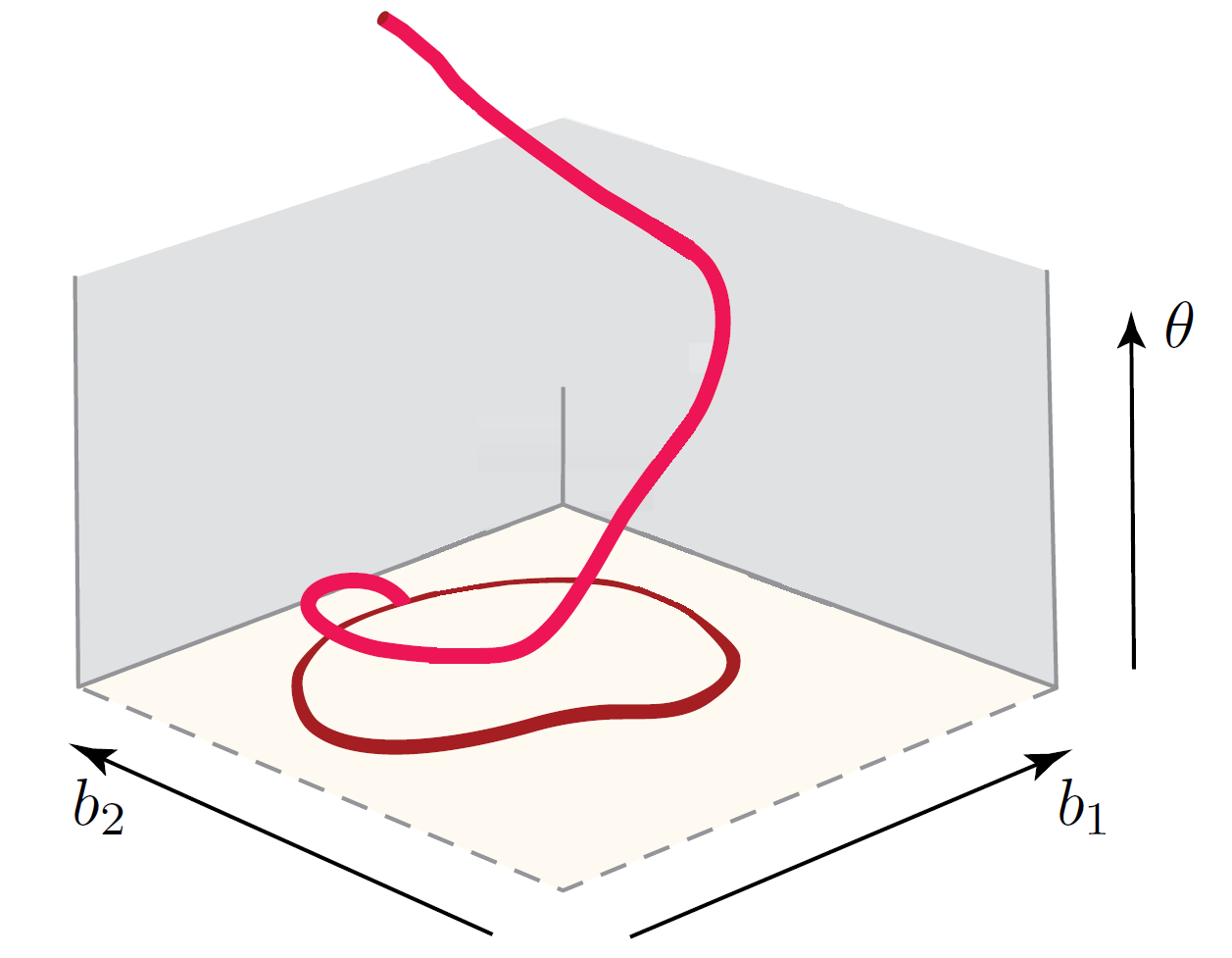}
\label{fig:WF}
 &
\includegraphics[width = 1.8in]{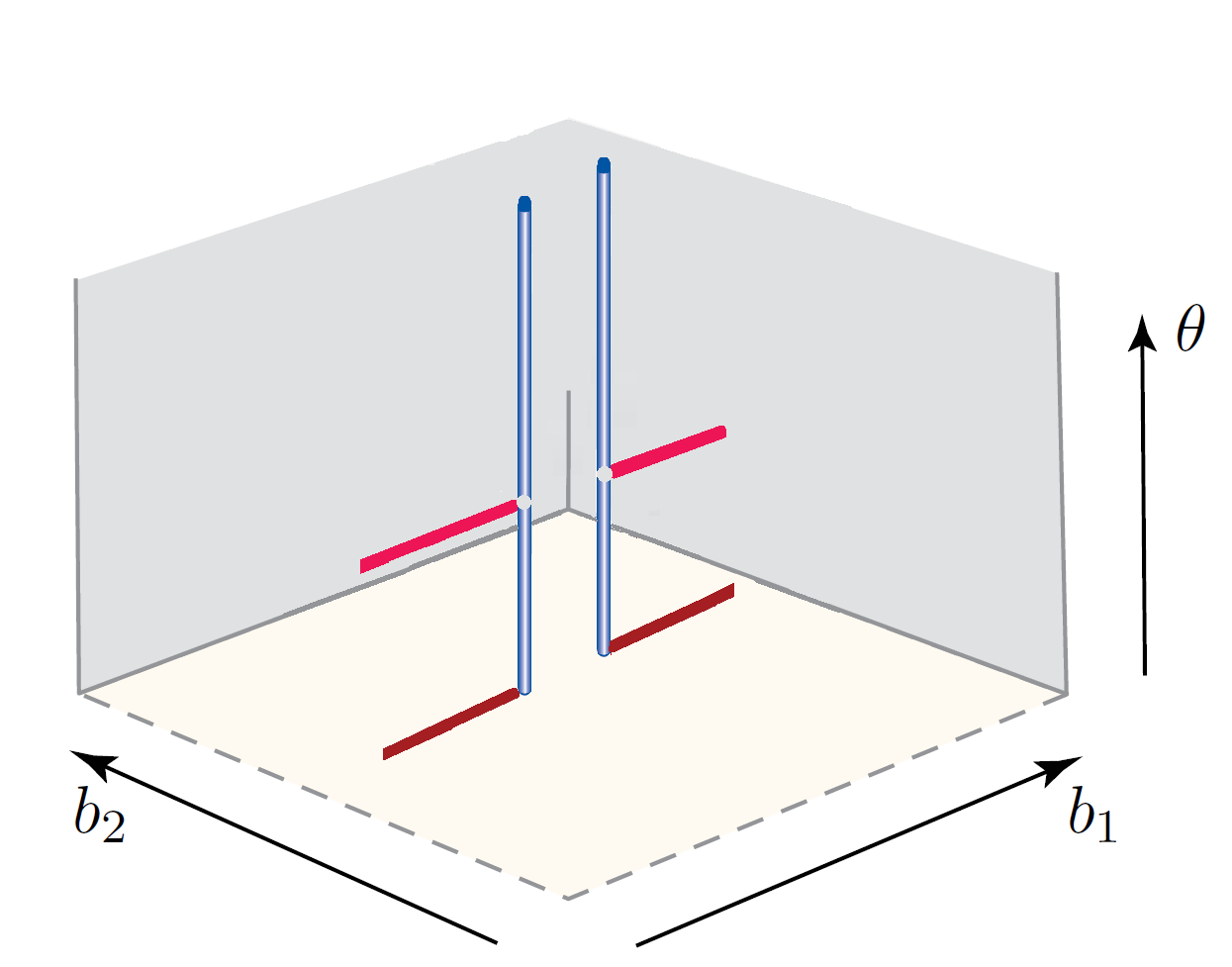}
\label{fig:WF_line}
\end{tabular}
\caption{Left: Wavefront set of a curvilinear singularity $\cC$. Right: Wavefront set of a masked linear singularity $M_h w\cL$.}
\end{center}
\end{figure}
One might ask where the geometry we mentioned before will come into play. This can best be explained and
illustrated using microlocal analysis in phase space.  For a more detailed
explanation of the fundamentals of microlocal analysis, see~\cite{Hoe03}, and for an application of microlocal analysis to derive a fundamental understanding of sparsity-based algorithms using shearlets and curvelets, see~\cite{CD05a,Grohs11,KL07}.  Phase space in this context is indexed by position-orientation
pairs $(b,\theta)$ which describe the singular behavior of a distribution.  The orientation component $\theta$ is
an element of real projective space, which for simplicity's sake we shall identify in what follows with $[0,\pi)$.
The wavefront set $WF(f)$ of a distribution $f$ is roughly the set of elements in the phase space at which $f$ is
nonsmooth.  First consider a curvilinear singularity $\cC$ along a closed curve $\tau: [0,1] \rightarrow \bR^2$:
\[ \cC = \int \delta_{\tau(t)}(\cdot)dt, \]
where $\delta_x$ is the usual Dirac delta distribution located at $x$.  As illustrated in Figure~\ref{fig:WF}, the wavefront set of $\cC$ is
\[ WF(\cC) = \{ (\tau(t),\theta(t)):t \in [0,1]\}, \]
where $\theta(t)$ is the normal direction of $\cC$ at $\tau(t)$.  Now consider the model from Section~\ref{sec:cont-mod},
 \[
 f = P_{\RR^2 \setminus \cM_h} w\cL.
 \]
As can be seen in Figure~\ref{fig:WF_line} the wavefront set of $f$ almost looks like $f$ itself except that the
wavefront set fills all possible angles (i.e., forms a spike) at the end points of the missing mask.  This is because
at the end points, the distribution is singular in all but the parallel direction.  Note that the wavefront set of the linear singularity does not have spikes at the end due to the smooth weight.  The difference
between the approximate phase space portrait of shearlets and wavelets is demonstrated in Figure~\ref{fig:WF_wav_shx}. The intuition behind the image comes from the fact that shearlets resolve the wavefront set \cite{Grohs11,KL07}.  Even though our shearlets and wavelets are smooth and thus do not have a wavefront set, by doing a continuous shearlet transform ($f \mapsto \ip{f}{a^{3/2}\sigma(S_\ell A_a \cdot-k)}$), one can get an approximation of phase space information which takes into account orientation, this is shown in Figure~\ref{fig:WF_wav_shx}.  This is similar in spirit to a wavelet spectrogram.

Furthermore, in Figure~\ref{fig:WF_clusters} (Left)
the small overlap of
the wavefront set of a cluster of shearlets with a spike in the phase space, which represents an end point of the mask
of missing information $\cM_h$, can be clearly seen.  Thus shearlet clusters are incoherent with the end points, meaning that the clusters do not overlap the spikes strongly in the phase space.  However, there is a lot of phase space overlap with the wavefront set away from the endpoints.  So it is easy to see how easily a cluster of shearlets can span a gap of missing data (Figure~\ref{fig:WF_clusters} (Right)).  Herrmann and Hennenfent call this property the ``principle of alignment" which explains why curvelets ``attain high compression on synthetic data as well as on real seismic data'' \cite{seis2}.  The phase space information of curvelets and shearlets are essentially the same \cite{GK12}.

\begin{figure}[h]
\begin{center}
\includegraphics[width = 3.0in]{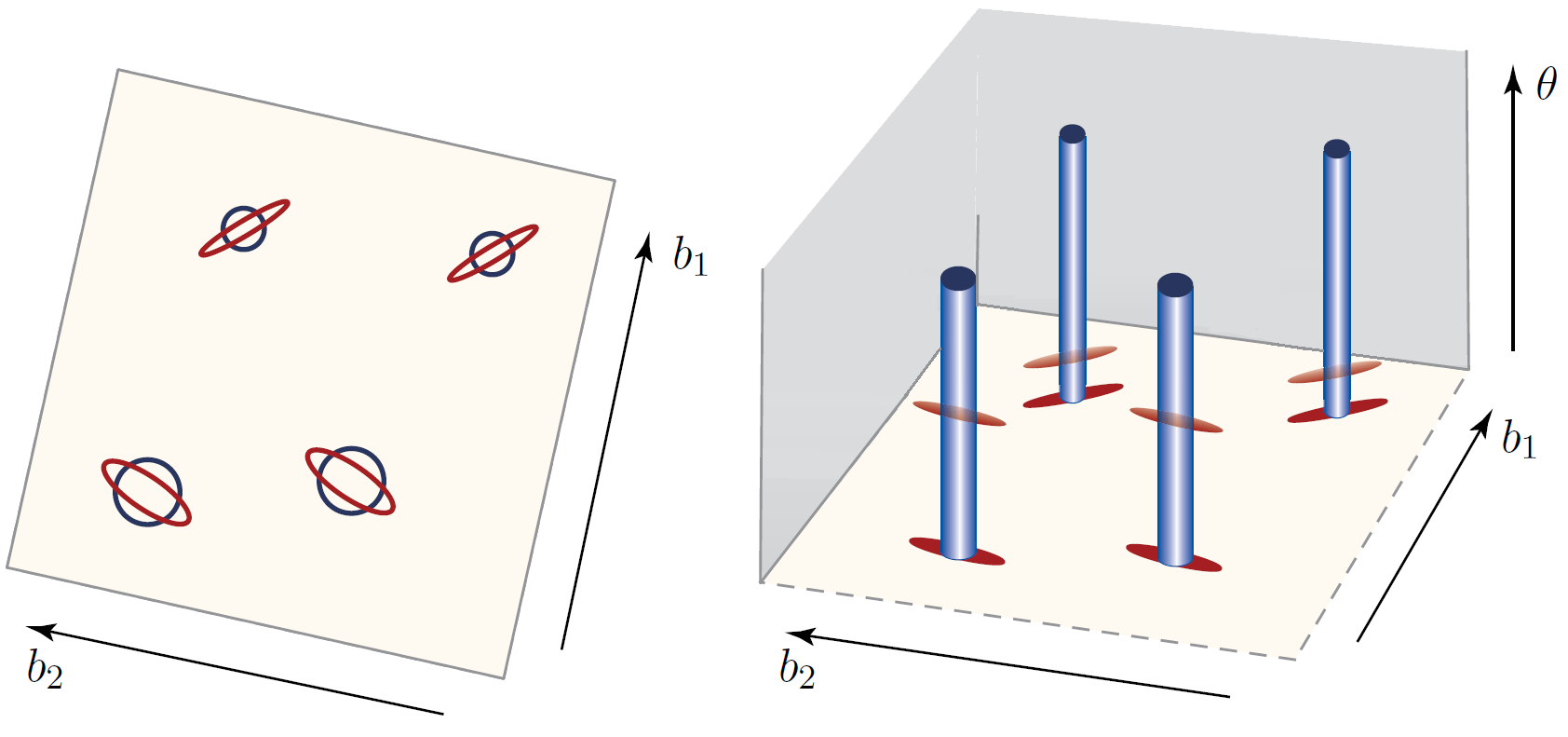}
\caption{Left: Effective supports of wavelets (disks) and shearlets (ellipses). Right: Phase space portrait of the same wavelets  (spikes) and shearlets (ellipses). }
\label{fig:WF_wav_shx}
\end{center}
\end{figure}
\begin{figure}[h]
\begin{center}
\includegraphics[width = 3.0in]{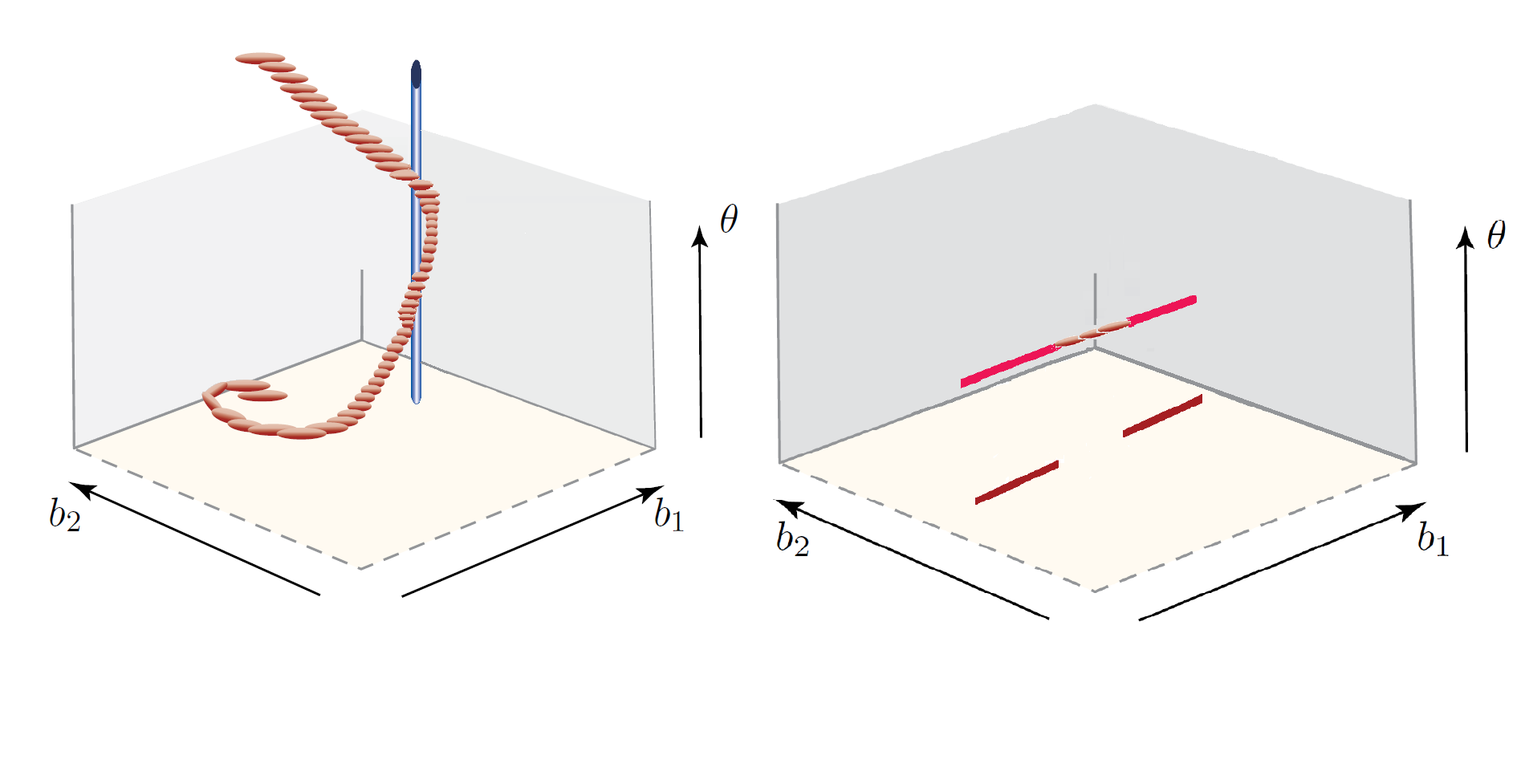}
\caption{Left: Phase space portrait of a cluster of shearlets and one single wavelet. Right: Phase space portrait of shearlets tiling a gap. }
\label{fig:WF_clusters}
\end{center}
\end{figure}

\subsection{Asymptotical Analysis}

The width of the area to be inpainted plays a key role, even when using other inpainting techniques. In \cite{CK06},
variational inpainting methods are analyzed theoretically, showing that the local thickness of the area to be
inpainted affects the success of the inpainting more than the overall size of the area to be inpainted.

Thus our analysis shall also take this into account. We accomplish this by also making the gap size $h$
dependent on the scale $j$. This leads to the problem of recovering $w\cL_j$ from knowledge of
\[
    f_j = P_{\RR^2 \setminus \cM_{h_j}} w\cL_j,
\]
for each scale $j$. Letting $L_j$ denote the recovered image by either one of the proposed algorithms,
we will show that asymptotically precise inpainting, i.e.,
\[
\frac{ \| L_j - w\cL_j \|_2}{\|w\cL_j\|_2 } \to 0, \qquad j \goto \infty,
\]
is achieved for wavelets provided that $h_j = o(2^{-2j})$ (Theorems~\ref{theo:waveletl1} and~\ref{theo:waveletthreshold}) as $j \to \infty$
and for shearlets provided that $h_j = o(2^{-j})$ (Theorems~\ref{theo:curveletl1} and~\ref{theo:curveletthreshold}) as
$j \to \infty$. In fact, this is exactly what one would imagine. Inpainting succeeds
provided that the gap size is comparable to the size of the analyzing elements.  The scale-dependent gap size allows us to analyze dependency on the size of the shearlets and wavelets in a clear way, providing a theoretical understanding of how inpainting algorithms work even though in practice the gap size is fixed.


\subsection{Wavelets versus Shearlets}

This observation seems to indicate that shearlets indeed perform better than wavelets. However,
the previously mentioned theorems just state positive results. In order to show that shearlets outperform wavelets in
the model situation which we consider, we require a negative result of the following type: If  $h_j> O(2^{-j})$  as
$j \to \infty$ and $L_j$ is recovered by wavelets, then
\[
\frac{ \| L_j - w\cL_j \|_2}{\|w\cL_j\|_2 } \not\to 0, \qquad j \goto \infty.
\]
And in fact, this is what we will prove in Theorem~\ref{thm:wavneg}. In this sense, we now have a mathematically
precise statement showing that shearlets are strictly better for inpainting in our model.

The only slight disappointment is the fact that this statement will only be proven for thresholding as the
recovery scheme. We strongly suspect that this result also holds for $\ell_1$ minimization. However, we are not
aware of any analysis tools strong enough to derive these results also in this situation.

\subsection{Our Approach}
Our analysis has focused primarily on revealing the fundamental mathematical concepts which lead to successful
image inpainting using wavelets or shearlets. The viewpoint we take, however, is that this is just the ``tip of the
iceberg,'' and the main results are susceptible of very extensive generalizations and extensions.  For example, our asymptotic analysis is based on a vertical mask of missing data from a horizontal wavefront.  Other masks applied to curved wavefronts could be considered.   The microlocal bending techniques employed in~\cite{DK10} seem to suggest that this approach will yield desirable results.
\subsection{Contents}

We begin in Section~\ref{sec:abs_anal} with an abstract analysis of data recovery via $\ell_1$ minimization
introducing clustered sparsity and concentration in a Hilbert space as tools.  We then apply the results in
Section~\ref{sec:abs_anal} to a particular class of inpainting problems which are described in Section~\ref{sec:set_up}.
In Sections~\ref{sec:pos_wave} and~\ref{sec:shear_pos}, we prove that both wavelets and shearlets, respectively,
are able to inpaint a missing band but that shearlets can handle wider gaps.  It is shown in Section~\ref{sec:neg_wave}
that the inpainting result for wavelets in Section~\ref{sec:pos_wave} is tight; i.e., shearlets strictly
outperform wavelets in the considered model situation. We discuss future directions of research and limitations of the current model in Section~\ref{sec:ext}.  Finally, Section~\ref{sec:aux_shear} is an appendix that
contains auxiliary results concerning shearlets needed for Section~\ref{sec:shear_pos}.


\section{Abstract Analysis of  Data Recovery}\label{sec:abs_anal}

We start by analyzing missing data recovery via $\ell_1$ minimization and thresholding in an abstract model
situation. The error estimates we will derive can be applied in a variety of situations. In this paper,
-- as discussed  before -- we aim to utilize them to analyze inpainting via wavelets and shearlets
following a continuum domain model. In fact, these error estimates will later on be applied to each
scale while deriving an asymptotic analysis.


\subsection{Abstract Model}

Let $x^0 \in \cH$ be a signal in a Hilbert space $\cH$. To model the data recovery problem correctly,
we assume that $\cH$ can be decomposed into a direct sum
\[
\cH = \cH_M \oplus \cH_K
\]
of a subspace $\cH_M$ which is associated with the {\em missing} part of $x^0$ and a subspace $\cH_K$ which
relates to the {\em known} part of the signal. Further, let $P_M$ and $P_K$ denote the orthogonal projections
onto those subspaces, respectively. The problem of data recovery can then be formulated as follows: Assuming
that $P_K x^0$ is known to us, recover $x^0$.

Following the philosophy of compressed sensing, suppose that there exists a Parseval frame $\Phi$ which
-- in a way yet to be made precise -- sparsifies the original signal $x^0$. Either $\Phi$ can be selected
non-adaptively such as choosing a wavelet or shearlet system which will be our avenue in the sequel, or $\Phi$ can be
chosen adaptively using dictionary learning algorithms such as \cite{K-SVD,EAHH99,OF1997}.

\medskip

To already draw the connection to the special situation of inpainting at this point, assume that $\cH = L^2(\RR^2)$.
If the measurable subset $B \subseteq \RR^2$ is the missing area of the image, we set $\cH_K = L^2(\RR^2 \setminus B)$
and $\cH_M = L^2(B)$.


\subsection{Inpainting via $\ell_1$ Minimization}\label{sec:anal}
A first methodology from compressed sensing to achieve recovery is $\ell_1$ minimization, which recovers
the original signal by solving
\[
    (\mbox{\sc Inp}) \qquad    {x}^\star = \argmin_{x}  \| \Phi^\ast x\|_1 \mbox{ subject to } P_K x = P_K x^0.
\]
We wish to remark that in this problem, the norm is placed on the {\em analysis} coefficients rather than on the
{\em synthesis} coefficients as in \cite{DE03,EB02} and other papers on basis pursuit. Since we intend to also
apply this optimization problem in the situation when $\Phi$ does not form a basis but merely a frame, the analysis and synthesis approaches are different. One reason to do this is to
avoid numerical instabilities which are expected to occur since, for each $x \in \cH$, the linear system of equations $x = \Phi c$
has infinitely many solutions, only the specific solution $\Phi^\ast x$ is analyzed.  Also, since we are only interested in correctly inpainting and not in computing the sparsest expansion, we can circumvent possible problems by solving the inpainting problem by selecting a particular coefficient sequence which expands out to the $x$, namely the analysis sequence.  A similar strategy was pursued in \cite{KKZ11b} and
\cite{DK08a}. Various inpainting algorithms which are based on the core idea of (\mbox{\sc Inp})
combined with geometric separation are heuristically shown to be successful in \cite{CCS10,DJLSX11,ESQD05}.

Interestingly, this minimization problem can be also regarded as a mixed $\ell_1$-$\ell_2$ problem \cite{KT09}, since
the analysis coefficient sequence $\Phi^\ast x$ is exactly the minimizer of
\[
\min \{ \norm{c}_2: c \in \ell_2, x = \Phi c \},
\]
that is, the coefficient sequence which is minimal in the  $\ell_2$ norm. The optimization problem in (\mbox{\sc Inp})
may also be thought of a relaxation of the \emph{cosparsity} problem
\[
     {x}^\star = \argmin_{x}  \| \Phi^\ast x\|_0 \mbox{ subject to } P_K x = P_K x^0.
\]
Theoretical results concerning cosparsity may be found in \cite{NDEG11,NDEG12}.

We also consider the noisy case.  Assume now that we know $\tilde{x} = P_K x^0 + n$, where $x^0$ and $n$ are unknown, but
$n$ is assumed to be small in the sense of $\norm{\Phi^\ast n}_1 \leq \epsilon$ for small $\epsilon$.  Also, clearly $n = P_K n$.  Then we solve
\[
(\mbox{\sc InpNoise}) \quad \tilde{x}^\star = \argmin_x \norm{\Phi^\ast x}_1 \mbox{ subject to $P_Kx = \tilde{x}$}.
\]

To analyze this optimization problem, we require the following notion, which intuitively measures the maximal fraction of the total
$\ell_1$ norm which can be concentrated to the index set $\Lambda$ restricted to functions in $\cH_M$. In this sense, the
geometric relation between the missing part $\cH_M$ and expansions in $\Phi$ is encoded.

\begin{definition}
Let $\Phi$ be a Parseval frame, and let
$\Lambda$ be a index set of coefficients. We then define the {\em concentration on $\cH_M$} by
\[ \kappa = \kappa(\Lambda, \cH_M) =
\sup_{f \in \cH_M} \frac{\norm{\mathbbm{1}_{\Lambda} \Phi^\ast f}_1}{\norm{\Phi^\ast f}_1}.\]
\end{definition}

This notion allows us to formulate our first estimate concerning the $\ell_2$ error of the reconstruction via (\mbox{\sc Inp}).
The reader should notice that the considered error $\norm{x^\star-x^0}_2$ is solely measured on $\cH_M$, the masked space,
since $P_K x^\star = P_K x^0$ due to the constraint in (\mbox{\sc Inp}).  Another important notion is that of \emph{clustered sparsity}.

\begin{definition}
Fix $\delta > 0$.  Given a Hilbert space $\cH$ with a Parseval frame $\Phi$, $x \in \cH$ is \emph{$\delta$-clustered sparse in $\Phi$ (with respect to $\Lambda$} if
\[
\norm{\mathbbm{1}_{\Lambda^c}\Phi^\ast x}_1 \leq \delta,
\]
where given a space $X$ and a subset $A \subseteq X$, $A^c$ denotes $X\backslash A$.
\end{definition}

We now present a pair of lemmas which were first published in \cite{KKZ11b} without proof.

\begin{lemma}
\label{lemm:abstractestimatewithkappa}
Fix $\delta > 0$ and suppose that $x^0$ is $\delta$-clustered  sparse in $\Phi$.  Let $x^\star$ solve (\mbox{\sc Inp}). Then
\[
  \norm{x^\star-x^0}_2 \le \frac{2\delta}{1-2\kappa}.
\]
\end{lemma}

The noiseless case Lemma~\ref{lemm:abstractestimatewithkappa} holds as a corollary to the case with noise, which follows.

\begin{lemma}\label{lem:inpnz}
Fix $\delta > 0$ and suppose that $x^0$ is $\delta$-clustered  sparse in $\Phi$.  Let $\tilde{x}^\star$ solve (\mbox{\sc InpNoise}).
Also assume that the noise satisfies $\Vert \Phi^\ast n \Vert_1 \leq \epsilon$.
Then
\[
\Vert \tilde{x}^\star - x^0 \Vert_2 \leq  \frac{2 \delta + (3+\kappa^2)\epsilon}{1-2\kappa} .
\]
\end{lemma}

\begin{proof}
Since $\Phi$ is Parseval,
\begin{equation}\label{eqn:par}
\Vert \tilde{x}^\star - x^0 \Vert_2 \leq \Vert \Phi^\ast(\tilde{x}^\star - x^0) \Vert_1.
\end{equation}
We invoke the relation $P_K \tilde{x}^\star = P_K x^0 +n$, which implies that $P_K (\tilde{x}^\star - x^0) = n$.
Using the definition of $\kappa$, we obtain
\begin{eqnarray}\label{eqn:kapeq}
\Vert \mathbbm{1}_{\Lambda} \Phi^\ast (\tilde{x}^\star - x^0 ) \Vert_1 & \leq & \Vert \mathbbm{1}_{\Lambda} \Phi^\ast
P_M(\tilde{x}^\star - x^0) \Vert_1 + \Vert \mathbbm{1}_{\Lambda} \Phi^\ast n \Vert_1 \nonumber  \leq  \kappa \Vert
\Phi^\ast P_M(\tilde{x}^\star - x^0)\Vert_1 + \Vert  \Phi^\ast n \Vert_1 \nonumber \\
 & \leq & \kappa \Vert \Phi^\ast (\tilde{x}^\star - x^0)\Vert_1 + (1 + \kappa) \Vert  \Phi^\ast n \Vert_1  \leq   \kappa \Vert \Phi^\ast (\tilde{x}^\star - x^0)\Vert_1 + (1 + \kappa) \epsilon.
\end{eqnarray}
It follows that
\[
\Vert \Phi^\ast (\tilde{x}^\star - x^0) \Vert_1  =  \Vert \mathbbm{1}_{\Lambda} \Phi^\ast(\tilde{x}^\star - x^0)\Vert_1
+  \Vert \mathbbm{1}_{\Lambda^c} \Phi^\ast(\tilde{x}^\star - x^0)\Vert_1 \leq  \kappa \Vert \Phi^\ast(\tilde{x}^\star - x^0) \Vert_1 +  \Vert \mathbbm{1}_{\Lambda^c} \Phi^\ast(\tilde{x}^\star - x^0)\Vert_1 + (1 + \kappa)\epsilon.
\]
The clustered sparsity of $x^0$ now implies
\[
\Vert \Phi^\ast (\tilde{x}^\star - x^0)\Vert_1
  \leq  \frac{1}{1- \kappa} \left(\Vert \mathbbm{1}_{\Lambda^c}
\Phi^\ast (\tilde{x}^\star - x^0) \Vert_1 +(1 + \kappa)\epsilon \right)  \leq \frac{1}{1- \kappa} \left( \Vert \mathbbm{1}_{\Lambda^c} \Phi^\ast \tilde{x}^\star \Vert_1 +\delta + (1 + \kappa)\epsilon\right). \label{eqn:inprel}
\]
Applying the sparsity of $x^0$ again and the minimality of $\tilde{x}^\star$, we have
\begin{eqnarray*}
\Vert \mathbbm{1}_{\Lambda^c} \Phi^\ast \tilde{x}^\star \Vert_1 & = & \Vert \Phi^\ast \tilde{x}^\star \Vert_1 -
\Vert \mathbbm{1}_{\Lambda} \Phi^\ast \tilde{x}^\star \Vert_1 
 \leq  \Vert \Phi^\ast (x^0 + n) \Vert_1 - \Vert \mathbbm{1}_{\Lambda} \Phi^\ast \tilde{x}^\star \Vert_1 \\
& \leq & \Vert \Phi^\ast x^0\Vert_1 - \Vert \mathbbm{1}_{\Lambda} \Phi^\ast \tilde{x}^\star \Vert_1 + \epsilon  \leq  \Vert \Phi^\ast x^0\Vert_1 + \Vert \mathbbm{1}_{\Lambda} \Phi^\ast (\tilde{x}^\star - x^0) \Vert_1 -
\Vert \mathbbm{1}_{\Lambda} \Phi^\ast x^0 \Vert_1 + \epsilon \\
& \leq & \Vert \mathbbm{1}_{\Lambda} \Phi^\ast (\tilde{x}^\star - x^0) \Vert_1 + \delta + \epsilon.
\end{eqnarray*}
Using (\ref{eqn:kapeq}) and (\ref{eqn:inprel}), this leads to
\begin{eqnarray*}
\Vert \Phi^\ast (\tilde{x}^\star - x^0)\Vert_1 &\leq& \frac{1}{1- \kappa} \left( \Vert \mathbbm{1}_{\Lambda^c}
\Phi^\ast \tilde{x}^\star \Vert_1 +\delta+ (1+\kappa)\epsilon  \right)  \\
& \leq & \frac{1}{1- \kappa} \left( \Vert \mathbbm{1}_{\Lambda} \Phi^\ast (\tilde{x}^\star - x^0) \Vert_1 + 2\delta \right)
+ \frac{(2+\kappa)\epsilon}{1-\kappa} \\
& \leq & \frac{1}{1- \kappa} \left( \kappa \Vert  \Phi^\ast (\tilde{x}^\star - x^0) \Vert_1 + 2\delta \right)
+ \frac{(3+2\kappa)\epsilon}{1-\kappa} .
\end{eqnarray*}
Combining this with (\ref{eqn:par}), we finally obtain
\[
\Vert \tilde{x}^\star - x^0 \Vert_2 \leq \left( 1 - \frac{\kappa}{1-\kappa}\right)^{-1} \frac{2 \delta
+ (3+2\kappa)\epsilon}{1-\kappa} = \frac{2 \delta + (3+2\kappa)\epsilon}{1-2\kappa}
\]
\end{proof} 

We now establish a relation between the concentration $\kappa(\Lambda,\cH_M)$ on $\cH_M$ and the notion of cluster coherence
$\mu_c$ first introduced in \cite{DK10}. For this, by abusing notation, we will write $P_M \Phi = \{ P_M \varphi_i \}_i$ and $P_K \Phi =  \{ P_K \varphi_i \}_i$ for the
projected frame elements.

To first introduce the notion of cluster coherence, recall that in many studies of $\ell_1$ optimization, one utilizes the
\emph{mutual coherence}
\[
    \mu (\Phi_1, \Phi_2) = \max_{j}  \max_{i} | \langle \varphi_{1i}, \varphi_{2j}
\rangle|,
\]
whose importance was shown by \cite{DH01}.  This may be called the {\it singleton coherence.}  We modify the definition to take
into account clustering of the coefficients arising from the geometry of the situation.

\begin{definition}
Let $\Phi_1 =\{\varphi_{1i}\}_{i \in I}$ and $\Phi_2 =\{\varphi_{2j}\}_{j \in J}$ lie in a Hilbert space $\cH$ and
let $\Lambda \subseteq I$. Then the
{\em cluster coherence} \index{cluster coherence}  $\mu_c (\Lambda,
\Phi_1; \Phi_2)$ of $\Phi_1$ and $\Phi_2$ with respect to $\Lambda$ is defined by
\[
\mu_c(\Lambda, \Phi_1; \Phi_2) = \max_{j \in J}  \sum_{i \in \Lambda} | \langle \varphi_{1i}, \varphi_{2j} \rangle|.
\]
\end{definition}

The following relation is a specific case of Proposition 3.1 in \cite{KKZ11b}. We include a proof for completeness.

\begin{lemma}
\label{lemm:kappamu}
We have
\[ \kappa(\Lambda,\cH_M) \le \mu_c(\Lambda,P_M \Phi;P_M \Phi) = \mu_c(\Lambda,P_M \Phi;\Phi).\]
\end{lemma}
\begin{proof}
For each $f \in \cH_M$, we choose a coefficient sequence $\alpha$
such that $f=\Phi \alpha$ and $\norm{\alpha}_1 \le
\norm{\beta}_1$ for all $\beta$ satisfying $f=\Phi \beta$.
Invoking the fact that $\Phi$ is a tight frame, hence $f = \Phi \Phi^* \Phi \alpha$,
and the fact that $f = (P_M \Phi) \alpha$, we obtain
\begin{eqnarray*}
\norm{\mathbbm{1}_{\Lambda} \Phi^* f}_1
& = & \norm{\mathbbm{1}_{\Lambda} (P_M \Phi)^* f}_1 =  \norm{\mathbbm{1}_{\Lambda} (P_M \Phi)^* (P_M \Phi) \alpha}_1\\
& \le & \sum_{i \in \Lambda}\left(\sum_j
\absip{P_M \phi_{i}}{P_M \phi_{j}} |\alpha_{j}|\right) = \sum_{j}\left(\sum_{i \in \Lambda}
\absip{P_M\phi_{i}}{P_M\phi_{j}}\right) |\alpha_{j}|\\[1ex]
& \le &  \mu_c(\Lambda,P_M\Phi;P_M\Phi) \norm{\alpha}_1\le  \mu_c(\Lambda,P_M\Phi;P_M\Phi) \norm{\Phi^* \Phi \alpha}_1\\[1ex]
& = & \mu_c(\Lambda,P_M\Phi;P_M\Phi) \norm{\Phi^* f}_1.
\end{eqnarray*}
\end{proof} 

Combining Lemmata \ref{lemm:abstractestimatewithkappa} and \ref{lemm:kappamu} proves the final noiseless estimate and
combining Lemmata \ref{lem:inpnz} and \ref{lemm:kappamu} proves the final estimate with noise:

\begin{proposition}
\label{prop:l1estimate}
Fix $\delta > 0$ and suppose that $x^0$ is $\delta$-clustered sparse in $\Phi$.  Let $x^\star$ solve (\mbox{\sc Inp}). Then
\[
  \norm{x^\star-x^0}_2 \le \frac{2\delta}{1-2\mu_c(\Lambda,P_M \Phi;\Phi)}.
\]
\end{proposition}

\begin{proposition}
Fix $\delta > 0$ and suppose that $x^0$ is $\delta$-clustered sparse in $\Phi$.  Let $\tilde{x}^\star$ solve (\mbox{\sc InpNoise}).
Also assume that the noise satisfies $\Vert \Phi^\ast n \Vert_1 \leq \epsilon$.
Then
\[
\Vert \tilde{x}^\star - x^0 \Vert_2 \leq \frac{2\delta + (3 + 2\kappa)\epsilon}{1 - 2\mu_c(\Lambda,P_M\Phi;\Phi)}.
\]
\end{proposition}

Let us briefly interpret this estimate, first focusing on the noiseless case. As expected the error decreases
linearly with the clustered sparsity. It should also be emphasized that both clustered sparsity and cluster coherence depend on
the chosen ``geometric set of indices'' $\Lambda$. Thus this set is crucial for determining whether $\Phi$ is a good dictionary
for inpainting. This will be illustrated in the sequel when considering a particular situation; however, $\Lambda$ is merely an
analysis tool and explicit knowledge of it is not necessary to recover data. Note that in general, the larger the set $\Lambda$ is,
the smaller $\norm{\mathbbm{1}_{\Lambda^c}\Phi^\ast x^0}_1$ is (i.e., $x^0$ is $\delta$-relatively sparse for a smaller $\delta$)
and the larger the cluster coherence is.  This seems to be a contradiction, but if $\Phi$ sparsifies $x^0$ well, then a small
set $\Lambda$ can be chosen which keeps $\norm{\mathbbm{1}_{\Lambda^c}\Phi^\ast x^0}_1$ small.  Finally, considering the noisy
case, as also expected the error estimate depends linearly on the $\ell_2$ bound
for the noise.


\subsection{Inpainting via Thresholding}

Another fundamental methodology from compressed sensing for sparse recovery is thresholding. The beauty
of this approach lies in its simplicity and its associated fast algorithms. Typically, it is also possible
to prove success of recovery in similar situations as in which $\ell_1$ minimization succeeds.

Various thresholding strategies are available such as iterative thresholding, etc. It is thus surprising
that the most simple imaginable strategy, which is to perform just {\em one step} of hard thresholding,
already allows for error estimates as strong of for $\ell_1$ minimization. We start by presenting this
thresholding strategy. For technical reasons, -- note also that this is no restriction at all -- we now
assume that the Parseval frame $\Phi = (\phi_i)_i$ consists of frame vectors with equal norm, i.e.,
$\norm{\phi_i} = c$ for all $i$.

\begin{figure}[h]
\centering
\framebox{
\begin{minipage}[h]{5in}
\vspace*{0.3cm}
{\sc \underline{One-Step-Thresholding}}

\vspace*{0.5cm}

{\bf Parameters:}\\[-2ex]
\begin{itemize}
\item Incomplete signal $\tilde{x} = P_K x^0$ (noiseless) or $P_K x^0 + n$ (with noise).\\[-2ex]
\item Thresholding parameter $\beta$.
\end{itemize}

\vspace*{0.25cm}

{\bf Algorithm:}\\[-2ex]
\begin{itemize}
\item[1)] {\it Threshold Coefficients with Respect to Frame $\Phi$:}\\[-2ex]
\begin{itemize}
\item[a)] Compute $\ip{\tilde{x}}{\phi_i}$ for all $i$. \\[-2ex]
\item[b)] Apply threshold and set $\cT = \{ i : \absip{\tilde{x}}{\phi_i} \ge \beta\}$. \\[-2ex]
\end{itemize}
\item[2)] {\it Reconstruct Original Signal:}\\[-2ex]
\begin{itemize}
\item[a)] Compute $x^\star = \Phi \mathbbm{1}_{\cT} \Phi^\ast \tilde{x}$.
\end{itemize}
\end{itemize}

\vspace*{0.25cm}

{\bf Output:}\\[-2ex]
\begin{itemize}
\item Significant thresholding coefficients: $\cT$.\\[-2ex]
\item Approximation to $x^0$: $x^\star$.
\end{itemize}
\vspace*{0.01cm}
\end{minipage}
}
\caption{{\sc One-Step-Thresholding} Algorithm to reconstruct $x^0$ from noiseless $P_K x^0$ or noisy $P_K x^0 +n$.}
\label{fig:onestepthresholding}
\end{figure}

The following result provides us with an estimate for the $\ell_2$ error of
the synthesized signal $x^\star$ computed via {\sc One-Step-Thresholding}.

\begin{proposition} \label{prop:thresholdingestimate}
Let $\cT$ and $x^\star$ be computed via the algorithm {\sc One-Step-Thresholding} (Figure \ref{fig:onestepthresholding}) for
noiseless data, and for $\delta > 0$ assume that $x^0$ is relatively sparse in $\Phi$ with respect to $\cT$.  Then
\[
  \norm{x^\star-x^0}_2 \le c  \left[ \delta + \norm{\mathbbm{1}_{\cT} \Phi^\ast P_M x^0}_1 \right].
\]
\end{proposition}

As before, Proposition \ref{prop:thresholdingestimate} follows as a corollary to the case with noise:

\begin{proposition}\label{prop:thresh_success}
Let $\mathcal{T}$ and $x^\star$ be computed via the algorithm {\sc One-Step-Thresholding} for data with noise, and for $\delta > 0$  assume
that $x^0$ is relatively sparse in $\Phi$ with respect to $\mathcal{T}$.
Also assume that the noise satisfies $\Vert \Phi^\ast n \Vert_1 \leq \epsilon$. Then
\[
\Vert x^\star - x^0 \Vert_2 \leq c  \left( \Vert\mathbbm{1}_{\mathcal{T}}\Phi^\ast P_Mx^0\Vert_1+ \delta + \epsilon\right).
\]
\end{proposition}

\begin{proof}
Invoking the decomposition of $\mathcal{H}$ and the fact that $\Phi$ is Parseval,
\[
\Vert x^\star - x^0 \Vert_2  =  \Vert \Phi \mathbbm{1}_{\mathcal{T}}\Phi^\ast(P_K x^0 + n) - \Phi \Phi^\ast P_K x^0 - P_M x^0 \Vert_2  =  \Vert \Phi \mathbbm{1}_{\mathcal{T}^c} \Phi^\ast P_K x^0 + \Phi \mathbbm{1}_{\mathcal{T}}\Phi^\ast n - P_M x^0 \Vert_2.
\]
Since
\[
P_M x^0 = \Phi \mathbbm{1}_{\mathcal{T}} \Phi^\ast P_M x^0 + \Phi \mathbbm{1}_{\mathcal{T}^c} \Phi^\ast P_M x^0
\]
and $P_K x^0 + P_M x^0 = x^0$, it follows that
\[
\Vert x^\star - x^0 \Vert_2 \leq \Vert\Phi \mathbbm{1}_{\mathcal{T}^c} \Phi^\ast x^0 \Vert_2 +\Vert \Phi \mathbbm{1}_{\mathcal{T}}
\Phi^\ast P_M x^0 \Vert_2 + \Vert \Phi \mathbbm{1}_{\mathcal{T}}\Phi^\ast n \Vert_2.
\]
It follows from the equal-norm condition on the frame $\Phi$ that for any $\ell_1$ sequence $x$,
\[
\Vert \Phi x \Vert_2 \leq c  \Vert x \Vert_1.
\]
Applying the clustered sparsity of $x^0$ we obtain
\[
\Vert x^\star - x^0 \Vert_2 \leq c  \left(\Vert\mathbbm{1}_{\mathcal{T}}\Phi^\ast P_Mx^0\Vert_1 + \delta + \epsilon \right),
\]
which is what we intended to prove.
\end{proof} 

As before, let us also interpret this estimate. Now the situation is slightly different from the estimate for
the $\ell_1$ approach. Again, the estimate depends linearly on the clustered sparsity and the noise. The difference
now is the appearance of the term $\Vert\mathbbm{1}_{\mathcal{T}}\Phi^\ast P_Mx^0\Vert_1$ in the numerator
instead of the cluster coherence in the denominator.  Note, however, that
\[
\norm{\mathbbm{1}_{\cT}\Phi^\ast P_Mx^0}_1 \leq \kappa \norm{\Phi^\ast P_M x^0}_1\leq \mu_c (\cT, P_M \Phi; \Phi) \norm{\Phi^\ast P_M x^0}_1.
\]
Thus both in the $\ell_1$ minimization case Proposition~\ref{prop:l1estimate} and in the thresholding case Proposition~\ref{prop:thresholdingestimate}, the bound on the error is lower when the cluster coherence is lower.  Furthermore, $\norm{\Phi^\ast P_M x^0}_1$ is a quantification of how much of the signal is missing, which clearly can not be too high.


\section{Mathematical Model}
\label{sec:set_up}

We next provide a specific mathematical model which is motivated by the fact that images are typically
governed by edges, which can most prominently be seen in, for example, seismic imaging (Figure~\ref{fig:seis}).
\begin{figure}[h]
\begin{center}
\includegraphics[height=1.3in]{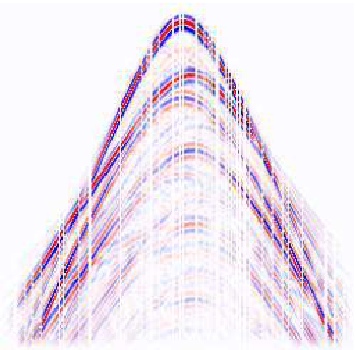}
\caption{Synthetic seismic data with randomly distributed missing traces -- Hennenfent and Hermmann \cite{seis3}}
\label{fig:seis}
\end{center}
\end{figure}
Following this line of thought, our model is based on line singularities
-- which can as explained later be extended to curvilinear singularities -- with missing data of the forms
as gaps or holes.  In this section, such a model for the original image and the mask will be introduced. Since the analysis
we derive later is based on the behavior in Fourier domain, the Fourier content of the models is another
focus.


\subsection{Image Model}
Inspired by seismic data with missing traces, an example of which is found in Figure~\ref{fig:seis}, we define our mathematical model.  The data can be viewed as a collection of curvilinear singularities which are missing nearly vertical strips of information.  We first simplify the model by considering linear singularities.  As shearlets are directional systems, we then simplify the model so that the linear singularity is horizontal. The specific mathematical model that we shall analyze is as follows.
Let $w: \bR \mapsto [0,1]$ be a smooth function that is supported in $[-\rho,\rho]$, where we always assume that
$\rho$ is sufficiently large, in particular, much larger than $h$ (a measure of the missing data which will be
defined in the next subsection).  For now, we consider as a prototype of a line singularity the weighted distribution $w\cL$ acting on
tempered distributions $\mathcal{S}'(\bR^2)$ by
\[
\ip{w\cL}{f} = \int_{-\rho}^{\rho} w(x_1) f(x_1,0) dx_1.
\]
Notice that this distribution is supported on the segment
\[
[-\rho,\rho] \times \{0\}
\] of the $x$-axis, hence can be
employed as a model for a horizontal linear singularity. The weighting was chosen to ensure that we are dealing with
an $L_2$-function after filtering. The Fourier transform of $w\cL$ can be computed to be
\[
\ip{\widehat{w\cL}}{f} = \ip{w\cL}{\hat f}=\int_{\bR} \hat{w}(\xi_1) \int_{\bR} f(\xi_1,\xi_2) d\xi_2 d\xi_1.
\]

Let now $\check{F_j}$ be a filter corresponding to the frequency corona $C_j$ at level $j$ (see Equation (\ref{eqn:corona})).
defined by its Fourier transform $F_j$,
\[
F_j = \sum_{\iota\in\{h,v,d\}} \left( W^\iota(2^{-2j}\xi) + W^\iota(2^{-2j-1}\xi) \right).
\]
To simplify the proofs for wavelets, we also define
\[
\widetilde{F}_j = \sum_{\iota\in\{h,v,d\}} W^\iota(2^{-j}\xi),
\]
so that $F_j = \widetilde{F}_{2j} + \widetilde{F}_{2j+1}$.  We use two bands for the wavelets so that the wavelet and shearlet systems will be compared on the same frequency corona.  This makes sense as the base ($j =1$) dilation for the $2D$ wavelets has determinant $4$, while the base dilation for the shearlets has determinant $8$. We consider the filtered version of $w\cL$ which we denote by $w\cL_j$, i.e.,
\beq \label{eqn:filter}
w\cL_j = w\cL \star \check{F_j} =\int_{\bR^2} w\cL(\cdot-t) \check F_j(t)dt.
\eeq

The next result provides us with an estimate of the norm of $w\cL_j$.

\begin{lemma}
\label{lemm:NormofLj}  For some $c > 0$,
\[
\| w\cL_j\|_2 \ge c  2^{j}, \qquad j \goto \infty.
\]
\end{lemma}

\begin{proof}
We have
\[
\|w\mathcal{L}_j\|_2\ge\left(\int_{\xi_1\in\bR} |\hat{w}(\xi_1)|^2d\xi_1\int_{|\xi_2|\in[2^{2j-4}c_0,2^{2j-1}c_0] }d\xi_2\right)^{1/2}\approx c {2^{j}}.
\]
\end{proof} 


\subsection{Masks}

Inspired by the missing sensor scenario in seismic data we will define the mask of a missing piece of the
image as follows. The mask
$\cM_h$ is a vertical strip of diameter $2h$ and is given by
\[
\cM_h = \mathbbm{1}_{\{|x_1| \le h\}}.
\]
For an illustration, we refer to Figure \ref{fig:masks}.
\begin{figure}[h]
\centering
\includegraphics[height=1.3in]{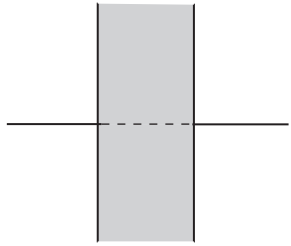}
\caption{Mask $\cM_h$ (gray shaded region), together with the linear singularity $w\cL$ (horizontal line with dashed center indicating part masked out).}
\label{fig:masks}
\end{figure}

For the convenience of the reader, we compute the associated Fourier transforms, where
as usual we set $\sinc(y)=\sin(\pi y)/(\pi y)$ for $y \in \bR$.

\begin{lemma}
\label{lemm:FouriertrafoM1}
We have
\[
\hat{\cM}_{h} = 2h\sinc(2h \xi_1) \hat{\cL}_y,
\]
where $\cL_y$ is the distribution acting as
\[
\ip{\cL_y}{f} = \int f(0,y) dy
\]
and $\ip{\hat\cL_y}{f}=\int f(x,0)dx$.
\end{lemma}

\begin{proof}
Define the planar Heaviside by $H(x) = \mathbbm{1}_{\{x_1 \ge 0\}}$. Since $\cL_y = \frac{\partial}{\partial x_1} H$,
we have $\hat{H}(\xi) = (2 \pi i \xi_1)^{-1} \hat{\cL}_y$. We now express $\cM_h$ in terms of $H$ by
\[
\cM_h = H(x+(h,0)) - H(x-(h,0)).
\]
This leads to
\[
\hat{\cM}_{h} = (e^{2 \pi ih\xi_1}-e^{-2 \pi ih\xi_1})(2 \pi i \xi_1)^{-1} \hat{\cL}_y= 2 \sin(2\pi h\xi_1)/(2 \pi \xi_1) \hat{\cL}_y
= 2 h\sinc(2 h \xi_1) \hat{\cL}_y.
\]
The proof is finished.
\end{proof} 


\subsection{Transfer of Abstract Setting}

All of the main proofs in Sections~\ref{sec:pos_wave} and \ref{sec:shear_pos} will follow a particular pattern.  Either
Proposition~\ref{prop:l1estimate} (in the case of $\ell_1$ minimization) or Proposition~\ref{prop:thresholdingestimate}
(in the case of thresholding) is applied to the situation in which $x^0$ is chosen to be the filtered linear singularity
$w\cL_j$, the Hilbert space $\cH_M$ is defined by $\{f  \cM_h : f \in L^2(\bR^2)\}$, and $\Phi$ is either the Parseval
system of Meyer wavelets or of shearlets at scale $j$.

In the analysis that follows, $\delta_j$ will denote the optimal $\delta$-clustered sparsity for filtered coefficients.  That is,
for $\ell_1$ minimization with a fixed filter level $j$, we will fix a set $\Lambda_j$ of significant coefficients of
$\Phi=\{\psi_\lambda\}_\lambda$ and set
\[
\delta_j=\sum_{\lambda\in\Lambda_j^c}|\ip{w\mathcal{L}_j}{\psi_\lambda}|.
\]
Similarly, we will analyze thresholding schemes by setting
\[
\delta_j=\sum_{\lambda\in\mathcal{T}_j^c}|\ip{w\mathcal{L}_j}{\psi_\lambda}|,
\]
where the $\mathcal{T}_j$ are the significant coefficients in {\sc One-Step-Thresholding} Algorithm.  The inpainting accomplished (i.e., the solution in Proposition~\ref{prop:l1estimate} or Proposition~\ref{prop:thresholdingestimate})
on the filtered levels $j$ will be denoted by $L_j$.  $w\cL_j$ will denote the filtered real image; that is, $w\cL \star \check{F_j}$,
where $w\cL$ is the original, complete image.  Thus, the main theorems in Sections~\ref{sec:pos_wave} and \ref{sec:shear_pos} will show that
\[
\frac{ \| L_j - w\cL_j \|_2}{\|w\cL_j\|_2 } \to 0, \qquad j \goto \infty.
\]
The results will specifically depend on the asymptotic behavior of the gap $h_j$.  For the proofs involving the Meyer system, the following notation will also be useful
\[
\widetilde{w\cL}_j = w\cL \star \widetilde{F}_j^\vee.
\]


\section{Positive Results for Wavelet Inpainting}
\label{sec:pos_wave}

We begin by proving theoretically for the first time what has been known heuristically; namely, that wavelets can successfully
inpaint an edge as long as not too much is missing.  In Section~\ref{subsec:l1Wav}, we investigate the inpainting results of $\ell_1$ minimization by estimating the
$\delta$-clustered sparsity $\delta_j$ and cluster coherence $\mu_c$ with respect to $\Phi=\{\psi_\lambda: \lambda=(\iota,j,k),
\iota=h,v,d; k\in\bZ^2\}$ and a proper chosen index set $\Lambda_j$. In Subsection~\ref{subsec:thrWav}, we similarly give
the estimation of $\delta_j$ and $\mu_c$ for inpainting using  thresholding.


\subsection{$\ell_1$ Minimization}
\label{subsec:l1Wav}
In what follows, we use the compact notation
$\ipa{a}:=(1+|a|^2)^{1/2}$.
 We first need to choose the set of significant coefficients appropriately. We do this
by setting
\[
\tilde{\Lambda}_j = \{(\iota;j,k) : |k_1| \le \rho  n_j  2^j, \, |k_2| \le n_j,\iota = h,v,d\},
\]
where $n_j = 2^{\epsilon j}$. This choice of $\Lambda_j = \tilde{\Lambda}_{2j} \cup \tilde{\Lambda}_{2j+1}$ forces the clustered sparsity to grow
slower than the growth rate of $\norm{w\cL_j}_2$:

\begin{lemma}\label{lemm:waveletl1delta}
$\delta_j=o(1)=o(\|w\cL_j\|_2), \quad j\rightarrow\infty$.
\end{lemma}
\begin{proof}
By definition, we have
\begin{eqnarray*}
\delta_j & = &\sum_{\lambda\in\Lambda_j^c}|\ip{w\mathcal{L}_j}{\psi_\lambda}| \leq \sum_{\lambda\in\Lambda_j^c}(|\ip{\widetilde{w\cL}_{2j}}{\psi_\lambda}| + |\ip{\widetilde{w\cL}_{2j+1}}{\psi_\lambda}| \\
& \leq &\sum_{\lambda\in\tilde{\Lambda}_{2j}^c}(|\ip{\widetilde{w\cL}_{2j}}{\psi_\lambda}| + \sum_{\lambda\in\tilde{\Lambda}_{2j+1}^c}|\ip{\widetilde{w\cL}_{2j+1}}{\psi_\lambda}| \\
& =: & \tilde{\delta}_{2j} + \tilde{\delta}_{2j+1}.
\end{eqnarray*}
We now compute
\[
\tilde{\delta}_j=\sum_{\lambda\in\tilde{\Lambda}_j^c}|\ip{\widetilde{w\mathcal{L}}_j}{\psi_\lambda}|=\sum_{\lambda\in\tilde{\Lambda}_j^c}
|\ip{\widehat{\widetilde{w\mathcal{L}}_j}}{\hat\psi_\lambda}|;
\]
that is,
\begin{eqnarray*}
\lefteqn{\tilde{\delta}_j = \sum_{\lambda\in\tilde{\Lambda}_j^c}\left|\int_{\bR^2} 2^{-j}\hat{w}(\xi_1)
{\widetilde {F_j}}(\xi)W^\iota(\xi/2^j)e^{-2\pi i\ip{k}{\xi/2^j}}d\xi\right|}
\\&:=&\sum_{\lambda\in\tilde{\Lambda}_j^c}\left|\int_{\bR^2} \hat{G}_j(\xi)e^{-2\pi i\ip{k}{\xi/2^j}}d\xi\right|,
\end{eqnarray*}
where $\hat{G_j}(\xi) = 2^{-j}\hat{w}(\xi_1)\widetilde F_j(\xi)W^\iota(\xi/2^j)$ is a smooth and compactly supported function that is essentially supported on
\[
[-1/\rho,1/\rho]\times [-2^jc_0,2^jc_0].
\]
Applying the change of variable $(\xi_1,\xi_2)\mapsto(\rho^{-1}\xi_1,2^j\xi_2)$
ensures that $\hat{G}_j(\rho^{-1}\xi_1,2^j\xi_2)$ is smooth and compactly supported independent of $j$. Then
\[
\left|\int_{\bR^2} \hat{G}_j(\xi)e^{-2\pi i\ip{k}{ \xi/2^j}}d\xi\right| \le \tilde{c}_N \|\hat{G}_j\|_\infty(\rho^{-1}2^j)\ipa{|(\rho^{-1}k_1/2^j,k_2)|}^{-N} \le c_N (\rho^{-1}2^j)\ipa{|(\rho^{-1}k_1/2^j,k_2)|}^{-N}.
\]
Consequently, $\tilde{\delta}_j/c_N$ is bounded above by
\begin{eqnarray*}
\lefteqn{\rho^{-1} 2^j\sum_{\lambda\in \Lambda_j^c}\ipa{|(\rho^{-1}k_1/2^j,k_2)|}^{-N}}\\
&\le&\rho^{-1} 2^j\Bigg(\sum_{|k_1|\ge \rho n_j 2^j,k_2}\ipa{|(\frac{\rho^{-1}k_1}{2^j},k_2)|}^{-N}+
\sum_{|k_1|\le \rho n_j 2^j, |k_2|\ge n_j}\ipa{|(\frac{\rho^{-1}k_1}{2^j},k_2)|}^{-N}\Bigg)\\
&\le&
\rho^{-1} 2^j\Bigg(\int_{\rho n_j 2^j}^\infty\int_{\bR}\ipa{|(\frac{\rho^{-1}x_1}{2^j},x_2)|}^{-N}dx_2dx_1 +\int_{0}^{\rho n_j 2^j}\int_{n_j}^\infty\ipa{|(\frac{\rho^{-1}x_1}{2^j},x_2)|}^{-N}dx_2dx_1\Bigg)\\
&\le&
\int_{ n_j }^\infty\int_{\bR}\ipa{|(x_1,x_2)|}^{-N}dx_2dx_1+
\int_{0}^{ n_j}\int_{n_j}^\infty\ipa{|(x_1,x_2)|}^{-N}dx_2dx_1\\
&\le &  2^{j\epsilon(1-N)}.
\end{eqnarray*}
Thus,
\[\delta_j \leq c 2^{2j\epsilon(1-N)} + 2^{(2j+1)\epsilon(1-N)}
\]
and for $N$ large enough, $\delta_j\rightarrow 0$ as $j\rightarrow \infty$.
\end{proof} 

On the other hand, the choice of $\Lambda_j$ offers low cluster coherence as well:

\begin{lemma}\label{lemm:waveletl1stripmuc}
For $h_j=o(2^{-2j})$ as $j\rightarrow \infty$, we have
\[
\mu_c(\Lambda_j,\{\cM_{h_j}\psi_\lambda\};\{\psi_\lambda\})\rightarrow0, \quad j\rightarrow\infty.
\]
\end{lemma}

\begin{proof}
We again first consider $\tilde{\Lambda}_j$.  By definition, we have
\[
\mu_c(\tilde{\Lambda}_j,\{\cM_{h_j}\psi_\lambda\};\{\psi_\lambda\})
=\max_{\lambda'}\sum_{\lambda\in\tilde{\Lambda}_j}\left|\ip{\cM_{h_j}\psi_{\lambda}}{\psi_{\lambda'}}\right| =
\max_{\lambda'}\sum_{\lambda\in\tilde{\Lambda}_j}\left|\ip{\hat \cM_{h_j}\star\hat\psi_{\lambda}}{\hat\psi_{\lambda'}}\right|.
\]
Note that for $\lambda = (\iota,j,k)$, we can choose $\lambda'=(\iota',j,0)$.
\begin{eqnarray*}
\ip{\hat \cM_{h_j}\star\hat\psi_{\lambda}}{\hat\psi_{\lambda'}}
&=&\int_{\bR^2}\int_{\bR^2}\hat \cM_{h_j}(\xi)\hat\psi_\lambda(\tau-\xi)d\xi\overline{\hat\psi_{\lambda'}(\tau)}d\tau\\
&=&\int_{\bR^2}\int_{\bR}2h_j\sinc(2h_j \xi_1)\hat\psi_{\lambda}(\tau-(\xi_1,0))d\xi_1\overline{\hat\psi_{\lambda'}(\tau)}d\tau\\
&=&\int_{\bR}2h_j\sinc(2h_j\xi_1) \Bigg[\int_{\bR^2}2^{-2j}W^{\iota}(\frac{\tau-(\xi_1,0)}{2^j}) e^{-2\pi i \ip{k}{\frac{\tau-(\xi_1,0)}{2^j}}} W^{\iota'}(\frac{\tau}{2^j})
d\tau\Bigg]d\xi_1\\
&=&2^{j}2h_j
\int_{\bR^2}
\Bigg[\int_{\bR}\sinc(2^j 2h_j \xi_1)W^{\iota}((\tau-(\xi_1,0)))e^{2\pi ik_1\xi_1}
d\xi_1 W^{\iota'}(\tau)\Bigg] e^{-2\pi i \ip{k}{ \tau}}d\tau\\
&=:&2^{j}2h_j
\int_{\bR^2}
\hat{g}_j(\tau) e^{-2\pi i \ip{k}{ \tau}}d\tau,
\end{eqnarray*}
where
\begin{equation}
\hat g_j(\tau) :=W^{\iota'}(\tau)\int_{\bR}\sinc(2^j 2h_j \xi_1)W^{\iota}((\tau-(\xi_1,0)))e^{2\pi i\ip{k_1}{\xi_1}} d\xi_1\label{eqn:waveg}
\end{equation}
is a smooth function supported on a box independent of $j$. Hence,
$\left|\int\hat g_j(\tau)e^{-2\pi ik \tau}d\tau\right|\le c_N\|\hat g_j\|_\infty\ipa{|k|}^{-N}$,
and
\[
\|\hat g_j\|_\infty \le c \sup_{\tau} \int |\sinc(2^j 2h_j\xi_1)||W^{\iota}(\tau-(\xi_1,0))|d\xi_1\le c \|\sinc(2^j 2h_j \cdot)\|_2\le c (2^jh)^{-1/2}.
\]

Consequently, we have
\[
\mu_c(\tilde{\Lambda}_j,\{\cM_{h_j}\psi_\lambda\};\{\psi_\lambda\})
\le c_N2^jh_j(2^{j}h_j)^{-1/2}\sum_{k\in \bZ^2}\ipa{|k|}^{-N}
\le c_N(2^{j}h_j)^{1/2},
\]
where
\[
\mu_c(\Lambda_j,\{\cM_{h_j}\psi_\lambda\};\{\psi_\lambda\})= \mu_c(\tilde{\Lambda}_{2j},\{\cM_{h_j}\psi_\lambda\};\{\psi_\lambda\})
+ \mu_c(\tilde{\Lambda}_{2j+1},\{\cM_{h_j}\psi_\lambda\};\{\psi_\lambda\}).
\]
which goes to $0$ as $j \rightarrow \infty$ by assumption.
\end{proof} 

We would like to remark at this point that we do not need the strong condition that
$h_j = o(2^{-2j})$ as $j \goto \infty$. In fact, carefully handling the constants
in the proof of Lemma~\ref{lemm:waveletl1stripmuc} will lead us to the condition
\[
\mu_c(\Lambda_{j},\{ \cM_{h_j}  \psi_\lambda \} ; \{ \psi_\lambda\}) \le c_N  (2^{2j}h_j)
\]
with precise knowledge of the value of $c_N$. Since ultimately, we ``only" need the
cluster coherence to boundedly stay a\-way from $1/2$, we only require the weaker
condition of
\[
2^{2j}h_j \le \frac{1}{2c_N}-\epsilon \quad \mbox{for some }\epsilon > 0 \mbox{ and for all } j \ge j_0.
\]
This condition would then be also sufficient for deriving the following theorem.

We now apply Proposition \ref{prop:l1estimate} to Lemmata \ref{lemm:NormofLj},
\ref{lemm:waveletl1delta}, and \ref{lemm:waveletl1stripmuc} to obtain
the desired convergence for the normalized $\ell_2$ error of the reconstruction
$L_j$ derived from~(\ref{eq:Inp}), where in this case $L=w\cL_j$ and $\Phi$ are wavelets $\psi_\lambda$ at scale $j$.

\begin{theorem}
\label{theo:waveletl1}
For $h_j=o(2^{-2j})$ and $L_j$ the solution to~(\ref{eq:Inp}) with $\Phi$ the 2D Meyer Parseval system,
\[
\frac{ \| L_j - w\cL_j \|_2}{\|w\cL_j\|_2 } \to 0, \qquad j \goto \infty.
\]
\end{theorem}

This result shows that if the size of the gap shrinks faster than
$2^{-2j}$ -- i.e., the size of the gap is asymptotically smaller than $2^{-2j}$ -- or if
the gap shrinks at the same rate than $2^{-2j}$ with an exactly prescribed factor,
we have asymptotically perfect inpainting.


\subsection{Thresholding}
\label{subsec:thrWav}

We will now study thresholding as an inpainting method, which is from a computational
point of view  much easier to apply than $\ell_1$ minimization. Our analysis will show that we can derive the same
asymptotic performance as for $\ell_1$ minimization.

Our first claim concerns the set of the thresholding coefficients $\cT_j$.

\begin{lemma}
\label{lemm:waveletthresholdstripthrescoeff}
For $h_j = o(2^{-2j})$ as $j \goto \infty$, there exist thresholds $\{\beta_j\}_j$ such that,
for all $j \ge j_0$,
\[
\{k : |k_1| \le \rho  2^{2j(1+n_1)}, \, |k_2| \le 2^{2j n_1}\} \subseteq
\cT_j
\]
for positive $j_0$ and $n_1$.
\end{lemma}
\begin{proof}
We again first analyze $\widetilde{w\cL}_j$. By Plancherel, we can rewrite the coefficients which we have to threshold
as follows:
\[ 
\absip{(1-\cM_{h_j}) \widetilde{w\cL}_j}{\psi_{\lambda}}
= |\ip{\delta_0 \star \widehat{\widetilde{w\cL}}_j}{\hat{\psi}_{\lambda}}-\ip{\hat{\cM}_{h_j}\star \widehat{\widetilde{w\cL}}_j}{\hat{\psi}_{\lambda}}|.
\]
Choose a function $F$ such that $F(\cdot/2^j)=\widetilde F_j$. Then,
\[
\widehat{\widetilde{w\cL}_j}(\xi)=\widehat{w\cL }(\xi){\widetilde{F}_j}(\xi) = \widehat{w\cL}(\xi){F}(\xi/2^j).
\]
As we are analyzing a horizontal line singularity, we only need to consider
\[
\hat\psi_{\lambda}=2^{-j}W^v(\xi/2^j)e^{-2\pi i\ip{k}{\xi/2^j}}
\]
for large
wavelet coefficients.  Then, the first term equals
\begin{eqnarray*}
\ip{\delta_0 \star \widehat{w\cL}_j}{\hat{\psi}_{\lambda}}
&=&   2^{-j} \int \hat{w}(\xi_1) \int (F W^v)(\xi/2^j) e^{-2\pi i\ip{k/2^j}{\xi}} d\xi\\
& = & \int \left[ \int \hat{w}(\xi_1) (F W^v)((\xi_1/2^j,\xi_2)) e^{-2\pi i \ip{k_1/2^j}{ \xi_1}}
d\xi_1 \right] e^{-2\pi i \ip{k_2}{ \xi_2}} d\xi_2.
\end{eqnarray*}
By using Lemma \ref{lemm:FouriertrafoM1}, we derive for the second term:
\begin{eqnarray*}
\lefteqn{\ip{\hat{\cM}_{h_j}\star \widehat{w\cL}_j}{\hat{\psi}_{\lambda}}
 =  2h_j \int \sinc(2h_j \tau_1) \int \hat{w}(\xi_1) (\hat{\psi}_{\lambda}  F_j)((\tau_1,0)+(\xi_1,\xi_2))
d\xi d\tau_1}\\
&\hspace{-3mm} = \hspace{-3mm} & 2h_j  2^{-j} \int \sinc(2h_j \tau_1)
\int \hat{w}(\xi_1) F(\xi_1/2^j,\xi_2/2^j) W^v((\tau_1+\xi_1)/2^j,\xi_2/2^j) e^{-2\pi i \ip{k/2^j}{\tau_1+\xi_1,\xi_2}} d\xi d\tau_1\\
& \hspace{-3mm} =\hspace{-3mm}  & 2 h_j \int \Bigg[\int \hat{w}(\xi_1) \int \sinc((h_j/\pi) \tau_1) F(\xi_1/2^j,\xi_2) W^v(((\tau_1+\xi_1)/2^j,\xi_2))
e^{-2\pi i \ip{k_1}{(\tau_1+\xi_1)/2^j}} d\tau_1 d\xi_1 \Bigg] e^{-2\pi i \ip{k_2}{ \xi_2}} d\xi_2.
\end{eqnarray*}
Let $\hat{G}$ now be the function
\begin{eqnarray*}
\hat{G}(\xi_2) & = & \int \hat{w}(\xi_1) \Bigg[(F W^v)((\xi_1/2^j,\xi_2))+\\
& & - 2h_j\int \sinc((h_j/\pi) \tau_1) F(\xi_1/2^j,\xi_2) W^v((\tau_1+\xi_1)/2^j,\xi_2) e^{-2\pi i(k_1/2^j)\tau_1} d\tau_1 \Bigg]  e^{-2\pi i(k_1/2^j)\xi_1} d\xi_1\\
&=&\int \hat{w}(\xi_1) \hat{H}_{\xi_2}(\xi_1) e^{-2\pi i(k_1/2^j) \xi_1} d\xi_1
\end{eqnarray*}
with
\[
\hat{H}_{\xi_2}(\xi_1)  = (F W^v)((\xi_1/2^j,\xi_2)|)- 2 h_j\int \sinc((h_j/\pi) \tau_1) F(\xi_1/2^j,\xi_2) W^v((\tau_1+\xi_1)/2^j,\xi_2)
e^{-2\pi i \ip{(k_1/2^j)}{ \tau_1}} d\tau_1.
\]
The function $\hat G$ is supported on the set $[1/2,2]$, which is independent of $j$.
By standard arguments, we can deduce that
\beq \label{eq:waveletthresholdstripthrescoeff3}
\absip{(1-\cM_{h_j}) \widetilde{w\cL}_j}{\psi_{\lambda}}
\le c_{N_1}  \norm{\hat{G}}_\infty  \ang{|k_2|}^{-N_1}.
\eeq
Let us now investigate the term $\norm{\hat{G}}_\infty$ further.
Using Plan\-cherel and the support properties of $w$,
\begin{eqnarray*}
\Big|\int \hat{w}(\xi_1) \hat{H}_{\xi_2}(\xi_1) e^{-2\pi i \ip{k_1/2^j}{ \xi_1}} d\xi_1\Big|
&=&  |(\hat{w}  \hat{H}_{\xi_2})^\vee (-k_1/2^j)|
=  |(w \star H_{\xi_2}) (-k_1/2^j)|\\
& =&  \Big|\int w(-k_1/2^j -x) H_{\xi_2}(x) dx\Big|
 \approx  c  \Big|\int_{-k_1/2^j  - \rho}^{-k_1/2^j + \rho} H_{\xi_2}(x) dx\Big|.
\end{eqnarray*}
For the analysis of the function $H_{\xi_2}$, we use well-known properties of the Fourier transform to derive
\begin{eqnarray*}
H_{\xi_2}(x)  &= & \left((F W^v)(\cdot/2^j,\xi_2)\right)^\vee(x) - \hspace{-1mm} \left(\!\!(\!2h_j \sinc(2h_j\cdot)e^{-2\pi i(k_1/2^{j})\cdot})
\star ((F W^v)(\cdot/2^j,\xi_2)\!)\!\!\right)^{\hspace{-1mm}\vee}\hspace{-1mm}(\!-x\!)\\[1ex]
& = & \left((F W^v)(\cdot/2^j,\xi_2)\right)^\vee(x) -\left(2h_j\sinc((2h_j\cdot) )e^{-2\pi i(k_1/2^{j})\cdot}\right)^\vee(-x) \left((F W^v)(\cdot/2^j,\xi_2)\right)^\vee(-x)\\[1ex]
& = & \left((F W^v)(\cdot/2^j,\xi_2)\right)^\vee(x) -\mathbbm{1}_{[-h_j,h_j]}(x-k_1/2^{j})  \left((F W^v)(\cdot/2^j,\xi_2)\right)^\vee(-x).
\end{eqnarray*}
Hence, since $h_j < \rho$,
\begin{eqnarray} 
\nonumber c  \Big|\int_{-k_1/2^j-\rho}^{-k_1/2^j+\rho} H_{\xi_2}(x) dx\Big| 
& = & c  \Big|\int_{k_1/2^j-\rho}^{k_1/2^j+\rho} ((F W^v)(\cdot/2^j,\xi_2))^\vee(x)- \int_{k_1/2^j-h_j}^{k_1/2^j+h_j} ((F W^v)(\cdot/2^j,\xi_2)^\vee(x) dx\Big|\\ \label{eq:wavethres1}
& = & c  \Big|\int_{k_1-2^j\rho}^{k_1-2^jh_j} + \int_{k_1+2^jh_j}^{k_1+2^j\rho} ((F W^v)(|(\cdot,\xi_2)|))^\vee(x) dx\Big|.
\end{eqnarray}
Notice that the bounds of integration indeed make sense, since the values of $k_1$ which lie ``in between $h_j$ and $\rho$'' should play
an essential role.  Due to the regularity of  $W$, there exist some $N_2$ and
$c$ (possibly differing from the one before, but we do not need to distinguish constants here) such that
\[
|((F W^v)(|(\cdot,\xi_2)|)^\vee(x)| \le c  \ang{|x|}^{-N_2},
\]
and hence by~(\ref{eq:wavethres1}),
\beq \label{eq:waveletthresholdstripthrescoeff5}
\norm{\hat{G}}_\infty \le c  \ang{\min\{|k_1 - 2^j\rho|,|k_1 + 2^j\rho|\}}^{-N_2}.
\eeq

Finally, we have to study how the function $\hat{H}$ relates to $h$, which will show the behavior of the
coefficients as they approach the center of the mask. For this,
setting
\[
\hat{J}_{\xi_2}(\tau_1) = (F W^v)((\tau_1+\xi_1)/2^j,\xi_2) e^{-2\pi i\ip{k_1/2^j}{ \tau_1}},
\]
we obtain
\begin{eqnarray*}
\lefteqn{|(F W^v)(\xi_1/2^j,\xi_2)- 2 h_j\int \sinc((h_j/\pi) \tau_1) (F W^v)((\tau_1+\xi_1)/2^j,\xi_2)
e^{-2\pi i \ip{k_1/2^j}{\tau_1}} d\tau_1|}\\
& = & |\hat{J}_{\xi_2}(0)-2 h_j\int \sinc((h_j/\pi) \tau_1)\hat{J}_{\xi_2}(\tau_1) d\tau_1|\\
& = & |\hat{J}_{\xi_2}(0)-\int \hat{1}_{[-h_j,h_j]}(\tau_1) \hat{J}_{\xi_2}(\tau_1) d\tau_1|\\
& = & |\hat{J}_{\xi_2}(0)-\int_{-h_j}^{h_j} J_{\xi_2}(x) dx| =  |\int_{|x|>h_j} J_{\xi_2}(x) dx|.
\end{eqnarray*}
Hence another way to estimate $\|\hat G\|_\infty$ is by
\begin{eqnarray*}
\frac{\norm{\hat{G}}_\infty}{c}& \le&   \norm{\hat{H}_{\xi_2}}_\infty\\
& \le &\max_{\xi_1,\xi_2}\Bigg|\int_{|x|>h_j} ((F W^v)((\cdot+\xi_1)/2^j,\xi_2))^\vee(x-k_1/2^j) dx\Bigg|\\
& \le &\max_{\xi_2} \Bigg|\int_{|x|>2^jh_j} ((F W^v)((\cdot,\xi_2))^\vee(x-k_1) dx\Bigg|.
\end{eqnarray*}
Certainly, the minimum is attained in the center of the mask, i.e., with $k=0$. So
combining this with~(\ref{eq:waveletthresholdstripthrescoeff3}) and~(\ref{eq:waveletthresholdstripthrescoeff5}),
\[
\absip{(1-\cM_{h_j}) \widetilde{w\cL}_j}{\psi_{\lambda}}
\le  c  \max_{\xi_2}\Bigg|\int_{|x|>2^jh_j} \hspace*{-0.2cm}((F W^v)((\cdot,\xi_2)))^\vee(x) dx\Bigg|
 \ang{|k_2|}^{-N_1}\ang{\min\{|k_1 - 2^j\rho|,|k_1 + 2^j\rho|\}}^{-N_2}.
\]
which is what we intend to use as a ``model.'' Observe that this indeed is also intuitively the right
estimate, since the $k_2$ component has to decay rapidly away from zero,
thereby sensing the singularity in zero in this direction. In contrast, the $k_1$ component
stays greater or equal to $\ang{2^j\rho}^{-N_2}$ up to the point $2\rho2^j$ and then
decays rapidly in accordance with the fact that up to the point $k_1=\rho2^j$ we
are ``on'' the line singularity which decays smoothly with $\hat{w}$. Moreover, the first
term models the behavior in the mask, which is also nicely supported by the fact that
the crucial product $2^{2j}h_j$ is appearing therein.

We now apply the triangle inequality
\[
\absip{(1-\cM_{h_j}) w\cL_j}{\psi_{\lambda}}
 \leq  \absip{(1-\cM_{h_j}) \widetilde{w\cL}_{2j}}{\psi_{\lambda}} + \absip{(1-\cM_{h_j}) \widetilde{w\cL}_{2j+1}}{\psi_{\lambda}}.
\]
Since $2^{2j} h_j \to 0$ and $2^{2j+1} h_j \to 0$ as $j \to \infty$, we have as $j \to \infty$
\[
\max_{\xi_2}\Bigg|\int_{|x|>2^{2j}h_j} ((F W^v)((\cdot,\xi_2)))^\vee(x)dx\Bigg|
\to   C.
\]
We now set the thresholds $\beta_j$ to be
\[
\frac{c (C-\epsilon)}{ \ang{|2^{2j \epsilon}|}^{N_1}\ang{\min\{|(2^{2j \epsilon}-1)2^{2j}\rho|,|(2^{2j \epsilon}+1)2^{2j}\rho|\}}^{N_2}}.
\]
This choice immediately proves the claim of the lemma.
\end{proof} 

Note that  $\Lambda_j\subseteq\{k : |k_1| \le \rho  2^{2j(1+n_1)}, \, |k_2| \le 2^{2j n_1}\}\subseteq\cT_j$ for some $n_1>0$.
For such a choice of $\cT_j$, we have the following lemma.

\begin{lemma}
\label{lemm:waveletthresholddelta}
\[
\delta_j = \sum_{k \in \cT_j^c} \absip{w\cL_j}{\psi_{\lambda}} = o(\|w\cL_j\|_2)\textrm{, $j \goto \infty$}.
\]
\end{lemma}

\begin{proof}
We observe from the proof of Lemma \ref{lemm:waveletl1delta} that the desired property is
automatically satisfied provided that, for all $j \ge j_0$, the set $\cT_j$ satisfies
\[
\cT_j \supseteq \{k : |k_1| \le \rho  2^{2j(1+\nu_1)}, \, |k_2| \le 2^{2j\nu_1}\}\supseteq\Lambda_j,
\]
for some $\nu_1 > 0$, which is implied by Lemma \ref{lemm:waveletthresholdstripthrescoeff}.
\end{proof} 

We next analyze the second term in the estimate from Prop\-osition \ref{prop:thresholdingestimate}.

\begin{lemma}
\label{lemm:waveletthresholdstriplastterm}
For $h_j = o(2^{-2j})$ as $j \goto \infty$,
\[
 \sum_{k \in \cT_j} \absip{\cM_{h_j}  w\cL_j}{\psi_{\lambda}} = o(2^{j/2}), \qquad j \goto \infty.
\]
\end{lemma}
\begin{proof}
We first need to derive some estimates dependent on $k$ for the term
$\absip{\cM_{h_j}  \widetilde{w\cL}_j}{\psi_{\lambda}}$. By using the definitions of $\cM_{h_j}$
and $\widetilde{w\cL}_j$ and a change of variables, we first obtain $\ip{\cM_{h_j}  \widetilde{w\cL}_j}{\psi_{\lambda}}=$
\[
2 h_j \int \Bigg[\int \hat{w}(\xi_1) \int \sinc((2h_j) \tau_1) F(\xi_1/2^j,\xi_2) W^v(((\tau_1+\xi_1)/2^j,\xi_2))
e^{-2\pi i\ip{k_1}{(\tau_1+\xi_1)/2^j}} d\tau_1 d\xi_1 \Bigg]  e^{-2\pi i\ip{k_2}{ \xi_2}} d\xi_2.
\]
Here $F(\cdot/2^j)=\widetilde F$. Let $\hat{G}$ now be the function
\begin{eqnarray*}
\hat{G}(\xi_2)& =&  \int \hat{w}(\xi_1)  2 h_j \int \sinc((h_j/\pi) \tau_1) F(\xi_1/2^j,\xi_2) W^v((\tau_1+\xi_1)/2^j,\xi_2)
e^{-2\pi i \ip{k_1/2^j}{\tau_1+\xi_1}} d\tau_1 d\xi_1\\
&= &\int \hat{w}(\xi_1) \hat{H}_{\xi_2}(\xi_1) e^{-2\pi i \ip{k_1/2^j}{ \xi_1}} d\xi_1.
\end{eqnarray*}
with
\[\hat{H}_{\xi_2}(\xi_1) = 2h_j \int \sinc((h_j/\pi) \tau_1) F(\xi_1/2^j,\xi_2) W^v((\tau_1+\xi_1)/2^j,\xi_2)
e^{-2\pi i\ip{k_1/2^j}{\tau_1}} d\tau_1.
\]
The function $\hat G$ is supported on the set $[-1/4,-1/16]\cup[1/16,1/4]$, which is independent of $j$. Hence, we have
\beq \label{eq:wavethres_eq1}
\absip{\cM_{h_j}  w\cL_j}{\psi_{\lambda}}
\le c_{N_1}  \norm{\hat{G}}_\infty  \ang{|k_2|}^{-N_1}.
\eeq
By Plancherel's theorem and the support properties of $w$,
\begin{eqnarray*}
\Big|\int \hat{w}(\xi_1) \hat{H}_{\xi_2}(\xi_1) e^{-2\pi i \ip{k_1/2^j}{\xi_1}} d\xi_1\Big|
& = & |(\hat{w}  \hat{H}_{\xi_2})^\vee (-k_1/2^j)| =  |(w \star H_{\xi_2}) (-k_1/2^j)|\\
& = & \Big|\int w(-k_1/2^j -x) H_{\xi_2}(x) dx\Big| \approx   c  \Big|\int_{-k_1/2^j-\rho}^{-k_1/2^j+\rho} H_{\xi_2}(x) dx\Big|.
\end{eqnarray*}
Next, using well-known properties of the Fourier transform, we can manipulate $H_{\xi_2}$:
\begin{eqnarray*}
H_{\xi_2}(x) &  =&  \left((2h_j\sinc(2h_j \cdot )e^{-2\pi ik_1/2^{j}\cdot}) \star (F W^v(\cdot/2^j,\xi_2))\right)^{\hspace{-1mm}\vee}\hspace{-1mm}(-x)\\
& = & \left(2h_j\sinc(2h_j \cdot )e^{-2\pi ik_1/2^{j}\cdot}\right)^\vee(-x)  \left((F W^v)(\cdot/2^j,\xi_2)\right)^\vee(-x)\\
& = & \mathbbm{1}_{[-h_j,h_j]}(-x-k_1/2^{j})  \left((F W^v)(\cdot/2^j,\xi_2)|)\right)^\vee(-x).
\end{eqnarray*}
Hence, since $h_j < \rho$,
\begin{eqnarray*}
c  \Big|\int_{-k_1/2^j -\rho}^{-k_1/2^j +\rho} H_{\xi_2}(x) dx\Big|
& = & c  \Big|\int_{k_1/2^j -h_j}^{k_1/2^j +h_j} ((F W^v)(\cdot/2^j,\xi_2))^\vee(x) dx\Big|\\
& = & c  \Big|\int_{k_1-2^jh_j}^{k_1+2^jh_j} ((F W^v)((\cdot,\xi_2)))^\vee(x) dx\Big|.
\end{eqnarray*}
Notice that this indeed makes sense, since due to the masking the length of the
line singularity isn't allowed to play a role here. Due to the regularity of $W$, there exists some
constants $N_2$ and $c$ such that
\[
|((F W^v)(|(\cdot,\cdot)|)^\vee(x)| \le c  \ang{|x|}^{-N_2}.
\]
Hence,
\[
\norm{\hat{G}}_\infty \le c  \ang{\min\{|k_1 - 2^jh_j|,|k_1 + 2^jh_j|\}}^{-N_2}.
\]
Combining this estimate with~(\ref{eq:wavethres_eq1}), we obtain
\[
\absip{\cM_{h_j}  \widetilde{w\cL}_j}{\psi_{\lambda}}
\le c  \ang{|k_2|}^{-N_1}  \ang{\min\{|k_1 - 2^jh_j|,|k_1 + 2^jh_j|\}}^{-N_2},
\]
which is what we intend to use.

Finally,
\begin{eqnarray*}
\sum_{k \in \cT_j} \absip{\cM_{h_j}  w\cL_j}{\psi_{\lambda}} &\le&  c \Bigg( \sum_{k \in \cT_j} \ang{|k_2|}^{-N_1}  \ang{\min\{|k_1 - 2^{2j}h_j|,|k_1 + 2^{2j}h_j|\}}^{-N_2} +\\
&&+ \sum_{k \in \cT_j} \ang{|k_2|}^{-N_1}  \ang{\min\{|k_1 - 2^{2j+1}h_j|,|k_1 + 2^{2j+1}h_j|\}}^{-N_2} \Bigg)\\
& \le &  c. 
\end{eqnarray*}
\end{proof} 

Notice that this result holds for any $\cT_j$, which again is intuitively clear since
if it holds for the claimed on, then extending the set $\cT_j$ does not change the
estimate due to the fact that $\cM_{h_j}  w\cL_j$ is zero ``outside."

We now apply Proposition \ref{prop:thresholdingestimate} to Lemmata \ref{lemm:NormofLj},
\ref{lemm:waveletthresholddelta}, and \ref{lemm:waveletthresholdstriplastterm} to obtain
the desired convergence for the normalized $\ell_2$ error of the reconstruction
$L_j$ from {\sc One-Step-Thresholding} in Figure \ref{fig:onestepthresholding}.
Again, in this case $x=w\cL_j$ and $\Phi$ are wavelets $\psi_\lambda$ at scale $j$.
\begin{theorem}
\label{theo:waveletthreshold}
For $h_j=o(2^{-2j})$ and $L_j$ the solution to~(\ref{eq:Thresh}) with $\Phi$ the 2D Meyer Parseval system,
\[
\frac{ \| L_j - w\cL_j \|_2}{\|w\cL_j\|_2 } \to 0, \qquad j \goto \infty.
\]
\end{theorem}

This result shows that {\sc One-Step-Thresholding} fills in gaps of the same size as $\ell_1$
minimization (\mbox{\sc Inp}) in an asymptotic sense when considering the $\ell_2$ error.


\section{Shearlet Inpainting Positive Results}
\label{sec:shear_pos}
In this section, $\Phi$ is the shearlet frame as in (\ref{eqn:shearsys}) in Subsection~\ref{sec:shear_def}.  The general approach in
this section is the same as in the preceding section.  We show that use of the analysis coefficients of the
shearlet system through either $\ell_1$ minimization or thresholding will successfully inpaint a line across
a missing strip.  Namely, in Subsection~\ref{subsec:l1Shx}, we investigate the inpainting results of $\ell_1$ minimization by estimating the
$\delta$-clustered sparsity $\delta_j$ and cluster coherence $\mu_c$ with respect to $\{\sigma_\eta: \eta = (\iota, j,
\ell,k), \iota\in\{h,v,\emptyset\};
|\ell|\le 2^{j}; k\in\bZ^2\}$ and a properly chosen index set $\Lambda_j$. In Subsection~\ref{subsec:thrShx}, we
similarly give the estimation of $\delta_j$ and $\mu_c$ for inpainting using thresholding.  Some of the proofs
in this section are very similar in spirit to the corresponding ones in Section~\ref{sec:pos_wave} but decidedly
more technical due to the structural difference between wavelets and shearlets.  The auxiliary functions
(\ref{eqn:waveg}) and (\ref{eqn:shearg}) in the proofs of Lemma~\ref{lemm:waveletl1stripmuc} and
Theorem~\ref{lemm:shearletl1stripmuc} demonstrate this relationship quite well.


\subsection{$\ell_1$ Minimization}
\label{subsec:l1Shx}
For our analysis we choose the set of significant shearlet coefficients to be
\[
\Lambda_j = \{(\iota;j,k,\ell) : |k_1| \leq  \rho  n_j  2^{j},\, |k_2| \le n_j, \,\ell = 0; \iota=v\}
\]
where we revive the notion $n_j=2^{\epsilon 2 j}$ from the previous subsection.

Now we can show the clustered sparsity of the shearlet coefficients with the choice of $\Lambda_j$.
\begin{lemma}
\label{lemm:shearletl1delta}
For $\eps < 1/4$,
\[
\delta_j  = o(2^{j}), \qquad  j \goto \infty .
\]
\end{lemma}
\begin{proof}
By the definition, we have
\begin{eqnarray*}
\delta_j&=&\!\sum_{|k_1|\ge \rho n_j 2^{j},|k_2|\le n_j,\ell=0}|\ip{w\cL_j}{\sigma_{j,\ell,k}^v}|+\!\sum_{|k_2|\ge  n_j,\ell=0}|\ip{w\cL_j}{\sigma_{j,\ell,k}^v}|
+\!\sum_{k\in\bZ^2,\ell\neq 0}|\ip{w\cL_j}{\sigma_{j,\ell,k}^v}|+\\
&&+
\!\sum_{k\in\bZ^2,\ell}|\ip{w\cL_j}{\sigma_{j,\ell,k}^h}| + \sum_{k\in\bZ^2} \absip{w\cL_j}{\sigma_{j,\pm 2^j,k}}
\\&=:&T_1+T_2+T_3+T_4+T_5.
\end{eqnarray*}
To estimate $T_1$, we first estimate $\ip{w\cL}{\sigma_{\eta}}$ for the case $\ell = 0$ and $\iota = v$. By Lemma~\ref{Lemm:decaySlope0} in Section~\ref{sec:aux_shear},
\begin{eqnarray*}
\ip{w\cL}{\sigma_{j,k,0}^v}
&\le& c_N {a_j^{-1/2}} \ipa{|k_2|}^{-1}\ipa{[k_2^2+a_j^{-4}\min_{\pm}(a_jk_1\pm\rho)^2]^{1/2}}^{2-N}\\
&\le& c_N{a_j^{-1/2}}\ipa{|k_2|}^{-1} \ipa{[k_2^2+\min_{\pm}(a_j^{-1}k_1\pm a_j^{-2}\rho)^2]^{1/2}}^{2-N}\\
&\le&
c_N {a_j^{-1/2}}\ipa{|k_2|}^{-1} \ipa{a_j^{-2}\min_{\pm}|a_jk_1\pm\rho|}^{2-N},
\end{eqnarray*}
Therefore,  we have
\begin{eqnarray*}
T_1&\le& c_N a_j^{-1/2} a_j^{-\epsilon} \sum_{|k_1|\ge \rho a_j^{-1-2\epsilon}}
\ipa{a_j^{-2}\min_{\pm}|a_j k_1\pm\rho)|}^{2-N}
\\&\le&
 c_N a_j^{-1/2}a_j^{-2\epsilon} \sum_{|k_1|\ge \rho a_j^{-1-2\epsilon}}
\ipa{\min_{\pm}|a_j^{-1}k_1\pm a_j^{-2}\rho)|}^{2-N}.
\end{eqnarray*}
Note that $a_j^{-2\epsilon}=n_j=2^{2j\epsilon}$.
Since
\begin{eqnarray*}
\int_{|x|>\rho a_j^{-1-2\epsilon}}\ipa{|a_j^{-1}x- a_j^{-2}\rho)|}^{2-N}dx
&=&a_j \int_{|y|>\rho a_j^{-2-2\epsilon}}\ipa{|y-a_j^{-2}\rho|}^{2-N}dy
\\&
\le&
a_j\int_{|y|>\rho a_j^{-2}}\ipa{|y|}^{2-N}dy
\le c_N a_j^{1+2(N-3)},
\end{eqnarray*}
we obtain
\[
T_1\le c_N a_j^{1/2-2\epsilon+2(N-3)}.
\]

For $T_2$, we have
\begin{eqnarray*}
\frac{T_{2}}{c_N a_j^{-1/2}}
& \le & \sum_{k_1\in\bZ, |k_2|\ge a_j^{-2\epsilon}}\ipa{[k_2^2+\min_{\pm}(a_j^{-1}k_1\pm a_j^{-2}\rho)^2]^{1/2}}^{2-N}\\
&\le&
\sum_{|k_1|\le \rho a_j^{-1-2\epsilon}, |k_2|\ge a_j^{-2\epsilon}}\ipa{[k_2^2+\min_{\pm}(a_j^{-1}k_1\pm a_j^{-2}\rho)^2]^{1/2}}^{2-N}\\
&&+\sum_{|k_1|>\rho a_j^{-1-2\epsilon}, |k_2|\ge a_j^{-2\epsilon}}\ipa{[k_2^2+\min_{\pm}(a_j^{-1}k_1\pm a_j^{-2}\rho)^2]^{1/2}}^{2-N}
\\&=:&T_{2,1}+T_{2,2}.
\end{eqnarray*}

For $T_{2,1}$, we have
\begin{eqnarray*}
T_{2,1}&\le& c\int_{|x_1|<\rho a_j^{-1-2\epsilon}}\int_{|x_2|>a_j^{-2\epsilon}}\ipa{|x_2|}^{2-N}dx_2dx_1\\
&\le& c a_j^{-1+2(N-4)\epsilon}.
\end{eqnarray*}

For $T_{2,2}$, we have
\begin{eqnarray*}
T_{2,2}&\le& c a_j\int_{x_1>\rho a_j^{-2-2\epsilon}}\int_{x_2>a_j^{-2\epsilon}}\ipa{|(x_1,x_2)|}^{2-N}dx_2dx_1\\
&\le& c a_j^{2(N-3)(1+2\epsilon)}.
\end{eqnarray*}
Therefore,
\[
T_2\le c_N a_j^{-3/2+2(N-1)\epsilon}.
\]

For $T_3$, we convert the result in Lemma~\ref{lemm:decaySlope1} in Section~\ref{sec:aux_shear} to the discrete case.

\begin{lemma} Let $t_1=a^2_j(k_1-\ell k_2)$ and $t_2=a_j k_2$ with $a_j=2^{-j}$.
\begin{enumerate}
\item[{\rm(i)}] For $t_1\neq 0$ and $t_2\neq 0$,  we have
\[|\ip{w\cL_j}{\sigma_{j,\ell,k}^h}|\le c_{N} e^{-c a_j^{-1}} a_j^{-1/2}{|a^2_j(k_1-\ell k_2)|}^{-N}{|a_j k_2|}^{-N} a_j^{N},
\]
and
\[|\ip{w\cL_j}{\sigma_{j,\ell,k}^v}|\le c_{N} e^{-c a_j^{-2}} a_j^{-1/2}{|a_j(k_1-\ell k_2)|}^{-N}{|a_j k_2|}^{-N}  a_j^{2N}.
\]

\item[{\rm(ii)}] When exactly one of $t_1$ or $t_2$ is $0$ and $\iota\in\{ h,v\}$, we have
\[|\ip{w\cL}{\sigma_{j,\ell,k}^\iota}|
\le c_{L}\left[\max\{a^2_j|k_1-\ell k_2|,a_j |k_2|\}\right]^{-L} a_j^{-1/2}  e^{-c a_j^{-1}\ell}.
\]
\item[{\rm(iii)}] For $t_1=t_2= 0$ and $\iota\in\{ h,v\}$, we have
\[|\ip{w\cL}{\sigma_{j,\ell,k}^\iota}|\le c a_j^{-1/2}  e^{-c a_j^{-1}}.
\]
\end{enumerate}
\end{lemma}
For $t_1:=a^2_j(k_1-\ell k_2)\neq 0$ and $t_2:=a_j k_2\neq 0$, we have
\begin{eqnarray*}
a_j^{3}\sum_{k\in\bZ^2,t_1\neq 0,t_2\neq 0}{|a^2_j(k_1-\ell k_2)|}^{-N}{|a_j k_2|}^{-N}
&\le& a_j^{3}\int_{\{x\;:\; x_1\neq\ell x_2, x_2\neq 0\}}{|a^2_j(x_1-\ell x_2)|}^{-N}{|a_j x_2|}^{-N}dx_1dx_2
\\&< & c\cdot
\int_{|x_1|\ge 1,|x_2|\ge 1}{|x_1|}^{-N}{|x_2|}^{-N}dx_1dx_2
\\&<& \infty.
\end{eqnarray*}
Hence
\[
\sum_{k\in\bZ^2,t_1\neq 0,t_2\neq 0}{|a^2_j(k_1-\ell k_2)|}^{-N}{|a_j k_2|}^{-N}<ca_j^{-3}.
\]
Similarly, for $t_1=0$ or $t_2=0$, we have
\[\sum_{k\in\bZ^2,t_1=0\mbox{ or }t_2=0}\left[\max\{a^2_j|k_1-\ell k_2|,a_j |k_2|\}\right]^{-N}< c a_j^{-3}.
\]
The estimate for (iii) follows by direct computation.  Therefore, by the above estimates (i),  (ii), and (iii), and that
\[
T_3= \sum_{\ell=1}^{a_j^{-1}}\sum_{k\in\bZ^2,(t_1,t_2)\neq0}|\ip{w\cL_j}{\sigma_{j,\ell,k}^v}|+
\sum_{\ell=1}^{a_j^{-1}}|\ip{w\cL_j}{\sigma_{j,\ell,0}^v}|,
\]
we obtain
\[
T_3\le\sum_{\ell=1}^{a_j^{-1}} c_N a_j^{-1/2} e^{-c a_j^{-1}}(a_j^{-3}+1)
\le c_N a_j^{N}\quad\forall N\ge0.
\]
Similarly, for $T_4$,
\[
T_4
\le\sum_{\ell=1}^{a_j^{-1}} c_N a_j^{-1/2}e^{-c a^{-1}}(a_j^{-3} +1)
\le c_N a_j^{N}\quad\forall N\ge0.
\]
Finally, since the ``seam" elements $\sigma_{j,\ell,k}$ are only slight modifications of the $\sigma_{j,\ell,k}^\iota$, $T_5 \leq c_Na_j^N$ for all $N\geq0$.

Combining the estimates for $T_1,\ldots, T_5$, we are done.
\end{proof}

Next we estimate the cluster coherence
\[
\mu_c(\Lambda_{j},\{ \cM_{h_j}  \sigma_\eta \} ; \{ \sigma_\eta\})
\]
and show that it converges to zero as $j \to \infty$ when $h_j$ is related $j$ by $h_j = o(2^{-j})$ as $j \goto \infty$.
We wish to remark that the size of the gaps which can be filled with
asymptotically high precision is dramatically larger than the corresponding size for
wavelet inpainting.

\begin{theorem}\label{lemm:shearletl1stripmuc}
For $h_j = o(2^{-j})$
\[
\mu_c(\Lambda_j, \{\cM_{h_j}\sigma_\eta\}; \{\sigma_{\eta}\})\rightarrow0, \quad j\rightarrow \infty
\]
with $\eta=(\iota, j,\ell,k)$ and $\iota\in\{h,v,\emptyset\}$.
\end{theorem}
\begin{proof}
We have
\begin{eqnarray*}
\lefteqn{\mu_c(\Lambda_j, \{\cM_{h_j}\sigma_{\eta}\};
 \{\sigma_{\eta}\})
 =\max_{\eta_2}\sum_{\eta_1\in \Lambda_j}|\ip{{\cM_{h_j}\sigma_{\eta_1}}}{\sigma_{\eta_2}}|}
\\&\le&
\max_{\eta_2, \iota=v}\sum_{\eta_1\in \Lambda_j}|\ip{{\cM_{h_j}\sigma_{\eta_1}}}{\sigma_{\eta_2}}| + 
\max_{\eta_2, \iota = h}\sum_{\eta_1\in \Lambda_j}|\ip{{\cM_{h_j}\sigma_{\eta_1}}}{\sigma_{\eta_2}}|+
\max_{\eta_2, \iota = \emptyset}\sum_{\eta_1\in \Lambda_j}|\ip{{\cM_{h_j}\sigma_{\eta_1}}}{\sigma_{\eta_2}}|
\\&=:&T_1+T_2 + T_3.
\end{eqnarray*}

We bound $T_1$ using simple substitutions:
\begin{eqnarray*}
T_1&\le& \sum_{(\iota;j,\ell,k)\in\Lambda_j}
|\ip{{\cM_{h_j}\sigma_{j,\ell,k}^{v}}}{\sigma_{j,0,0}^{v}}|
\\&\le&
\sum_{(\iota;j,\ell,k)\in\Lambda_j}
\bigg\vert\int_{\bR}2h_j\sinc(2h_j\xi_1)\Bigg[\int_{\bR^2}2^{-3j}\cW(\frac{\tau}{2^{2j}})   \cW\left(\frac{(\tau_1-\xi_1,\tau_2)}{2^{2j}}\right)V(\ell+2^{j}\frac{\tau_1-\xi_1}{\tau_2})V(2^{j}\frac{\tau_1}{\tau_2})\times\\
&& \times
 e^{-2\pi i\ip{t}{(\tau-(\xi_1,0))A_{2^{-j}}^v S_{\ell}^v}}d\tau \Bigg]d\xi_1 \bigg\vert
\\
&\le & 2 (2^{j}h_j)
\sum_{(\iota;j,\ell,k)\in\Lambda_j}
\bigg\vert\int_{\bR}\sinc({2^{j}2h_j}\xi_1)\Bigg[\int_{\bR^2} \cW(\frac{\tau_1}{2^j},\tau_2) \cW\left(\frac{\tau_1-\xi_1}{2^j},\tau_2\right)V(\ell+\frac{\tau_1-\xi_1}{\tau_2})V(\frac{\tau_1}{\tau_2})
 \times\\
&& \times e^{2\pi it_12^{j}\xi_1} e^{-2\pi i \ip{t}{ A_{1/a_j}^v\tau}}d\tau \Bigg]d\xi_1\bigg\vert
\\&\le&
2 (2^{j}h_j)
\sum_{(\iota;j,\ell,k)\in\Lambda_j}
\bigg\vert\int_{\bR}\hat g_j(\tau)e^{-2\pi i \ip{t}{ A_{1/a_j}^v\tau} }d\tau\bigg\vert,
\end{eqnarray*}
where $t=A_{1/a_j}^v S_\ell^v k$ with $a_j=2^{-j}$ and
\begin{equation}
\hat g_j(\tau):=\!\int_{\bR}\sinc({2^{j}2h_j}\xi_1) V\left(\ell+\frac{\tau_1-\xi_1}{\tau_2}\right) e^{2\pi it_12^{j}\xi_1}d\xi_1 \cW\left(\frac{\tau_1}{2^j},\tau_2\right)\cW\left(\frac{\tau_1-\xi_1}{2^j},\tau_2\right)V\left(\frac{\tau_1}{\tau_2}\right). \label{eqn:shearg}
\end{equation}
Note that the support of $\cW(\tau_1/2^j,\cdot)$ and of $\cW\left(\frac{\tau_1-\xi_1}{2^j},\cdot\right)$ of variable $\tau_2$ is independent of $j$ and the support of
$V(\cdot/\tau_2)$ of variable $\tau_1$ is depending only on $\tau_2$. Hence,
$\hat g_j(\tau)$ is smooth and compactly supported on a box $\Xi$ of volume independent of $j$,
\[
|\int \hat g_j(\tau)e^{2\pi it \tau}d\tau|\le c_N\|\hat g_j\|_\infty \ipa{|t|}^{-N}.
\]
Note that
\[
\|\hat g_j\|_\infty\le c (2^{j}h_j)^{-1/2};
\]
therefore,
\[
T_1\le c (2^{j}h_j)^{1/2} \sum_{k\in\bZ^2}\ipa{|k|}^{-N}\rightarrow 0, j\rightarrow\infty.
\]

We now bound $T_2$:
\begin{eqnarray*}
\lefteqn{T_2\le \sum_{(\iota;j,\ell,k)\in\Lambda_j}
|\ip{{\hat{\cM}_{h_j}\hat\sigma_{j,\ell,k}^{v}}}{\hat\sigma_{j,\ell,0}^{h}}|}
\\&\le &
\sum_{(\iota;j,\ell,k)\in\Lambda_j}
\int_{\bR}2h_j\sinc(2h_j\xi_1)\Bigg[\int_{\bR^2}\hat\sigma_{a_j,s,0}^v(\tau-(\xi_1,0)) \hat\sigma_{a_j,s',0}^h(\tau)e^{-2\pi i\ip{t}{\tau-(\xi_1,0)}}d\tau\Bigg]d\xi_1
\\
&\le &
\sum_{(\iota;j,\ell,k)\in\Lambda_j}
\int_{\bR^2}\Bigg[\int_{\bR}2h_j\sinc(2h_j\xi_1)\hat\sigma_{a_j,s,0}^v(\tau-(\xi_1,0)) \hat\sigma_{a_j,s',0}^h(\tau)d\xi_1\Bigg]e^{-2\pi i\ip{t}{\tau-(\xi_1,0)}}d\tau
\\&=:&
\sum_{(\iota;j,\ell,k)\in\Lambda_j}
\int_{\bR^2}\hat g_j(\tau) e^{-2\pi i\ip{t}{\tau}}d\tau,
\end{eqnarray*}
where
\[
\hat g_j(\tau):=\int_{\bR}2h_j\sinc(2h_j\xi_1)\hat\sigma_{a_j,s,0}^v(\tau-(\xi_1,0))\hat\sigma_{a_j,s',0}^h(\tau)e^{2\pi it_1\xi_1}d\xi_1.
\]
Using integration by parts, we obtain
\begin{eqnarray*}
|\int_{\bR^2}\hat g_j(\tau)e^{-2\pi i\ip{t}{\tau}}d\tau|
&\le&
c_{L,M}\ipa{|t_1|}^{-L}\ipa{|t_2|}^{-M}\|D^{L,M}\hat g_j\|_\infty \supp(\hat g_j)\\
&\le&
c_{L,M}\ipa{|t_1|}^{-L}\ipa{|t_2|}^{-M}\|D^{L,M}\hat g_j\|_\infty a_j^{-4},
\end{eqnarray*}
where
\begin{eqnarray*}
|D^{L,M}\hat g_j|
&\le& 2h_j
\int_{\bR}|\sinc(2h_j\xi_1)||D^{L,M}(\hat\sigma_{a_j,s,0}^v(\tau-(\xi_1,0))\hat\sigma_{a_j,s',0}^h(\tau))|d\xi_1\\
&\le&
2h_j\|\sinc(2h_j\cdot)\|_2||D^{L,M}(\hat\sigma_{a_j,s,0}^v(\tau-(,0))\hat\sigma_{a_j,s',0}^h(\tau))\|_2
\\&\le& c_{L,M} 2h_j^{1/2} ||D^{L,M}(\hat\sigma_{a_j,s,0}^v(\tau-(,0))\hat\sigma_{a_j,s',0}^h(\tau))\|_\infty  a_j^{-1} .
\end{eqnarray*}
Since
\begin{eqnarray*}
\frac{\partial^N}{\partial \tau_1^N}(\hat\sigma_{a,s,0}^v\hat\sigma_{a,s',0}^h) & = &O(a_j^{3/2} a_j^{N})\textrm{ and}\\
\frac{\partial^N}{\partial \tau_2^N}(\hat\sigma_{a,s,0}^v\hat\sigma_{a,s',0}^h) &=& O(a_j^{3/2} a_j^{N}).
\end{eqnarray*}
Consequently, as $ j\rightarrow \infty$,
\begin{eqnarray*}
T_2& \le& h_j^{1/2} \sum_{(\iota; j,\ell,k)\in\Lambda_j} c_{N} \ipa{|t_1|}^{-N}\ipa{|t_2|}^{-N} a_j^{-4}  a_j^{-1} a_j^{3/2} a_j^{2N}\\
&\le& a_j^{2N-3/2}h_j\rightarrow 0.
\end{eqnarray*}
By construction, $T_3 \leq 2^{-1/2} (T_1 + T_2)$.
\end{proof}

Notice that -- in contrast to the wavelet result -- here we require the stronger condition
$(2^{j}h_j) \to 0$ as $j \goto \infty$ to handle the additional angular component.

We now apply Proposition \ref{prop:l1estimate} to Lemmata \ref{lemm:NormofLj},
\ref{lemm:shearletl1delta}, and \ref{lemm:shearletl1stripmuc} to obtain
the desired convergence for the normalized $\ell_2$ error of the reconstruction
$L_j$ from~(\ref{eq:Inp}). In this case $L=w\cL_j$ and $\Phi$ are shearlets $\sigma^\iota_{j,\ell,k}$ at scale $j$.

\begin{theorem}
\label{theo:curveletl1}
For $h_j=o(2^{-j})$ and $L_j$ the solution to~(\ref{eq:Inp}) with $\Phi$ the shearlet system defined using the Meyer wave\-let
\[
\frac{ \| L_j - w\cL_j \|_2}{\|w\cL_j\|_2 } \to 0, \qquad j \goto \infty.
\]
\end{theorem}

This result shows that we have asymptotically perfect inpainting as long as the size of the gap shrinks faster than
$2^{-{j}}$.  The similar result for wavelet inpainting, Theorem \ref{theo:waveletl1}, only guarantees such successful inpainting
when the gap is asymptotically smaller than $2^{-2j}$.


\subsection{Thresholding}
\label{subsec:thrShx}

Our first claim concerns the set of the thresholding coefficients $\cT_j:=\{\eta = (\iota;j,\ell,k): |\ip{w\cL_j}{\sigma_\eta}|\ge \beta_j\}$ for some $\beta_j>0$.

\begin{lemma}
\label{lemm:curveletthresholdstripthrescoeff}
For $h_j = o(2^{-j})$ as $j \goto \infty$, there exist thresholds $\{\beta_j\}_j$ such that,
for all $j \ge j_0$,
\[
\{(\iota; j, \ell,k) \!:\! |k_1| \le \rho  2^{2j(1+\nu_1)}, \, |k_2| \le 2^{2j\nu_1},\, \ell=0;\iota=v\} \subseteq
\cT_j\]
for some $j_0$, $\nu_1$, and $\nu_2 < 1/4$.
\end{lemma}

\begin{proof}
We first observe that
\[\absip{(1-\cM_{h_j}) w\cL_j}{\sigma^v_{j,\ell,k}}
= |\ip{\delta_0 \star \widehat{w\cL}_j}{\hat{\sigma}^v_{j,\ell,k}}-\ip{\hat{\cM}_{h_j}\star \widehat{w\cL}_j}{\hat{\sigma}^v_{j,\ell,k}}|.
\]
The first term equals
\begin{equation}
\ip{\delta_0 \star \widehat{w\cL}_j}{\hat{\sigma}^v_{j,\ell,k}}
 = 2^{j/2} \int \Bigg[ \int \hat{w}(\xi_1) F(\xi_1/2^{2j},\xi_2)\cW(\xi_1/2^{2j},\xi_2) V(\ell+2^{-j}\xi_1/\xi_2)
e^{-2\pi i \ip{b_1}{ \xi_1}} d\xi_1 \Bigg]e^{-2\pi i\ip{2^{2j}b_2}{ \xi_2}} d\xi_2;
\end{equation}
whereas, by using Lemma \ref{lemm:FouriertrafoM1}, we derive for the second term

\begin{eqnarray*}
\ip{\cM_{h_j}  w\cL_j}{\sigma^v_{j,\ell,k}}
& = & 2h_j \int \sinc(2h_j \tau_1) \int \hat{w}(\xi_1) F_j(\xi_1,\xi_2)  \hat{\sigma}^v_{j,\ell,k} (\xi_1+\tau_1,\xi_2) d\xi d\tau_1\\
& = & 2^{j/2} \int \Bigg[\int \hat{w}(\xi_1)  2h_j \int \sinc(2h_j \tau_1) F(\xi_1/2^{2j},\xi_2)\times\\
&& \times \cW(\xi_1/2^{2j},\xi_2)V(\ell+2^{-j}\frac{\tau_1+\xi_1}{\xi_2})e^{-2\pi i \ip{b_1}{\tau_1+\xi_1}} d\tau_1 d\xi_1 \Bigg] e^{-2\pi i\ip{2^{2j}b_2}{ \xi_2}} d\xi_2\\
& =: & 2^{j/2}\int \hat G(\xi_2)e^{-2\pi i\ip{2^{2j}b_2}{\xi_2}}d\xi_2.
\end{eqnarray*}

By standard arguments, we can deduce that
\[
\absip{(1-\cM_{h_j})  w\cL_j}{\sigma^v_{j,\ell,k}}
\le c_{N_1}  2^{j/2}  \norm{\hat{G}}_\infty  \ang{|2^{2j}b_2|}^{-N_1}.
\]
By $b_2 = k_2/2^{2j}$ due to $b =  (A^v_{2^{-j}} S_{-\ell}^v)^Tk$, we have
\beq \label{eq:curvethres_eq1}
\absip{(1-\cM_{h_j})  w\cL_j}{\sigma^v_{j,\ell,k}}
\le c_{N_1}  2^{j/2}  \norm{\hat{G}}_\infty  \ang{|k_2|}^{-N_1}
\eeq
Let us now investigate the term $\norm{\hat{G}}_\infty$ further. We define
\begin{eqnarray*}
\hat{H}_{\xi_2}(\xi_1) & =&  F(\xi_1/2^{2j},\xi_2)\cW(\xi_1/2^{2j},\xi_2) V(\ell+2^{-j}\xi_1/\xi_2) - 2h_j \int \sinc(2h_j \tau_1)F(\xi_2/2^{2j},\xi_2)\cW(\xi_1/2^{2j},\xi_2)\times\\
&& \times V(\ell+2^{-j}\frac{\xi_1+\tau_1}{\xi_2})e^{-2\pi i \ip{b_1}{\tau_1}} d\tau_1
\end{eqnarray*}
and hence need to analyze
\beq \label{eq:curvethres_eq2}
\norm{\hat{G}}_\infty = \left|\int \hat{w}(\xi_1) \hat{H}_{\xi_2}(\xi_1) e^{-2\pi i \ip{b_1}{\xi_1}} d\xi_1\right|.
\eeq
By Plancherel's theorem and the support properties of $w$,
\begin{eqnarray*}
\Big|\int \hat{w}(\xi_1) \hat{H}_{\xi_2}(\xi_1) e^{-2\pi i \ip{b_1}{ \xi_1}} d\xi_1\Big|
& =&  |(\hat{w}  \hat{H}_{\xi_2})^\vee (-b_1)|\\
& \approx&   c  \Big|\int_{-b_1-\rho}^{-b_1+\rho} H_{\xi_2}(x) dx\Big|.
\end{eqnarray*}
We now need to compute $H$. Using well-known properties of the Fourier transform, we manipulate $H_{\xi_2}$ to obtain
\begin{eqnarray*}
 H_{\xi_2}(x)& = & \left(F(\cdot/2^{2j},\xi_2)\cW(\cdot/2^{2j},\xi_2) V(\ell+2^{-j}(\cdot/\xi_2))\right)^\vee(x)+\\
 &&  - \Big((2h_j \sinc(2h_j\cdot )e^{-2\pi ib_1\cdot}) \star (F(\cdot/2^{2j},\xi_2)\cW(\cdot/2^{2j},\xi_2) V(\ell+2^{-j}(\cdot/\xi_2))\Big)^\vee(-x)\\
& = & \left(F(\cdot/2^{2j},\xi_2)\cW(\cdot/2^{2j},\xi_2)V(\ell+2^{-j}(\cdot/\xi_2))\right)^\vee(x) +\\
&& - \left(2h_j\sinc(2h_j\cdot )e^{-2\pi ib_1\cdot}\right)^\vee(-x)  \left(F(\cdot/2^{2j},\xi_2)\cW(\cdot/2^{2j},\xi_2)V(\ell+2^{-j}\cdot/\xi_2)\right)^\vee(-x)\\
& = &\left(F(\cdot/2^{2j},\xi_2)\cW(\cdot/2^{2j},\xi_2)V(\ell+2^{-j}(\cdot/\xi_2))\right)^\vee(x) +\\
&& - \mathbbm{1}_{[-h_j,h_j]}(x-b_1) \left(F(\cdot/2^{2j},\xi_2)\cW(\cdot/2^{2j},\xi_2)V(\ell+2^{-j}\cdot/\xi_2)\right)^\vee(-x).
\end{eqnarray*}
Hence, since $h_j < \rho$,
\begin{eqnarray*}
\Big|\int_{-b_1-\rho}^{-b_1+\rho} H_{\xi_2}(x) dx\Big|
& = & \Big|\int_{b_1-\rho}^{b_1+\rho} \left(F(\cdot/2^{2j},\xi_2)\cW(\cdot/2^{2j},\xi_2) V(\ell+2^{-j}(\cdot/\xi_2))\right)^\vee(x)\\
& & - \int_{b_1-h_j}^{b_1+h_j}\! \left(F(\cdot/2^{2j},\xi_2)\cW(\cdot/2^{2j},\xi_2) V(\ell+2^{-j}(\cdot/\xi_2))\right)^{\hspace{-1pt}\vee}\hspace{-1pt}(x) dx\Big|\\
& = &  \Big|\int_{2^{j}(b_1-\rho)}^{2^{j}(b_1-h_j)} + \int_{2^{j}(b_1+h_j)}^{2^{j}(b_1+\rho)}
\Big(F(\cdot/2^{2j},\xi_2)\cW(\cdot/2^{2j},\xi_2)\times\\
&& \times V(\ell+2^{-j}(\cdot/\xi_2))\Big)^{\vee}(x)dx\Big|.
\end{eqnarray*}
Notice that this indeed makes sense, since the values $k_1$ ``in between $h_j$ and $\rho$'' should play
an essential role.
As already observed in the proof of~(\ref{eq:curvethres_eq1}), we have $b_1\approx k_1/2^{j}$ for $j$ large and small $|\ell k_2|$
(since $b_1=2^{-j}k_1+2^{-2j}\ell k_2$), and hence
\begin{eqnarray*}
\lefteqn{c  \Big|\int_{-b_1-\rho}^{-b_1+\rho} H(x) dx\Big|}\\
& \approx & c  \Big|\int_{k_1-2^{j}\rho}^{k_1-2^{j}h_j} + \int_{k_1+2^{j}h_j}^{k_1+2^{j}\rho} \Big(F(\cdot/2^{2j},\xi_2)\cW(\cdot/2^{2j},\xi_2) V(\ell+2^{-j}(\cdot/\xi_2))\Big)^\vee(x) dx\Big|.
\end{eqnarray*}
Notice that this fact also implies that the function
\[
\left(F(\cdot/2^{2j},\xi_2)\cW(\cdot/2^{2j},\xi_2) V(\ell+2^{-j}(\cdot/\xi_2))\right)^\vee
\]
is independent of $j$.
Due to the regularity of $W$, there exist some $N_2$ and $c$ such that
\[
|\left(F(\cdot/2^{2j},\xi_2)\cW(\cdot/2^{2j},\xi_2) V(\ell+2^{-j}(\cdot/\xi_2))\right)^\vee(x)| \le c  \ang{|x|}^{-N_2},
\]
and hence by~(\ref{eq:curvethres_eq2}) and the previous computation,
\beq \label{eq:curveletthresholdstripthrescoeff5}
\norm{\hat{G}}_\infty \le c  \ang{\min\{|k_1 - 2^{j}\rho|,|k_1 + 2^{j}\rho|\}}^{-N_2}.
\eeq
Finally, we  study how the term $\hat{H}$ relates to $h_j$. For this,
we set
\[
\hat{J}_{\xi_2}(\tau_1) = F(\xi_1/2^{2j},\xi_2)\cW(\xi_1/2^{2j},\xi_2)V(\ell+2^{-j}\frac{\xi_1+\tau_1}{\xi_2})e^{-2\pi i\ip{b_1}{\tau_1}}
\]
Now,
\begin{eqnarray*}
|\hat{H}_{\xi_2}(\xi_1)|
& = & \bigg|F(\xi_1/2^{2j},\xi_2)\cW(\xi_1/2^{2j},\xi_2) V(\ell+2^{-j}\xi_1/\xi_2) - 2h_j \int \sinc(2h_j\tau_1)  F(\xi_1/2^{2j},\xi_2)\cW(\xi_1/2^{2j},\xi_2)\times\\
&& \times V(\ell+2^{-j}\frac{\xi_1+\tau_1}{\xi_2})e^{-2\pi i \ip{b_1}{\tau_1} }d\tau_1\bigg|\\
& = & |\hat{J}_{\xi_2}(0)-2 h_j\int \sinc(2h_j \tau_1)\hat{J}_{\xi_2}(\tau_1) d\tau_1|\\
& = & |\hat{J}_{\xi_2}(0)-\int \hat{1}_{[-h_j,h_j]}(\tau_1) \hat{J}_{\xi_2}(\tau_1) d\tau_1|\\
& = & |\hat{J}_{\xi_2}(0)-\int_{-h_j}^{h_j} J_{\xi_2}(x) dx|\\
& = & |\int_{|x|>h_j} J_{\xi_2}(x) dx|.
\end{eqnarray*}
Hence another way to estimate~(\ref{eq:curvethres_eq2}) is by
\begin{eqnarray*}
\norm{\hat{G}}_\infty & \le&  c  \norm{\hat{H}}_\infty\\
& \le & c  \max_{\xi_1,\xi_2}\Bigg|\int_{|x|>h_j}
(F(\xi_1/2^{2j},\xi_2)\cW(\xi_1/2^{2j},\xi_2) V(\ell+2^{-j}\frac{\cdot+\xi_1}{\xi_2}) )^\vee(x-b_1) dx\Bigg|\\
& \le & c  \max_{\xi_1,\xi_2}\Bigg|\int_{|x|>2^{j}h_j}
(F(\xi_1/2^{2j},\xi_2)\cW(\xi_1/2^{2j},\xi_2) V(\ell+\frac{\cdot+2^{-j}\xi_1}{\xi_2}) )^\vee(x-2^{j}b_1) dx\Bigg|.
\end{eqnarray*}
Certainly, the minimum is attained in the center of the mask, i.e., with $b=0$. So by
combining this with~(\ref{eq:curvethres_eq1}) and~(\ref{eq:curveletthresholdstripthrescoeff5}),
\begin{eqnarray*}
\absip{(1-\cM_{h_j}) w\cL_j}{\sigma^v_{j,\ell,k}}
& \le & c  2^{j} \Bigg|\int_{|x|>2^{2j}h_j} \max_{\xi_1,\xi_2}\Bigg|\int_{|x|>2^{j}h_j}
(F(\xi_1/2^{2j},\xi_2)\times\\
&&\times \cW(\xi_1/2^{2j},\xi_2)V(\ell+\frac{\cdot+2^{-j}\xi_1}{\xi_2}) )^\vee(x-2^{2j}b_1)dx\Bigg| \times \\ 
& & \times \ang{\min\{|k_1 - 2^{2j}\rho|,|k_1 + 2^{2j}\rho|\}}^{-N_2} \ang{|k_2|}^{-N_1},
\end{eqnarray*}
which is what we intend to use as a ``model.'' Observe that this indeed is the right
intuitive estimate, since the $k_2$ component has to decay rapidly away from zero
thereby sensing the singularity in zero in this direction. In contrast, the $k_1$ component
stays greater or equal to $\ang{2^{2j}\rho}^{-N_2}$ up to the point $2\rho2^{2j}$ and then
decays rapidly in accordance with the fact that until the point $k_1=\rho2^{2j}$ we
are ``on'' the line singularity which decays smoothly up with $\hat{w}$. Also, the
required angle sensitivity is represented. Finally, the first
term models the behavior in the mask, which is also nicely supported by the fact that
the crucial product $2^{2j}h_j$ is appearing therein.  Set
\[J(\cdot)= F(\cdot/2^{2j},\xi_2)\cW(\cdot/2^{2j},\xi_2) V(\ell+2^{-j}(\cdot/\xi_2)).
\]
Since $2^{j} h_j \to 0$ as $j \to \infty$, letting $j \to \infty$ we have
\[
\left|\int_{|x|>2^{j}h_j}
\check{J}(x) dx\right|
\le C.
\]
We now use
\begin{eqnarray*}
\beta & = & c  2^{j/2}  (C-\epsilon) \ang{|2^{j \epsilon}|}^{-N_1}  \ang{\min\{|(2^{j \epsilon}-1)2^{j}\rho|,|(2^{j \epsilon}+1)2^{j}\rho|\}}^{-N_2}
\end{eqnarray*}
as a threshold. It follows
immediately that, for all $j \ge j_0$,
\[
\{(\iota; j,\ell,k) \!:\! |k_1| \le \rho  2^{2j(1+\nu_1)}, \, |k_2| \le 2^{2j\nu_1},\, \ell=0;\iota=v\} \subseteq
\cT_j
\]
for some $j_0$ and $\nu_1$.
\end{proof} 

\begin{lemma}
\label{lemm:curveletthresholddelta}
\[
\sum_{\eta \in \cT_j^c} \absip{w\cL_j}{\sigma_\eta} = o(2^{j}), \textrm{  $j \goto \infty$} .
\]
\end{lemma}

\begin{proof}
We observe from the proof of Lemma \ref{lemm:shearletl1delta}, that the desired property is
automatically satisfied provided that, for all $j \ge j_0$, the set $\cT_j$ contains
\[
\{(\iota;j,\ell,k) : |k_1| \le \rho 2^{2j(1/2+\nu_1)},\, |k_2| \le 2^{2j\nu_1}, \,\ell=0,\iota=v\},
\]
for some $\nu_1 > 0$, which is the content of Lemma \ref{lemm:waveletthresholdstripthrescoeff}.
\end{proof} 

We next analyze the second term in the estimate from Proposition \ref{prop:thresholdingestimate}.

\begin{lemma}
\label{lemm:curveletthresholdstriplastterm}
For $h_j = o(2^{-{j}})$ as $j \goto \infty$,
\[
 \sum_{\eta \in \cT_j} \absip{\cM_{h_j}  w\cL_j}{\sigma_\eta} = o(2^{j}) , \qquad j \goto \infty.
\]
\end{lemma}

\begin{proof}
First, we need to derive some estimates dependent on $(k,\ell)$ for the term
$\absip{\cM_{h_j}  w\cL_j}{\sigma^\iota_{j,\ell,k}}$. By using the definitions of $\cM_{h_j}$
and $w\cL_j$ and a change of variables, we obtain
\begin{eqnarray*}
\ip{\cM_{h_j}  w\cL_j}{\sigma^v_{j,\ell,k}}
 & = &2^{j/2} \int \Bigg[\int \hat{w}(\xi_1)  2h_j \int \sinc(2h_j \tau_1) F(\xi_1/2^{2j},\xi_2)\times\\
 && \times \cW(\xi_1/2^{2j},\xi_2)V(\ell+2^{-j}\frac{\tau_1+\xi_1}{\xi_2}) e^{-2\pi i b_1(\tau_1+\xi_1)} d\tau_1 d\xi_1 \Bigg]e^{-2\pi i\ip{2^{2j}b_2 }{\xi_2}} d\xi_2.
\end{eqnarray*}
Let $\hat{G}$ now be the function
\[
\hat{G}(\xi_2)  = \int \hat{w}(\xi_1)  2h_j \int \sinc(2h_j \tau_1) F(\xi_1/2^{2j},\xi_2) \cW(\xi_1/2^{2j},\xi_2)V(\ell+2^{-j}\frac{\tau_1+\xi_1}{\xi_2})e^{-2\pi i \ip{b_1}{\tau_1+\xi_1}} d\tau_1 d\xi_1.
\]
This function is supported on the set $[1/16,1/2]$, which is independent of $j$.  By standard arguments, we can deduce that
\beq \label{eq:curvethres_eq6}
\absip{\cM_{h_j}  w\cL_j}{\sigma^v_{j,\ell,k}}
\le c_{N_1}  2^{j/2}  \norm{\hat{G}}_\infty  \ang{|k_2|}^{-N_1}.
\eeq
Let us now investigate the term $\norm{\hat{G}}_\infty$ further. We define
\begin{eqnarray*}
\hat{H}_{\xi_2}(\xi_1) &= &2h_j \int \sinc(2h_j \tau_1) F(\xi_1/2^{2j},\xi_2)\cW(\xi_1/2^{2j},\xi_2) V(\ell+2^{-j}\frac{\tau_1+\xi_1}{\xi_2})e^{-2\pi i \ip{b_1}{\tau_1+\xi_1}} d\tau_1,
\end{eqnarray*}
and hence need to analyze
\beq \label{eq:curvethres_eq7}
\norm{\hat{G}}_\infty = \left|\int \hat{w}(\xi_1) \hat{H}_{\xi_2}(\xi_1) e^{-2\pi i \ip{b_1}{\xi_1}} d\xi_1\right|.
\eeq
By Plancherel's theorem and the support properties of $w$,
\[
\Big|\int \hat{w}(\xi_1) \hat{H}_{\xi_2}(\xi_1) e^{-2\pi i \ip{b_1}{ \xi_1}} d\xi_1\Big|
 =  |(\hat{w}  \hat{H}_{\xi_2})^\vee (-b_1)| \approx   c  \Big|\int_{-b_1-\rho}^{-b_1+\rho} H_{\xi_2}(x) dx\Big|.
\]
Next,
\begin{eqnarray*}
H_{\xi_2}(x) &=&  \Big((2h_j\sinc(2h_j\cdot)e^{-2\pi ib_1\cdot}) \star (F(\frac{\cdot}{2^{2j}},\xi_2)\cW(\frac{\cdot}{2^{2j}},\xi_2) V(\ell+2^{-j}(\cdot/\xi_2)))\Big)^\vee(-x)\\
& = & \left( 2h_j\sinc(2h_j \cdot)e^{-2\pi ib_1\cdot}\right)^\vee(-x)  \left(F(\cdot/2^{2j},\xi_2)\cW(\cdot/2^{2j},\xi_2)V(\ell+2^{-j}(\cdot/\xi_2))\right)^\vee(-x)\\
& = & \mathbbm{1}_{[-h_j,h_j]}(-x-b_1)  \left(F(\cdot/2^{2j},\xi_2)\cW(\cdot/2^{2j},\xi_2)V(\ell+2^{-j}(\cdot/\xi_2))\right)^\vee(-x).
\end{eqnarray*}
Hence, since $h_j < \rho$,
\begin{eqnarray*}
\Big|\int_{b-_1-\rho}^{-b_1+\rho} H_{\xi_2}(x) dx\Big|
&\!\! =\!\! &  \Big|\int_{b_1-h_j}^{b_1+h_j} \hspace{-4 pt}\left(F(\cdot/2^{2j},\xi_2)\cW(\cdot/2^{2j},\xi_2)V(\ell+2^{-j}
(\cdot/\xi_2))\right)^{\hspace{-3 pt}\vee}\hspace{-4 pt}(-x) dx\Big|\\
& \!\!= \!\!&  \Big|\int_{2^{j}(b_1-h_j)}^{2^{j}(b_1+h_j)}\hspace{-4 pt} \left(F(\cdot/2^{j},\xi_2)
\cW(\cdot/2^{j},\xi_2)V(\ell+(\cdot/\xi_2))\right)^{\hspace{-3 pt}\vee}\hspace{-4 pt}(-x) dx\Big|.
\end{eqnarray*}
Notice that this indeed makes sense, since due to the masking, the length of the
line singularity is not allowed to play a role here.
Since $(k,\ell)\in\cT_j$, we have
\[\Big|\int_{-b_1-\rho}^{-b_1+\rho} H(x) dx\Big|
 =  \Big|\int_{k_1-2^{j}h_j}^{k_1+2^{j}h_j} \hspace{-2 pt}(F(\cdot/2^{j},\xi_2)\cW(\cdot/2^{j},\xi_2)V(\ell+(\cdot/\xi_2)))^{\hspace{-1 pt}\vee}\hspace{-1 pt}(-x) dx\Big|.
\]
Due to the regularity of $W$, there exists some $N_2$ and
$c$ (possibly differing from the one before, but we do not need to distinguish those) such that
\[
|(F(\cdot/2^{j},\xi_2)\cW(\cdot/2^{j},\xi_2)V(\ell+(\cdot/\xi_2)))^\vee(-x)| \le c  \ang{|x|}^{-N_2},
\]
and hence by~(\ref{eq:curvethres_eq7}) and the previous computation,
\[
\norm{\hat{G}}_\infty \le c  \ang{\min\{|k_1 - 2^{j}h_j|,|k_1 + 2^{j}h_j|\}}^{-N_2}.
\]
Combining this estimate with~(\ref{eq:curvethres_eq6}), we obtain
\[
\absip{\cM_{h_j}  w\cL_j}{\sigma^v_{j,\ell,k}} \le c  2^{j/2}  \ang{|k_2|}^{-N_1}  \ang{\min\{|k_1 - 2^{j}h_j|,|k_1 + 2^{j}h_j|\}}^{-N_2},
\]
which is what we intend to use.

Hence,
\begin{eqnarray*}
\frac{1}{c}\sum_{\eta \in \cT_j} \absip{\cM_{h_j}  w\cL_j}{\sigma_\eta}
& \le &  2^{j/2} \sum_{\eta \in \cT_j} \ang{|k_2|}^{-N_1}  \ang{\min\{|k_1 - 2^{j}h_j|,|k_1 + 2^{j}h_j|\}}^{-N_2}\\
& \le & 2^{2j(1/4+\nu_2)}.
\end{eqnarray*}
Since $\nu_2 < 1/4$, the lemma is proven.
\end{proof} 

We now apply Proposition \ref{prop:thresholdingestimate} to Lemmata \ref{lemm:NormofLj},
\ref{lemm:curveletthresholddelta}, and \ref{lemm:curveletthresholdstriplastterm} to obtain
the desired convergence for the normalized $\ell_2$ error of the reconstruction
$L_j$ from {\sc One-Step-Thres\-hol\-ding} in Figure \ref{fig:onestepthresholding}.
In this case $x=w\cL_j$ and $\Phi$ are shearlets $\sigma^\iota_{j,\ell,k}$ at scale $j$.

\begin{theorem}
\label{theo:curveletthreshold}
For $h_j=o(2^{-j})$ and $L_j$ the solution to~(\ref{eq:Thresh}) with $\Phi$ the shearlet system defined using the Meyer wave\-let

\[
\frac{ \| L_j - w\cL_j \|_2}{\|w\cL_j\|_2 } \to 0, \qquad j \goto \infty.
\]
\end{theorem}

This result shows that if the size of the gap shrinks faster than
$2^{-j}$, the gap can be asymptotically perfect inpainted.


\section{A Comparison of Shearlet vs. Wavelets}\label{sec:neg_wave}

From the results of previous sections, we see that the size of the gaps which can be filled by
shearlets ($h_j=o(2^{-j})$) with asymptotically high precision is  larger than the corresponding size for
wavelets ($h_j=o(2^{-2j})$);
however, certainly we still need to prove that we cannot do better than the
presented rates for wavelet in order to show that shearlets perform better than wavelets. In fact, we show
that the rates presented for wavelets are indeed the ``critical scales'' for the thresholding case.

\begin{theorem}\label{thr_wav_neg} Let $\psi_{\lambda}$ be the Meyer Parseval wavelets.
Let $\cT$ be a index set such that
\[
\mathcal{T}\supseteq\{(\iota, j,0, (k_1,0)): |k_1|\le 2^{2j}h_j-K_0\}
\]
for some $K_0>0$ and $h_j>0$. Then, we have
\[
\sum_{\lambda\in\mathcal{T}}|\langle \cM_{h_j} w\mathcal{L}_j,\psi_{\lambda}\rangle|
= O(2^{2j}h_j).
\]
\end{theorem}
\begin{proof}

Recall that at level $j$, the signal $w\cL$ is filtered with the three corresponding frequency strips: \[\check{F}_j = \sum_{\iota\in\{h,v,d\}} \left( W^\iota(2^{-2j}\xi) + W^\iota(2^{-2j-1}\xi) \right)\] with \[\widetilde{F}_j = \sum_{\iota\in{h,v,d}} W^\iota(2^{-j}\xi)\] so that \[F_j = \widetilde{F}_{2j} + \widetilde{F}_{2j+1}.\] We can consider each of the filtered signals; i.e., consider
$w\cL_j^\iota:=w\cL\star F_j^\iota$ with $\iota=v,h,d$. Since the signal is a horizontal line segment, we only need to consider
$w\cL_j^h$.  For simplicity, we denote $w\cL_j:= w\cL_j^h$,  $F_j:=F_j^h$, and $\psi_\lambda = \psi_{j,k}^h =: \psi_{j,k}$. Note
that $\widetilde{F}^\vee_j(x,y) = 2^{2j}\phi(2^jx)\check W(2^jy)$.  We want to estimate the coefficients $|\ip{\cM_{h_j} w\cL_j}{\psi_\lambda}|$.  As with other proofs for wavelets, we first consider $\widetilde{w\cL}_j$.
By definition, we have
\begin{eqnarray*}
\ip{\cM_{h_j} \widetilde{w\cL}_j}{\psi_\lambda}
&=&\int_{|x|<h_j}\int_{y\in\bR} \widetilde{w\cL}_j(x,y)\psi_\lambda(x,y)dydx\\
&=&\int_{|x|<h_j}\int_{y\in\bR} (w\cL\star \widetilde{F}^\vee_j)(x,y)\psi_\lambda(x,y)dydx\\
&=&\int_{|x|<h_j}\int_{y\in\bR} \int_{z\in\bR^2 } w\cL(z_1,z_2) \widetilde{F}^\vee_j((x,y)-(z_1,z_2))dz \psi_\lambda(x,y)dydx.
\end{eqnarray*}
Now, by the definition of $w\cL$, we have
\begin{eqnarray*}
\ip{\cM_{h_j} \widetilde{w\cL}_j}{\psi_\lambda}
&=&\int_{|x|<h_j}\int_{y\in\bR} \int_{-\rho }^\rho w(z) \widetilde{F}^\vee_j(x-z,y)dz\psi_\lambda(x,y)dydx\\
&\approx& c\int_{|x|<h_j}\int_{y\in\bR} \int_{-\rho }^\rho \widetilde{F}^\vee_j(x-z,y)dz\psi_\lambda(x,y)dydx\\
&=&c\int_{|x|<h_j}\int_{y\in\bR} \int_{-\rho }^\rho 2^{2j} \phi(2^j(x-z))\check W(2^jy)dz 2^j\phi(2^jx-k_1)\check W(2^jy-k_2)dydx\\
&=&c2^j\int_{y\in\bR}\check W(2^jy)\check W(2^jy-k_2)dy 2^{2j} \int_{|x|<h_j} \int_{-\rho }^\rho \phi(2^j(x-z))dz\phi(2^jx-k_1)dx\\
&=&c   2^{2j}\int_{|x|<h_j} \int_{-\rho }^\rho \phi(2^j(x-z))dz\phi(2^jx-k_1)dx\\
&=&c   2^{j}\int_{|x|<h_j} \int_{-2^j\rho+2^jx }^{2^j\rho+2^jx} \phi(z)dz\phi(2^jx-k_1)dx\\
&=&c   \int_{-k_1-2^jh_j}^{-k_1+2^jh_j} \int_{-2^j\rho+x+k_1 }^{2^j\rho+x+k_1} \phi(z)dz\phi(x)dx.
\end{eqnarray*}
For each $x\in [-k_1-2^jh_j,-k_1+2^jh_j]$, we have $x+k_1\in[-2^jh_j,2^jh_j]$. Consequently,
we have
\[
[-2^j\rho+x+k_1,2^j\rho+x+k_1] \supseteq [-2^j(\rho-h_j),2^j(\rho-h_j)]
\]
for all $x\in [-k_1-2^jh_j,-k_1+2^jh_j]$.
Note that $\rho>h_j$. Hence, when $j$ is large enough, we have $ \int_{-2^j\rho+x+k_1 }^{2^j\rho+x+k_1} \phi(z)dz\approx c\neq 0$
due to $\int\phi(x)dx\neq 0$. Therefore, we have
\[
\ip{\cM_{h_j} w\cL_j}{\psi_\lambda}
\approx c  \left( \int_{-k_1-2^{2j}h_j}^{-k_1+2^{2j}h_j}+ \int_{-k_1-2^{2j+1}h_j}^{-k_1+2^{2j+1}h_j} \right) \phi(x)dx.
\]

As $\int \phi(x)dx\neq 0$, there exists $K_0>0$ such that
\[
\int_{|x|<K}\phi(x)dx \ge c_0
\]
for some $c_0>0$ as long as $K>K_0$.
Hence, when $j$ is large enough so that $2^{2j} h_j>K_0$ and $k_1\in [-(2^{2j}h_j-K_0),2^{2j}h-K_0]$, we have about $2^{2j}h_j-K_0$ many
coefficients that are larger than $c_0$. Consequently, when $j$ is large enough, we have
\[
\sum_{k\in\mathcal{T}}|\langle \cM_{h_j} w\mathcal{L}_j,\psi_\lambda\rangle|
= O(2^{2j}h_j)
\]
as long as the index set $\mathcal{T}\supseteq\{(\iota, j, 0,(k_1,0)): |k_1|\le 2^{2j}h_j-K_0\}$.

For the other orientations $w\cL_j^v$ and $w\cL_j^d$, the coefficients are negligible following calculations similar to above.
\end{proof} 

In the proof of Proposition~\ref{prop:thresh_success}, we have
\[\|x^\star - x^0\|_2=\|\Phi \mathbbm{1}_{\mathcal{T}^c}\Phi^* P_Kx^0+\Phi \mathbbm{1}_{\mathcal{T}}\Phi^* P_M x^0\|_2
=:\|T_1+T_2\|_2\ge \|T_2\|_2-\|T_1\|_2.
\]
In the wavelet threshold case, the first term corresponds to
$T_1=\sum_{k\in\mathcal{T}^c}|\ip{w\cL_j}{\psi_{\lambda}}|$, while the second term corresponds to
$T_2=\sum_{k\in\mathcal{T}}\langle \cM_{h_j}\cdot w\mathcal{L}_j,\psi_{\lambda}\rangle$ for some index set $\mathcal{T}$.
As shown in the wavelet threshold, to guarantee that the first term $\|T_1\|_2$ is small, the index set $\cT$ is chosen such that $\cT\supseteq\{(k_1,k_2):
|k_1|\le \rho 2^{2j(1+\nu_1)}, |k_1|\le 2^{2j\nu_2}\}$. But then the second term $\|T_2\|_2$ will be of order $O(2^{2j}h_j)$ as shown above. If
$h_j$ decays slower than order of $O(2^{-j})$, then we have $\|L_j-w\cL_j\| = O(2^{j})$. Thus, we have the following theorem:

\begin{theorem}\label{thm:wavneg}
For $h_j=\omega(2^{-j})$ and $L_j$ the solution to~(\ref{eq:Thresh}) where $\Phi$ is the 2D Meyer Parseval system,
\[
\frac{\|L_j-w\cL_j\|_2}{\|w\cL_j\|_2}\nrightarrow 0, \quad j \rightarrow \infty.
\]
\end{theorem}
That is, the wavelet threshold method does not fill the gap.
Heuristically, one can think about the situation when the gap size $h_j$ is fixed as 1. Consider the wavelets
$2^j\phi(2^jx-k_1)\check W(2^jy)$. Then as $j\rightarrow\infty$, the number of such wavelets that fall in the gap is about $O(2^{2j})$.
The norm  $\langle \cM_{h_j} w\cL, \psi_\lambda\rangle$ for any such wavelets in the gap is about the same. Consequently,
the total energy concentrated in the gap will be about $O(2^{2j})$.

 When $2^{2j}h_j\rightarrow0$ and since $|\phi(x)|\le c_N\langle |x| \rangle^{-N}$ for any $N$, we have
\[\absip{\cM_{h_j}  w\cL_j}{\psi_\lambda}\le c 2^{2j} h_j \langle\min\{|k_1-2^{2j}h_j|,|k_1+2^{2j}h_j|\}\rangle^{-N}.
\]
For the Meyer mother wavelets $W^v=\check W(x)\phi(y)$ and $W^d=\check W(x)\psi(y)$, the above inequality still holds.
In this case, the threshold method fills the gap.

Contrasting Theorem~\ref{theo:curveletthreshold} and Theorem~\ref{thm:wavneg}, we see that when the gap size $h_j$ decays
like $2^{j}$, the using the {\sc One-Step-Thresholding} algorithm produces a good approximation of the original image if
shearlets are used but does not if wavelets are used.

Figure~\ref{fig:compare} shows a comparison of wavelet- and shearlet-based inpainting results. In the left column, a seismic image
containing mainly curvilinear features is masked by 3 vertical bars. Using 2D Meyer tensor wavelets or shearlets -- we refer to the
ShearLab package in \url{www.shearlab.org} for codes of shearlet transforms --, the coefficients of the masked image are computed.
After applying the threshold and applying the backward transform we derive a first approximation of an inpainted image by
leaving the known part unchanged. These steps are then iterated with the threshold becoming smaller at each iteration. The
outcome is illustrated in the middle column of Figure~\ref{fig:compare}. The last column is the zoom-in comparison. From this,
we can also visually confirm that the shearlet system is superior to the chosen wavelet system when inpainting images governed by curvilinear
structures such as the exemplary seismic image.

\begin{figure}[h]
\begin{center}
\includegraphics[width=1.0in,height=1.0in]{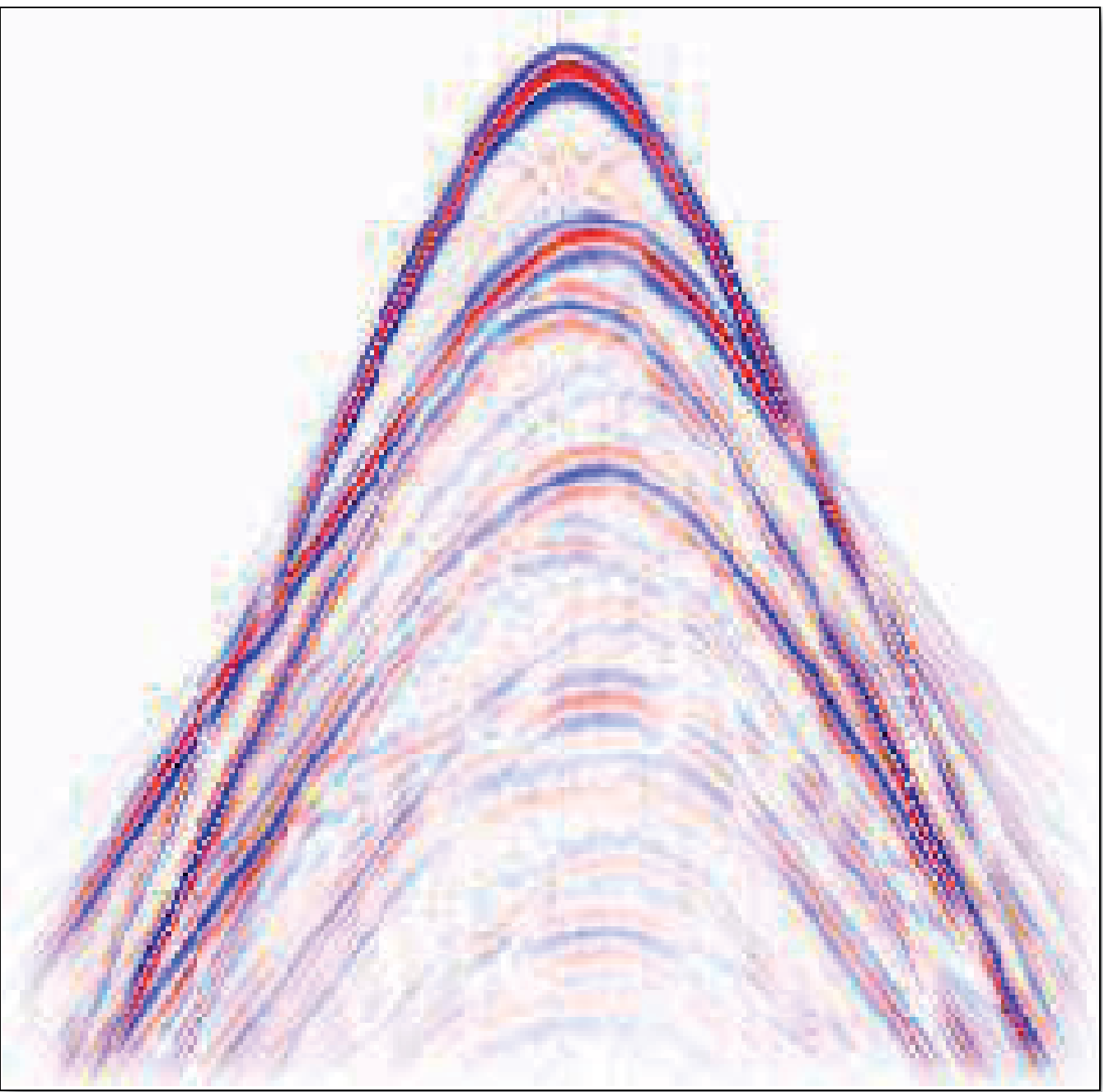}\,
\includegraphics[width=1.0in,height=1.0in]{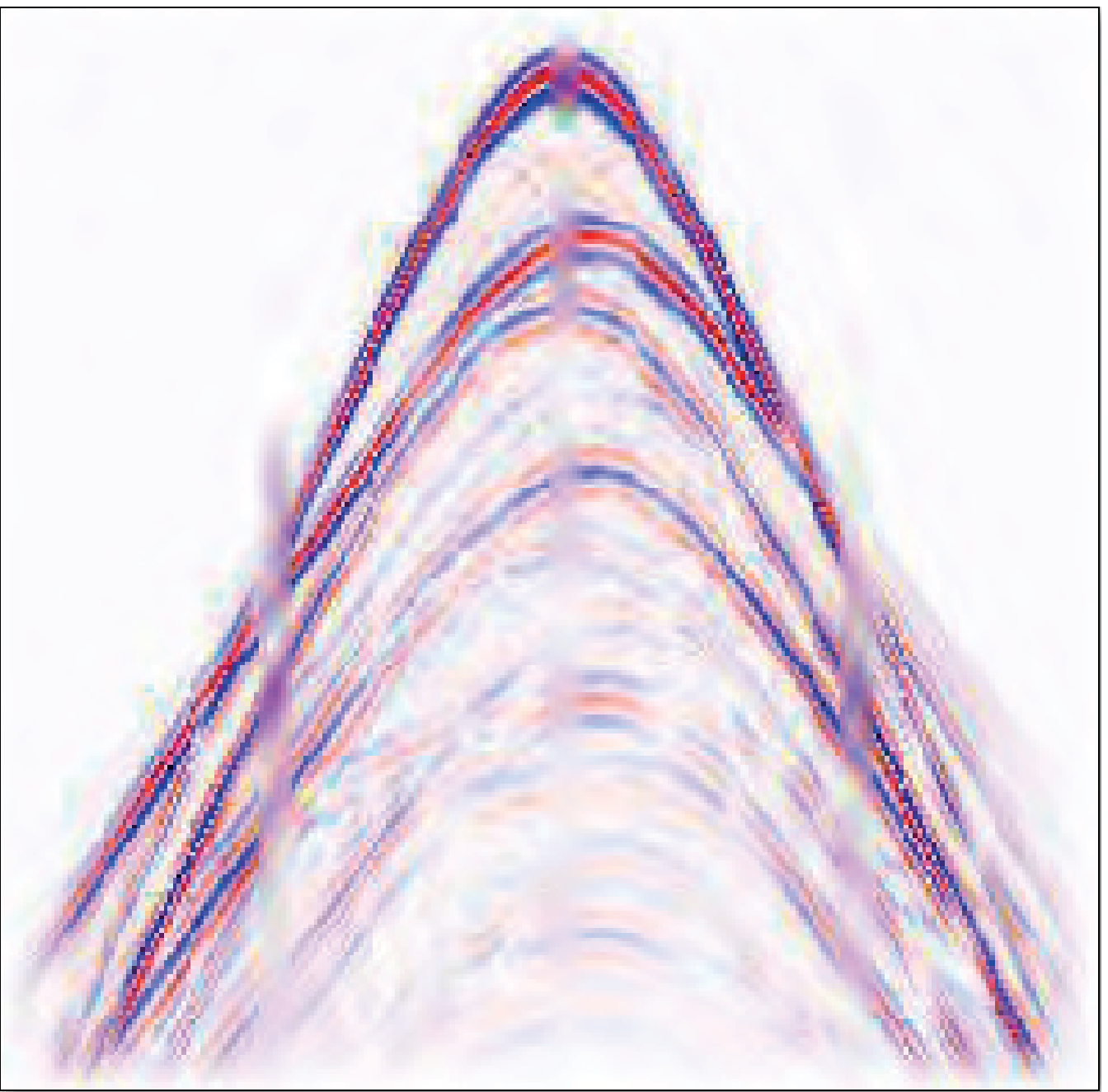}\,
\includegraphics[width=1.0in,height=1in]{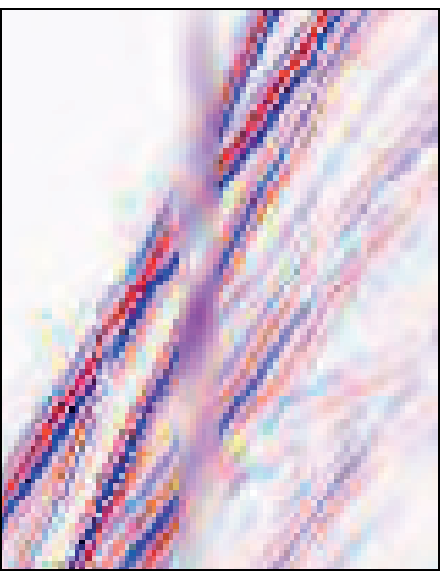}\\
\includegraphics[width=1.0in,height=1.0in]{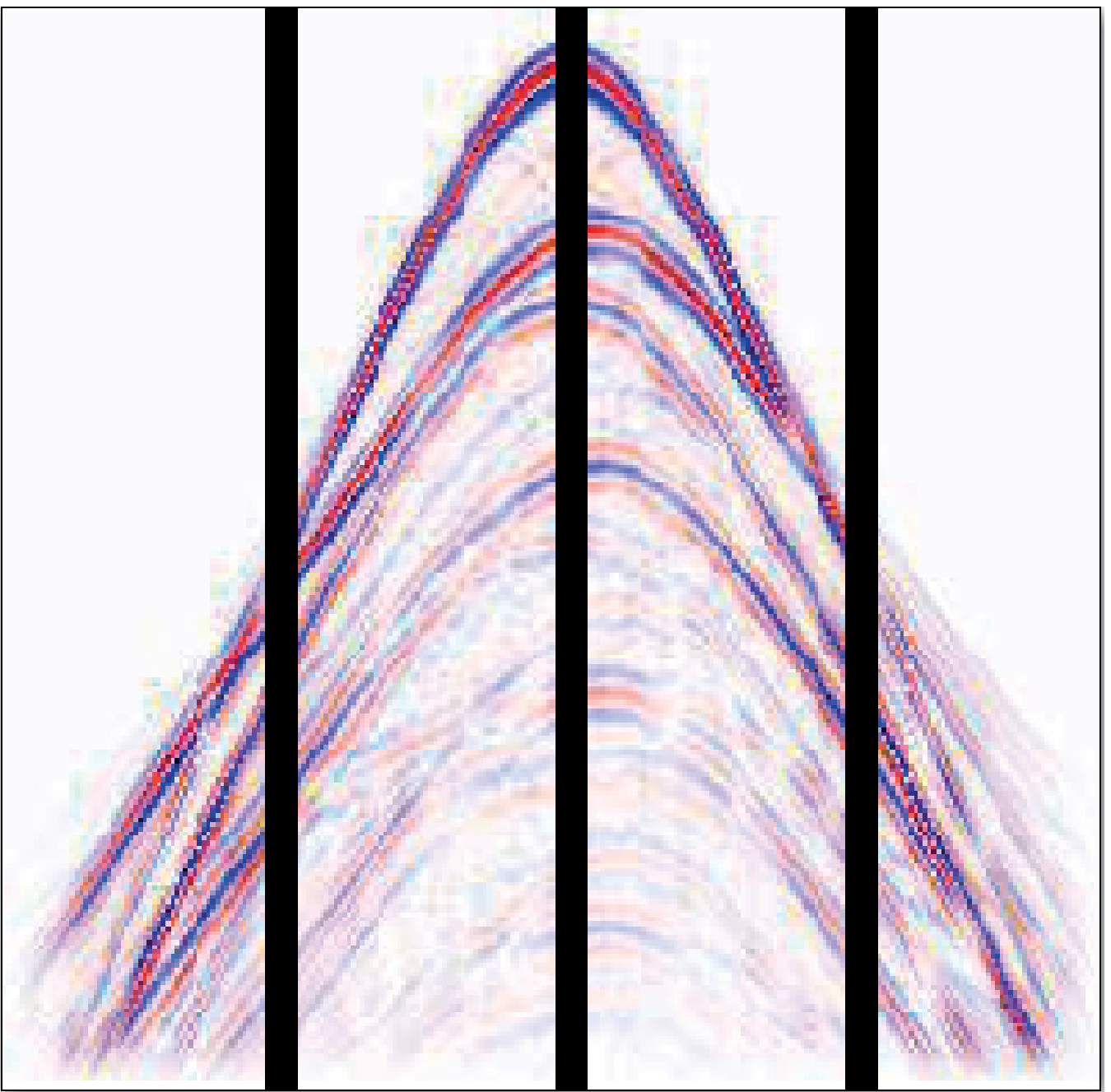}\,
\includegraphics[width=1.0in,height=1.0in]{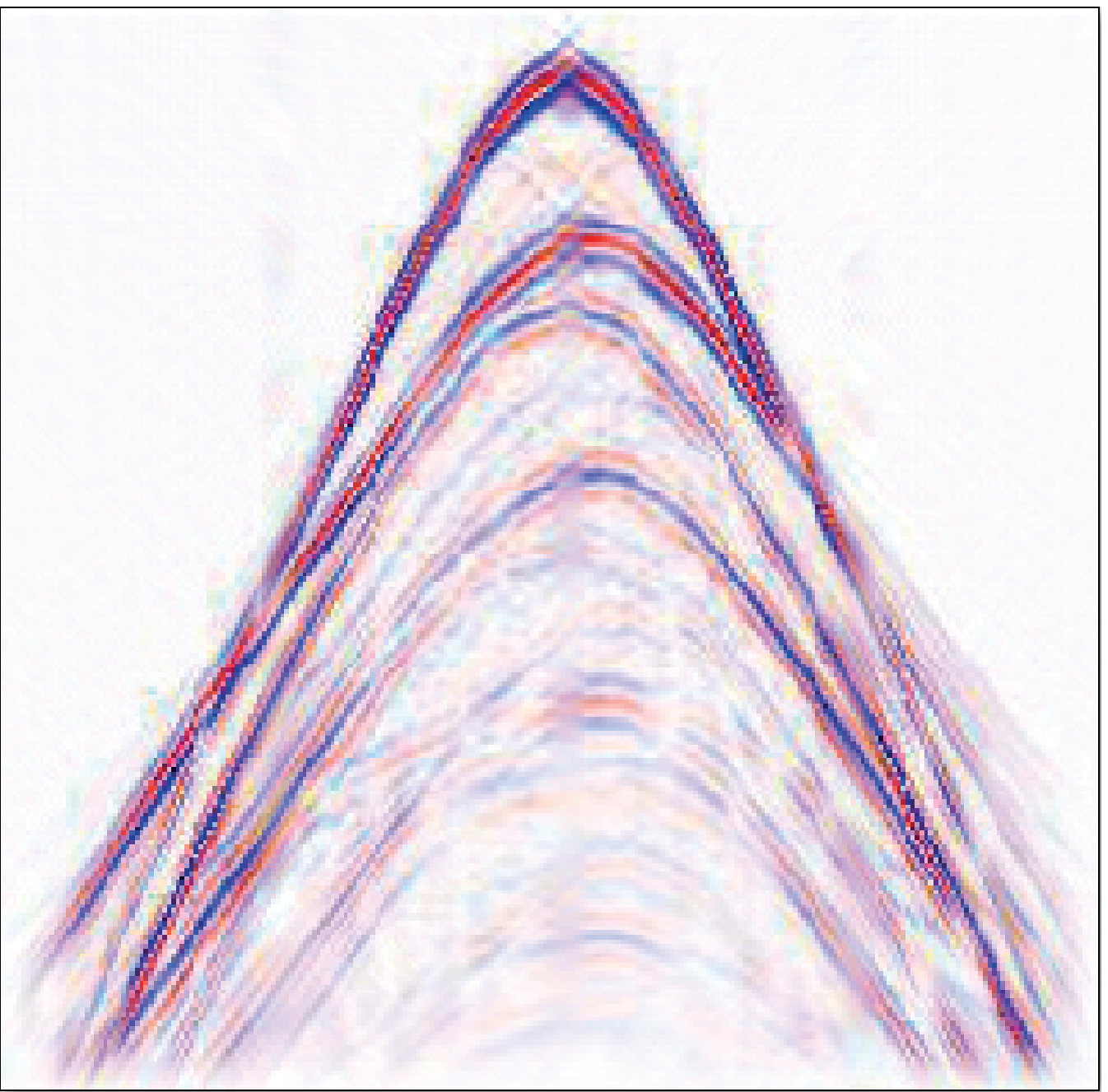}\,
\includegraphics[width=1.0in,height=1in]{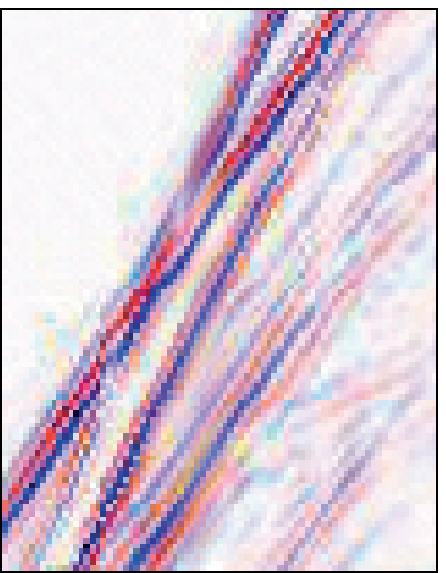}
\caption{Left column:  original image and  missing data. Middle column: wavelet inpainting and shearlet inpainting. Right column: wavelet zoom in and shearlet zoom in.}
\label{fig:compare}
\end{center}
\end{figure}


\section{Extensions and Future Directions}\label{sec:ext}

As mentioned previously, we believe that this work and \cite{KKZ11b} make important steps in a new direction of theoretical analysis of inpainting problems.  When taking into account the similar results concerning geometric separation in \cite{DK10} and \cite{DK08a}, clustered sparsity could provide a new paradigm to prove theoretical results in a variety of problems involving sparsity.  With this in mind, we mention possible extensions of this work as well as current limitations.

\bitem
\item {\em More General Singularity Models}. We anticipate that our results can be generalized to a much broader setting.  In \cite{DK10,DK08a}, curvilinear singularities were segmented and flattened out using the Tubular Neighborhood Theorem.  This was done in such a way as to be able to apply results concerning the clustering of curvelet coefficients along linear singularities to curvilinear singularities.  Using this technique, the results in this paper concerning line singularities $w\cL$ should be able to be extended to curvilinear singularities.
\item {\em Different Masks}. In this paper, we focus on a vertical strip as mask. However, after rotation other typical masks are locally vertical strips, and the analysis in our proofs occurred locally around the missing singularity.  It is possible to think of a
ball with radius $h$ as mask, in which case similar results should be obtained. Other imaginable
shapes could be horizontal strips, flat ellipsoids, and other polygonal objects.
\item {\em Different Recovery Techniques}. Both hard and soft iterative thresholding techniques are quite common and usually
produce convincing results. The results in this paper concern one-step-(hard)-thresholding rather than iterative thresholding.  As iterative thresholding is stronger than one-pass thresholding, we strongly believe that a similar
abstract analysis can be derived leading to asymptotically precise inpainting results in this case.
\item {\em Other Dictionaries}. It should also be pointed out that the results in Section~\ref{sec:abs_anal} hold for all Parseval frames.  Furthermore, the asymptotic analysis in Sections~\ref{sec:pos_wave} and \ref{sec:shear_pos} hold  not only for the
Meyer Parseval wavelets and  shearlets, but also, for instance, for radial wavelets -- or any types of wavelets
with isotro\-pic feature at each scale similar to the radial wavelets -- and other directional multiscale representation
systems such as curvelets. The necessary changes in the proofs are foreseeable. Also, the novel framework of
parabolic molecules advocated in \cite{GK12} could be applied.  Furthermore given the construction of
$3$-dimensional shearlets in \cite{KanL11,GLL10,GLL11,KLL12}, it seems likely that the proofs in Sections~\ref{sec:shear_pos}
and~\ref{sec:aux_shear} will generalize in a straight-forward but technical manner to the $3$-dimensional case.
\item {\em Noise}. Data is typically affected by noise, a situation we considered in the abstract setting.
This analysis can be directly applied also for the wavelet and shearlet inpainting results, leading to
the same asymptotical behavior,  provided that the noise $n$ is small comparing to the signal; i.e.,   the $\ell_1$ norm of $\Phi^* n$ is of order smaller than the $\ell_2$ norm of filtered signal.  However, in the literature, noise is typically measured by the $\ell_2$ not the $\ell_1$ norm.
\eitem


\section{Appendix: Decay of Shearlet Coefficients Related to Line Singularity}\label{sec:aux_shear}


We present the idea of a \emph{continuous shearlet system} in order to prove various auxiliary results. For $\iota \in \{h,w\}$, $a>0$, $s \in \bR$, and $t \in \bR^2$, define
 \[ \hat{\sigma}^\iota_{a,s,t}(\cdot) = a^{3/2}\cW(a^{2}\cdot)V^\iota(\cdot A_{a}^\iota S_{-s}^\iota)e^{2\pi i \ip{\cdot}{t}}. \]
It is easy to show that $\sigma^\iota_{a,s,t} = a^{-3/2}\sigma^{\iota,a,s}(S_{s}^\iota A_{a^{-1}}^\iota(\cdot-t))$ for some smooth function $\sigma^{\iota,a,s}$.
For $s = \pm a$, we similarly define the continuous version of the ``seam" elements $\sigma_{a,\pm a,t}$.
The discrete shearlet system $\{\sigma^\iota_{j,\ell,k}\}$ is then obtained by sampling $\sigma_{a,s,t}^\iota$  on the discrete set of points
\begin{eqnarray*}
\lefteqn{\{\iota = h,v\} \times \{ a = 2^{-j} : j \in \bN \} \times \{s = \ell: \ell \in \bZ,  |\ell| <  2^{j}\} \times \{ t \in A_{2^{-j}}^\iota S_{-\ell}^\iota Z^2\} }\\
&\cup&
\{\iota = \emptyset \} \times \{ a = 2^{-j} : j \in \bN \} \times \{s = \ell: \ell \in \bZ,  |\ell| =  2^{j}\} \times \{ t \in A_{2^{-j}}^\iota S_{-\ell}^\iota Z^2\}
\end{eqnarray*}
To prove that the choice of $\Lambda_j$ offers clustered sparsity for the shearlet frame, we need some auxiliary results.
The following lemma gives the decay estimate of the shearlet elements.

Note that if we define $\ipa{|t|_{a,s;\iota}}:= \ipa{|S_{s}^\iota A_{a^{-1}}^\iota t|}$, then
\[
|\sigma^\iota_{a,s,t}(x)|\le c_N a^{-3/2}\ipa{|x-t|_{a,s;\iota}}^{-N}.
\]

The following lemma is needed later for estimating the decay coefficients of the shearlet aligned with the singularity.

\begin{lemma} Let the line segment with respect to $(a,s,t;v)$ be
$Seg(a,s,t;v) :=\{S_{s}^v A_{a^{-1}}^v(x-t_1,-t_2): |x|\le \rho\}$. Then
\begin{enumerate}
\item Given the line
\[
Line(a,s,t;v):=\{S_{s}^v  A_{a^{-1}}^v(x-t_1,-t_2): x\in\bR\},
\]
the closest point $P_L$ to the origin on this line satisfies
\[
d_1^2:=\|P_L\|_2^2 = \frac{a^{-4}}{1+s^2}t_2^2.
\]

\item Set $x_0 =\frac{a^{-1}s}{1+s^2}t_2+t_1$. If $P_S$ is the closest point on the segment $Seg(a,s,t;v)$ to the origin, then
\begin{eqnarray*}
d_2^2&:=& \|P_S-P_L\|_2^2\\
& = &
\left\{\begin{array}{ll}
\min_\pm a^{-2}(1+s^2)(\pm\rho-x_0)^2 &x_0 \in [-\rho,\rho]\\
0 &x_0 \notin[-\rho,\rho]
\end{array}\right. .
\end{eqnarray*}
\end{enumerate}
\end{lemma}

\begin{proof}
Let $L(x):=S_s^v A_{a^{-1}}^v(x-t_1,-t_2)$. Then
\begin{eqnarray*}
\|L(x)\|_2^2&=&\|({a^{-1}}(x-t_1),a^{-1}s(x-t_1)-a^{-2}t_2)\|_2^2
\\&=&a^{-2}(x-t_1)^2\!+a^{-2}s^2(x-t_1)^2\!+a^{-4}t_2^2-2a^{-3}s(x-t_1)t_2
\\&=&a^{-2}(1+s^2)(x-t_1)^2+a^{-4}t_2^2-2a^{-3}s(x-t_1)t_2.
\end{eqnarray*}

Solving $\frac{d}{dx}\|L(x)\|_2 = 2(x-t_1)a^{-2}(1+s^2)-2a^{-3}st_2=0$, we have $x_0 =\frac{a^{-1}s}{1+s^2}t_2+t_1$.  It follows that
\[
\norm{P_L}_2^2 = \norm{L(x_0)}_2^2=\norm{L(\frac{a^{-1}s}{1+s^2}t_2+t_1)}_2^2=
\frac{a^{-4}}{1+s^2}t_2^2=:d_1^2.
\]

Note that $P_L\in Seg(a,s,t;v)$ if and only if $x\in [-\rho,\rho]$, in which case $d_2=0$. Otherwise,
\begin{eqnarray*}
d_2^2 &=& \min_\pm\|L(\pm\rho)-P_L\|_2^2
\\&=&\min_\pm\|L(\pm\rho)-P_L\|_2^2\\
& =&
\min_\pm\|({a^{-1}}(\pm \rho-x_0),-a^{-1}s(\pm \rho-x_0))\|_2^2\\
&=&\min_\pm a^{-2}(1+s^2)(\pm\rho-x_0)^2,
\end{eqnarray*}
which completes the proof.
\end{proof} 

We need another auxiliary lemma.
Note that
\[\ip{w\cL}{\sigma_{a,s,t}^\iota}=\ip{w\cL_j}{\sigma_{a,s,t}^\iota}.\]
\begin{lemma}
Define $R_N(x_0,y_0):=\int_{y_0}^\infty\ipa{|(x_0,\alpha)|}^{-N}d\alpha$ (which may be thought of as a ray integral). Then
for $y_0\ge0$,
\[
R_N(x_0,y_0)\le \pi \ipa{|x_0|}^{-1}\ipa{|(x_0,y_0)|}^{2-N}.
\]
\end{lemma}
\begin{proof}
Choose $\beta\in (0,1)$.  Then
\[
\int_{0}^\infty|f(\alpha)|d\alpha\le (\sup_{t\in(0,\infty)}|f(\alpha)|^\beta)\int_0^\infty|f(\alpha)|^{1-\beta}d\alpha.
\]
If we set $(1-\beta)N=2$ and $f(t) = \ipa{|(x_0,y_0+\alpha)|}^{-N}$, then we obtain
\[
R_N(x_0,y_0)\le (\sup_{v\in R(x_0,y_0)}\ipa{|v|}^{2-N}) \int_0^\infty\ipa{|(x_0,y_0+\alpha|}^{-2}d\alpha.
\]
Since
\begin{eqnarray*}
\int_{-\infty}^\infty \ipa{|(x_0,y)|}^{-M}dy& =& \ipa{|x_0|}^{-M}\int_{-\infty}^\infty\left\langle \frac{y}{\ipa{|x_0|}}\right\rangle^{-M}dy\\
&=& \ipa{|x_0|}^{-M+1}\int_{-\infty}^\infty\ipa{\alpha}^{-M}d\alpha,
\end{eqnarray*}
fixing $M=2$ and recalling the classic identity $\pi = \int_{-\infty}^\infty(1+\alpha^2)^{-1}d\alpha$ yield the bound
\[
\int_0^\infty\ipa{|(x_0,y_0+\alpha)|}^{-2}d\alpha\le \pi\ipa{|x_0|}^{-1}.
\]
Furthermore, since $y_0\ge0$,
\[
\sup_{v\in R(x_0,y_0)}\ipa{|v|}^{2-N}=\ipa{|(x_0,y_0)|}^{2-N}.
\]
This completes the proof.
\end{proof} 

Now we can estimate the decay of the shearlet coefficients aligned with the line singularity $w\cL$ as follows.
\begin{lemma} \label{Lemm:decaySlope0}
Retaining the notation as above, we have
\begin{eqnarray*}
\ip{w\cL}{\sigma_{a,s,t}^v}
&\le&  c_N \frac{a^{-1/2}}{\sqrt{1+s^2}} R_N(d_1,a^{-1}\sqrt{1+s^2}d_2)\\
&\le& c_N \frac{a^{-1/2}}{\sqrt{1+s^2}} \ipa{|d_1|}^{-1} \ipa{|(d_1,a^{-1}\sqrt{1+s^2}d_2|}^{2-N}.
\end{eqnarray*}
\end{lemma}

\begin{proof}
We have
\begin{eqnarray}
|\ip{w\cL}{\sigma_{a,s,t}^v} |&=& |\int_{-\rho}^\rho w_1(x)\sigma_{a,s,t}^v(x,0)dx| \nonumber\\
&\le& \int_{-\rho}^{\rho}|\sigma_{a,s,t}^v(x,0)|dx
\nonumber\\&
\le& c_N a^{-3/2}\int_{Seg(a,s,t;v)}  \ipa{|w|}^{-N}dw,\label{eqn:fixapp}
\end{eqnarray}
where we use an affine transformation of variables to turn the anisotropic norm $|(x,0)|_{a,s,t;v}$ into the Euclidean norm $|w|$.
Application of the same transformation to $[-\rho,\rho]\times \{0\}$  yields $Seg(a,s,t;v)$. The integral  in (\ref{eqn:fixapp}) is along a curve traversing $Seg(a,s,t;v)$ at speed $\nu_1=a^{-1}\sqrt{1+s^2}$. If we let $Ray(a,s,t;v)$ denote the ray
starting from $P_S$ and initially traversing $Seg(a,s,t;v)$, then
\begin{eqnarray*}
a^{-3/2} \int_{Seg(a,s,t;v)}\ipa{|w|}^{-N}dw
&\le&
 a^{-3/2} \int_{Ray(a,s,t;v)}\ipa{|w|}^{-N}dw\\
&\le&
a^{-3/2} \nu^{-1}\int_{\nu_1 Ray(a,s,t;v)}\ipa{|w|}^{-N}dw
\\&\le&
\frac{ a^{-1/2}}{\sqrt{1+s^2}}\int_{\nu_1d_2}^\infty\ipa{|(d_1,t)|}^{-N}dw
\\&\le&
\frac{ a^{-1/2}}{\sqrt{1+s^2}} R_N(d_1,\nu_1d_2).
\end{eqnarray*}

\end{proof} 

Next, we estimate the decay of the shearlet coefficients associated with those shearlets not aligned with the line singularity.
\begin{lemma} \label{lemm:decaySlope1}
Let $t=(t_1,t_2)$. We consider the following three cases:
\begin{enumerate}
\item[{\rm(i)}] $t_1\neq0$ and $t_2\neq 0$. Then we have
\[
|\ip{w\cL}{\sigma_{a,s,t}^v}|\le c_{L,M}{|t_1|}^{-L}{|t_2|}^{-M} a^{-1/2}  e^{-c a^{-1}s}
a^{2M},
\]
when $1\le|s|< a^{-1}$
\[
|\ip{w\cL}{\sigma_{a,s,t}^h}|\le c_{L,M}{|t_1|}^{-L}{|t_2|}^{-M} a^{-1/2}  e^{-c a^{-2}}a^{M}
\]
and for $s = \pm a^{-1}$
\[
|\ip{w\cL}{\sigma_{a,s,t}}|\le c_{L,M}{|t_1|}^{-L}{|t_2|}^{-M} a^{-1/2}  e^{-c a^{-1}}a^{M}.
\]
\item[{\rm(ii)}]
If exactly one of $t_1$ or $t_2$ is $0$, then we have
\[
|\ip{w\cL}{\sigma_{a,s,t}^\iota}|\le c_{L}{|t_1^2+t_2^2|}^{-L/2} a^{-1/2}  e^{-c a^{-1}s},  \iota = h,v.
\]
\item[{\rm(iii)}] $t_1=t_2= 0$. Then we have
\[
|\ip{w\cL}{\sigma_{a,s,t}^\iota}|\le c a^{-1/2}  e^{-c a^{-1}},\iota = h,v.
\]
\end{enumerate}

\end{lemma}

\begin{proof}
First, it is easy to show that
\[
\frac{\partial^L}{\partial \xi_1^L}\frac{\partial^M}{\partial \xi_2^M}|\hat\sigma_{a,s,0}^v |\le c_{L,M} a^{3/2} a^{L} a^{2M},
\]
By definition of the line singularity $w\cL$, we have
\begin{eqnarray*}
\ip{w\cL}{\sigma_{a,s,t}^v}
 &=& \int\int \hat{w}(\xi_1)\hat{\sigma}_{a,s,t}^v(\xi_1,\xi_2)d\xi_1d\xi_2
\\&= &\int e^{-2 \pi it_2\xi_2}\left[\int \hat{w}(\xi_1)\hat{\sigma}_{a,s,0}^v(\xi_1,\xi_2)e^{-2\pi i t_1\xi_1}d\xi_1\right] d\xi_2.
\end{eqnarray*}
For $t_1\neq 0$ and $t_2\neq 0$, when we repeatedly apply integration by parts, we have
\[
|\ip{w\cL}{\sigma_{a,s,t}^v}|\le C {|t_2|}^{-M}{|t_1|}^{-L}
\|h_{L,M}\|_{L^1(\bR)},
\]
where
\[
h_{L,M}(\xi_2) = \int D^{L,M}(\hat{w}(\xi_1)\hat{\sigma}_{a,s,0}^v(\xi_1,\xi_2))d\xi_1,
\]
and for some function $f$ which is sufficiently differentiable we define the multi index,
\[
D^{L,M}f(\eta_1,\eta_2) = \left(\frac{\partial}{\partial \eta_1}\right)^L
\left(\frac{\partial}{\partial \eta_2}\right)^M f(\eta_1,\eta_2).
\]

The next step is to estimate the term $|h_{L,M}(\xi_2)|$.

Let $\Xi_{a,s}(\xi_2)$ be the support of the function
\[
\xi_1\mapsto D^{L,M} (\hat{w}(\xi_1)\hat{\sigma}^v_{a,s,0}(\xi_1,\xi_2)).
\]
Note that for fixed $a,s$, the  function $\xi_1\mapsto\hat{w}(\xi_1)\hat\sigma_{a,s,0}^v(\xi_1,\xi_2)$ is  supported inside
$[c a^{-1}|s|,\frac{1}{2}a^{-1}s)$ for a constant $c < \frac{1}{2}$.   $h_{L,M}$ can then be written as
\[
h_{L,M}(\xi_2) = \int_{\Xi_{a,s}(\xi_2)}D^{L,M}(\hat{w}(\xi_1)\hat{\sigma}^v_{a,s,0}(\xi_1,\xi_2))d\xi_1.
\]
We then rewrite the integrand as
\[D^{L,M}(\hat{w}(\xi_1)\hat{\sigma}^v_{a,s,0}(\xi_1,\xi_2))=\sum_{\ell=0}^L{L\choose \ell}\hat{w}^{(\ell)}(\xi_1) D^{L-\ell,M}(\hat{\sigma}^v_{a,s,0}(\xi_1,\xi_2))
\]
Thus $|h_{L,M}(\xi_2)|$ is bounded by
\begin{eqnarray*}
|h_{L,M}(\xi_2)|
&\le& \sum_{\ell=0}^L{L\choose \ell}\left|\int_{\Xi(a,s)(\xi_2)}\hat{w}^{(\ell)}(\xi_1)D^{L-\ell,M}(\hat{\sigma}^v_{a,s,0}(\xi_1,\xi_2))d\xi_1\right|\\
&\le&
\sum_{\ell=0}^L{L\choose \ell}
\|\hat{w}^{(\ell)}()\|_{L^1[c a^{-1}|s|,a^{-1}|s|)}
 N^{L-\ell,M}(a,s)\\
&\le& c_{L,M}  e^{-c  a^{-1}s}
\sum_{\ell=0}^L{L\choose \ell}N^{L-\ell,M}(a,s)
\\
&\le&
c_{L,M}   e^{-c a^{-1}s}  a^{3/2} a^{2M}
\end{eqnarray*}
where
\[
N^{L-\ell,M}(a,s) = \|D^{L-\ell,M}\hat\sigma_{a,s,0}^v(\xi_1,\xi_2)\|_{L^\infty(\Xi_{a,s}(\xi_2))}
\]
Consequently, we have
\begin{eqnarray*}
\|h_{L,M}\|_{L^1(\bR)}&\le& c_{L,M} a^{-2}  e^{-c a^{-1}s} a^{3/2}
a^{M}\\
&\le& c_{L,M}  a^{-1/2}  e^{-c a^{-1}s} a^{2M}.
\end{eqnarray*}
Therefore,
\[
|\ip{w\cL}{\sigma_{a,s,t}^v}|\le c_{L,M}{|t_1|}^{-L}{|t_2|}^{-M} a^{-1/2}  e^{-c a^{-1}s}a^{2M}.
\]

Using the same approach, it is not difficult to show that for $|s|< a^{-1}$,
\[
|\ip{w\cL}{\sigma_{a,s,t}^h}|\le c_{L,M}{|t_1|}^{-L}{|t_2|}^{-M} a^{-1/2}  e^{-c a^{-2}}a^{M},
\]
and for $s = \pm a^{-1}$
\[
|\ip{w\cL}{\sigma_{a,s,t}}|\le c_{L,M}{|t_1|}^{-L}{|t_2|}^{-M} a^{-1/2}  e^{-c a^{-1}}a^{M}.
\]

The proofs for other cases are similar with simple modifications of the above procedure.
\end{proof}


\bibliography{SPIEreport}{}
\bibliographystyle{amsalpha}

\end{document}